\documentclass[11pt,reqno]{amsart}

 \usepackage{amsmath,amsthm,amsfonts}
\usepackage{latexsym,mathabx,txfonts}
\usepackage{amssymb}
\usepackage{slashed}
\usepackage{tikz}

\newcommand{\R}{\mathbb{R}}

\renewcommand{\Im}{\mathop{\mathrm{Im}}}

\renewcommand{\hat}{\widehat}

\numberwithin{equation}{section}

\newtheorem{thm}{Theorem}[section]

\newtheorem{lem}[thm]{Lemma}
\newtheorem{prop}[thm]{Proposition}

\theoremstyle{remark}
\newtheorem{rem}{Remark}[section]
\newtheorem{defn}{Definition}[section]

\newcommand{\Del}[1]{}

\begin{document}

\title[Concentric Two Bubble collapse]{Concentric bubbles concentrating in finite time \\ for the energy critical Wave Maps equation}

\author{Jacek Jendrej}
\address{Institut de Math\'ematiques de Jussieu,
Sorbonne Universit\'e,
4 place Jussieu, 75005 Paris, France
}
\email{jendrej@imj-prg.fr}
\author{Joachim Krieger}
\address{EPFL SB MATH PDE,
		B\^{a}timent MA,
		Station 8,
		CH-1015 Lausanne}
	\email{joachim.krieger@epfl.ch}
	
\begin{abstract}
We show that the energy critical Wave Maps equation from $\mathbb{R}^{2+1}$ to $\mathbb{S}^2$ and restricted to the co-rotational setting with co-rotation index $k = 2$ admits finite time blow up solutions of finite energy on $(0, t_0]\times \mathbb{R}^2$, $t_0>0$, and concentrating two concentric bubble profiles at the frequency scales 
$\lambda_1(t) = e^{\alpha(t)},\,\alpha(t)\sim \big|\log t\big|^{\beta+1}$, as well as $\lambda_2(t) = t^{-1}\cdot \big|\log t\big|^{\beta}$. The parameter $\beta>\frac32$ can be chosen arbitrarily. This shows that soliton resolution scenarios with finite time blow up and $N = 2$ collapsing profiles, i. e. bubble trees, do occur for this equation. 
\end{abstract}

\maketitle

\section{Introduction}

We consider $k$-co-rotational Wave Maps from $\mathbb{R}^{2+1}\longrightarrow \mathbb{S}^2$. Using the polar angle $u$ to characterise such a map, one arrives at the following scalar wave equation 
\begin{equation}\label{eq:kcorotational}
-u_{tt} + u_{rr} + \frac{1}{r}u_r = k^2\cdot \frac{\sin (2u)}{2r^2}. 
\end{equation}
In this paper we shall restrict to the case $k = 2$, which results in the model we shall henceforth consider
\begin{equation}\label{eq:2corotational}
-u_{tt} + u_{rr} + \frac{1}{r}u_r = 2\frac{\sin (2u)}{r^2}. 
\end{equation}
This equation admits the static solutions $Q(r) = 2\arctan(r^2)$. These solutions are known to play the role of fundamental building blocks for general solutions due to the soliton-resolution theorem of \cite{JL3}, see also \cite{DKMM} for the case $k = 1$. 
\\
While the latter result gives an a priori description of general solutions of \eqref{eq:2corotational}, in particular those blowing up in finite time, the question of whether such blow up solutions with more than one collapsing bubbling profile can actually arise has been open thus far. 
We observe that the existence of such solutions for the related {\it{harmonic map heat flow}}
\[
u_t = u_{rr} + \frac{1}{r}u_r - k^2\frac{\sin (2u)}{2r^2}
\]
has been disproved for the case $k = 1$ in \cite{VanderHout}, and a similar result may be expected for the case $k\geq 2$. 
\\

Here we show that \eqref{eq:2corotational} admits finite time blow up solutions with two collapsing profiles:
\begin{thm}\label{thm:Main} Let $\beta>\frac32$, and let $Q = 2\arctan(r^2)$. Fixing the scaling parameter $\lambda_2(t) = t^{-1}\cdot\big|\log t\big|^{\beta}$, there is a higher frequency scale 
\[
\lambda_1(t) = e^{\alpha(t)},\,\alpha(t)\sim \big|\log t\big|^{\beta+1},
\]
as well as $t_0 = t_0(\beta)>0$, such that there exists a finite energy solution of \eqref{eq:2corotational} of the form 
\[
u(t, r) = Q(\lambda_1(t)r) - Q(\lambda_2(t)r) + \epsilon(t, r),
\]
satisfying the relation 
\[
\lim_{t\rightarrow 0}\big\|\nabla_{t,r}\epsilon\big\|_{L^2_{r\,dr}(r\leq t)} = 0. 
\]
\end{thm}

We observe that the solution is only barely above energy regularity. This is due to the construction which is based on the choice of an {\it{outer profile}} built in close analogy to \cite{KST3}. It is expected that a suitable choice of $\beta$ as in \cite{RaRod} results in a smooth solution. 
We also note that our techniques do not yield any stability assertion for the constructed solution, although it appears likely that they are unstable. We further note that while there is freedom in choosing the scale $\lambda_2(t)$, the inner scale $\lambda_1(t)$ is determined in terms of the outer one via a differential equation. 

\subsection{Known multi bubble type solutions} The first examples of {\it{concentric two bubble solutions}}
for \eqref{eq:kcorotational} with $k \geq 3$, as well as for the equivariant Yang-Mills problem closely related to \eqref{eq:2corotational}, were obtained in \cite{J1, J2}. These solutions are {\it{pure two bubble solutions}} which exist globally in one time direction. More precisely, they decouple asymptotically into two bubbles, one of which is static while the other one gets re-scaled dynamically
(its scale converges to $0$ in large time). In particular, the energy of these solutions needs to be exactly twice the energy of one static bubble. 
Uniqueness of this type of solutions for $k \geq 4$ was proved in \cite{JL2},
and the work \cite{JL4} gave their asymptotic description in both time directions
(in particular proving inelastic collision of the two bubbles).
\\
The case $k = 1$ was analyzed in \cite{Rodrig}, where among other results,
it was proved that pure two bubble solutions do not exist in this case.
\\

We emphasize that the solutions constructed in the present paper {\it{are not of threshold type}},
in the sense that their energy is strictly larger than twice the energy of one static bubble.
Moreover, there is a continuum of such solutions. 
We also note that it is likely that our technique can be adapted to construct finite time blow up solutions with more than two concentric and collapsing bubbles. 

\subsection{Other related results}
An extensive literature is devoted to the study of small solutions to the general (non co-rotational) wave maps equation. Influential papers in this direction include \cite{KlM, KS97, Tao01, Tataru01}. We emphasise that the first two of these papers show that is natural to consider initial data of regularity only barely above energy class for the energy critical wave maps, even outside the equivariant setting. 
\\
First fundamental results about the dynamical behavior of large solutions of \eqref{eq:kcorotational} were obtained in \cite{CTZcpam, CTZduke, STZ92, STZ94}, see also \cite[Chapter 8]{ShSt00},
the main conclusion being the decay of energy at the self-similar scale,
which in particular excludes self-similar blow-up, but also, as proved in \cite{Struwe03}, leads to \emph{bubbling}:
if $u$ is a solution of \eqref{eq:kcorotational} which blows up at time $0$, then there exist
sequences $t_n \to 0$ and $0 < \lambda_n \ll t_n$ such that
\begin{equation}
(u(t_n, \lambda_n \cdot), \lambda_n u_t(t_n, \lambda_n \cdot)) \to m\pi \pm Q,
\end{equation}
the convergence being understood in the topology induced by the energy locally (on bounded sets).

The bubbling also implies that a solution whose energy is smaller than the energy of $Q$ cannot blow up.
It was proved in \cite{CKLS15} that such a solution actually has radiative behavior.
Above this threshold energy, finite time blow-up can occur, as was proved in the works \cite{KST1, RoSt10, RaRod}.
Important progress toward the soliton resolution was obtained in \cite{CKLS2}.
\emph{Sequential} soliton resolution, that is convergence to a superposition of solitons for a sequence of times,
was proved in \cite{Cote15} for $k \in \{1, 2\}$, and in~\cite{JK} for $k \geq 3$.

In \cite{JL4}, continuous in time resolution was proved at the minimal possible energy level allowing for existence of a two-bubble. As a relatively simple consequence, continuous in time resolution
was proved in \cite{JL5} under the assumption that the solution contains at most two bubbles.

\subsection{Notation} 
An inequality
\[
F\lesssim G
\]
for nonnegative quantities $F, G$ means the existence of a universal positive constant $C$ such that 
\[
F\leq C\cdot G.
\]

We shall sometimes use the notation
\[
F\lesssim \tau^{-p+}
\]
to imply that for any $\epsilon>0$ there is a constant $C_{\epsilon}$ such that 
\[
F\leq C_{\epsilon}\cdot \tau^{-p+\epsilon}. 
\]
Such inequalities will always occur in a setting where $\tau\geq \tau_*$ for some sufficiently large $\tau_*\gg 1$. 
\\
We shall also write 
\[
\mu(t)\sim \lambda(t)
\]
for $t\in (0, t_0]$ provided there are positive constants $C_{1,2}>0$ such that 
\[
C_1\cdot\lambda_1(t)\leq \mu(t)\leq C_2\cdot \lambda_2(t)\,\forall t\in (0, t_0].
\]
For positive functions $\mu(t), \lambda(t)$ defined on $(0, t_0]$, we shall write $\mu(t)\ll\lambda(t)$, provided 
\[
\lim_{t\rightarrow 0_+}\frac{\mu(t)}{\lambda(t)} = 0. 
\]

\section{Strategy behind the construction}

In essence we implement the following steps:
\begin{itemize}
\item Construction of the 'outer profile' solution, which is a finite time blow up solution of energy regularity of the form 
\[
\tilde{Q}_2(t, r) = Q\big(\lambda_2(t)r\big) + v(t, r),\,\lambda_2(t) = t^{-1}\big|\log t\big|^{\beta},\,\beta>\frac32, 
\]
where 
\[
\lim_{t\rightarrow 0}\big\|\nabla_{t,r}v(t,\cdot)\big\|_{L^2_{r\,dr}(r\lesssim t)} = 0
\]
While the correction $v$ experiences a shock across the light cone $t = r$ for generic choice of $\beta$, it satisfies the following crucial higher regularity bounds:
\[
\big\|\nabla_{t,r}S^kv(t,\cdot)\big\|_{L^2_{r\,dr}(r\lesssim t)}\lesssim_k \big|\log t\big|^{-1},
\]
where $k\geq 0$ is arbitrary. Here $S = t\partial_t + r\partial_r$ is the {\it{scaling vector field}}. 
\item Construction of approximate solution of the form 
\begin{equation}\label{eq:uNansatz}
u_N(t, r) = Q\big(\lambda_1(t)r) - \tilde{Q}_2(t, r)  + v_N(t, r), 
\end{equation}
which satisfies the estimate 
\begin{equation}\label{eq:approxsolnerror}
 \big\|{-}u_{N, tt} + u_{N, rr} + \frac{1}{r}u_{N,r} - 2\frac{\sin (2u_N)}{r^2}\big\|_{L^2_{r\,dr}(r\lesssim t)}\lesssim \tau^{-N}
\end{equation}
provided we restrict $t\in (0, t_0]$ with $t_0 = t_0(\beta, N)>0$. We use the notation 
\[
\tau = e^{(\beta+1)^{-1}\cdot \big|\log t\big|^{\beta+1}}. 
\]
The construction of this highly accurate approximate solution involves a careful construction of the scaling function $\lambda_1(t)$ at the same time. 
\item Completion of the approximate solution $u_N$ to an exact solution 
\[
u = u_N + \epsilon
\]
by solving the nonlinear wave equation for $\epsilon$ via a suitable parametrix method. This requires the conditions $N\gg 1$, $0<t\leq t_0 = t_0(\beta, N)$. 
\end{itemize}

\subsection{Some comments on the second step}

The construction of the approximate solution is the most delicate part of our method. It consists in the construction of 
\[
v_N = \sum_{j=0}^N h_j
\]
by means of $N$ refinements, which are obtained inductively. Each of the corrections $h_j$ is obtained by solving a wave equation with a time dependent potential term concentrating at two diverging frequency scales. Then the exact equation which $v_N$ would have to satisfy in order that $u_N$ becomes a solution is  
given by 
 \begin{equation}\label{eq:exactvNequation}\begin{split}
  &{-}v_{N,tt} + v_{N, rr} + \frac{1}{r}v_{N, r} - \frac{4\cos\big(2Q_1- 2\tilde{Q}_2\big)}{r^2}v_N = \sum_{j=1}^3 E_j,\\
  &E_1 = \frac{2\cos\big(2Q_1- 2\tilde{Q}_2\big)}{r^2}\cdot \big[\sin\big(2v_N\big) - 2v_N\big],\\
  &E_2 = \frac{2}{r^2}\big[\sin\big(2Q_1- 2\tilde{Q}_2\big) - \sin\big(2Q_1\big) + \sin\big(2\tilde{Q}_2\big)\big] +Q_{1,tt},\\
  &E_3 =  \frac{2\sin\big(2Q_1- 2\tilde{Q}_2\big)}{r^2}\cdot \big[\cos\big(2v_N\big) - 1\big]. 
  \end{split}\end{equation}
  Throughout we set $Q_1 = Q(\lambda_1(t)r)$. \\
  The key difficulty in solving this problem consists in dealing with the inhomogeneous equation 
  \begin{equation}\label{eq:keylineartwofreqprofile}
  -v_{tt} + v_{rr} + \frac{1}{r}v_{r} - \frac{4\cos\big(2Q_1- 2\tilde{Q}_2\big)}{r^2}v = f, 
  \end{equation}
which contains a potential term $\frac{4\cos\big(2Q_1- 2\tilde{Q}_2\big)}{r^2}$ which is time dependent and concentrated at the frequency scales $\lambda_1(t), \lambda_2(t)$, and where it is our intention to have $\lambda_1(t)\gg \lambda_2(t)$. The strategy is to approach \eqref{eq:keylineartwofreqprofile} in a two-step process, resulting only in an approximate solution but with a more rapidly decaying source term. This is quite analogous to the procedure in \cite{KST1}, \cite{KST2}, \cite{KST3}, except that we do not use the {\it{free}} wave equation near the light cone, but the one involving the {\it{lower frequency bubble potential}}. Precisely, we make the ansatz
\[
v = v_0 + v_1,
\]
where the function $v_0$ solves the wave equation 
\begin{equation}\label{eq:outerwavevzero}
 -v_{0,tt} + v_{0,rr} + \frac{1}{r}v_{0,r} - \frac{4\cos\big(2\tilde{Q}_2\big)}{r^2}v_0 = f
 \end{equation}
 Solving the wave equation via the Duhamel formula has a smoothing effect at the frequency scale of the outer profile, namely $\lambda_2(t)$. 
 \\
 The function $v_1$ is then chosen to compensate for the error produced by $v_0$, namely 
 \[
  \frac{4\cos\big(2Q_1- 2\tilde{Q}_2\big) - 4\cos\big(2\tilde{Q}_2\big)}{r^2}v_0
  \]
  Instead of working with a wave equation for $v_1$, we neglect time derivatives at the inner scale and let $v_1$ solve 
  \begin{equation}\label{eq:innerwavevone}
   v_{1,rr} + \frac{1}{r}v_{1,r} - \frac{4\cos\big(2Q_1\big)}{r^2}v_1 =  \frac{4\cos\big(2Q_1- 2\tilde{Q}_2\big) - 4\cos\big(2\tilde{Q}_2\big)}{r^2}v_0.
   \end{equation}
   This can be accomplished via a variation of constants formula after passing to the variable $R = \lambda_1(t)\cdot r$ and leads to good bounds, {\it{provided one forces a suitable orthogonality property of the right hand side via addition of a correction term}}. In fact, these correction terms need to be added at each inductive stage for the construction of $h_j$, and in turn lead to the precise choice of the inner scaling parameter $\lambda_1(t)$. 
   The fact that one {\it{gains smallness}} in the construction of $v_1$ is due to the smoothing effect inherent in the construction of $v_0$, which leads to a good bound on 
   \[
   \frac{v_0}{r^2}.
   \]
   The ansatz $v = v_0 + v_1$ combining the solutions of \eqref{eq:outerwavevzero} and \eqref{eq:innerwavevone}
only approximately solves \eqref{eq:keylineartwofreqprofile}, and so the process is repeated. As it results in derivative losses, one implements only finitely many steps. This suffices, however, as it will enable us to achieve the estimate \eqref{eq:approxsolnerror}. 
   
   To begin with, we first construct the 'outer bubble solution' $\tilde{Q}_2$. The precise dynamics of the inner bubble will hinge on the fine structure of the outer bubble. 
   
   \section{The outer solution}

Here we provide details on the construction of $\tilde{Q}_2(t, r) = Q(\lambda_2(t)r) + v(t, r)$, $\lambda_2(t) = \frac{\big|\log t\big|^{\beta}}{t}$, in analogy to the paper \cite{KST3} devoted to the energy critical Yang Mills equation. The precise result we are aiming for is the following 
\begin{thm}\label{thm:outersolnYManalogue} Let $\lambda_2(t) = t^{-1}\cdot\big({-}\log t\big)^{\beta}$, $\beta\geq \frac32$. Then there is $t_0 = t_0(\beta)>0$ such that the equation 
\begin{equation}\label{eq:keq2corotational}
\big({-}\partial_{tt} + \partial_{rr} + \frac{1}{r}\partial_r\big)u = 2\frac{\sin(2u)}{r^2}
\end{equation}
admits a finite time blow up solution on $(0, t_0]\times \mathbb{R}_+$ of the form 
\[
\tilde{Q}_2(t, r) = Q(\lambda_2(t)r) + v(t, r),\,Q(r) = 2\arctan(r^2),
\]
where we have the bounds 
\[
\sum_{k=0}^K\|\nabla_{t,r}S^k v\|_{L^2}\lesssim_K |\log t|^{-1},\,S = t\partial_t + r\partial_r,
\]
as well as the bound $\big|v(t, r)\big|\lesssim |\log t|^{-1}$. 
\end{thm}
The proof of this result proceeds in exact analogy to the one in \cite{KST3}. We outline the intermediate key steps. To begin with, we need to construct an approximate solution $u_N$ by relying on a suitable algebraic framework. The precise statement for this is the following:
\begin{thm}\label{thm:approxoutersolnYManalogue} For every integer $N\geq 1$ there exists an approximate solution $u_N$ for \eqref{eq:keq2corotational} of the form 
\[
u_N(t, r) = Q(\lambda_2(t)r) + v_{10}(t, r) + v_N(t, r),\,\lambda_2(t) = t^{-1}\cdot\big({-}\log t\big)^{\beta},
\]
such that 
\[
e_N = \Box u_N -  2\frac{\sin(2u_N)}{r^2}
\]
satisfies for any $K\geq 0$ 
\[
\sum_{k=0}^K\big|S^ke_N\big|\lesssim_{K,N} t^{-2}|\log t|^{-N},\,r\lesssim t. 
\]
The first correction $v_{10}$ is given by
\[
v_{10}  = \frac{1}{(t\lambda_2)^2}\cdot \frac{R^4}{1+R^4},
\]
where $R = \lambda_2(t)\cdot r$, a notation we adhere to in this section. 
For the remaining correction we have the pointwise bound
\begin{align*}
\sum_{k=0}^K\big|S^kv_N(t, r)\big|\lesssim_{K,N} \frac{r^2}{t^2|\log t|},\,r\lesssim t,\,S = t\partial_t + r\partial_r.
\end{align*}
\end{thm}
The proof of this is in essence identical to the one in \cite{KST3}. 
\begin{proof}
In the sequel, we shall frequently encounter the operator 
\begin{equation}\label{eq:operatormathcalL}\begin{split}
\mathcal{L}: &= -\partial_R^2 - \frac{1}{R}\partial_R + \frac{4\cos\big(2Q(R)\big)}{R^2}\\
& =  -\partial_R^2 - \frac{1}{R}\partial_R  + \frac{4}{R^2} - \frac{32R^2}{(1+R^4)^2}.
\end{split}\end{equation}
This gets conjugated into 
\begin{equation}\label{eq:operatortildemathcalL}
\tilde{\mathcal{L}} = R^{\frac12}\mathcal{L}R^{-\frac12} = -\partial_R^2 + \frac{15}{4R^2} -  \frac{32R^2}{(1+R^4)^2}.
\end{equation}
The operator $\mathcal{L}$ arises upon linearizing the wave maps equation around the bulk part $Q(R) = 2\arctan(R^2)$. 
\\

{\bf{Spaces I}}: For the sequel, we first introduce the spaces $S^m\big(R^k(\log R)^l\big)$ consisting of analytic functions $v: [0,\infty)\longrightarrow \mathbb{R}$ satisfying the following properties:
\begin{itemize}
\item $v$ vanishes to order $m$ at $R = 0$ and $R^{-m}v$ has an even Taylor expansion there. 
\item $v$ has an absolutely convergent epxansion near $R = +\infty$ of the form
\[
v(R) = \sum_{0\leq j\leq l+2i}c_{ij}R^{k-2i}(\log R)^j. 
\]
\end{itemize}
This is essentially Definition 3.4 from \cite{KST3}. 
\\

{\bf{Step 1}}:
We start by observing that the error $e_0$ arising upon substituting the approximation $u_0: = Q(R)$, $R = \lambda_2(t)r$ into \eqref{eq:keq2corotational} satisfies  
\[
t^2 e_0 = -t^2\partial_t^2 Q(R) = -\big(1+\frac{\beta}{|\log t|}\big)^2\cdot R\Phi'(R) - \big(1 + \frac{\beta}{|\log t|} - \frac{\beta}{|\log t|^2}\big)\Phi(R),
\]
where $\Phi(R) = \frac{4R^2}{1+R^4}$. We then construct a first correction $v_1$ by setting 
\begin{equation}\label{eq:voneequation}
\big(t\lambda_2\big)^2\tilde{\mathcal{L}}\sqrt{R}v_1 = \sqrt{R} t^2 e_0. 
\end{equation}
A fundamental system for $\tilde{\mathcal{L}}$ is given by 
\begin{equation}\label{eq:phi0theta0}
\phi_0(R) := R^{\frac12}\Phi(R),\,\theta_0(R): = \frac{-1 + 8R^4\log R + R^8}{16 R^{\frac32}(1+R^4)}. 
\end{equation}
A solution of \eqref{eq:voneequation} corresponding to the leading order term $R\Phi'(R) + \Phi(R)$ in $t^2 e_0$ is given by 
\[
v_{10}  = \frac{1}{(t\lambda_2)^2}\cdot \frac{R^4}{1+R^4}, 
\]
as indeed $(t\lambda_2)^2\mathcal{L}v_{10} = R\Phi'(R) + \Phi(R)$. Setting $t^2\tilde{e}_0: = t^2 e_0  + R\Phi'(R) + \Phi(R)$, we can then set 
\[
v_1 = v_{10} + v_{11},
\]
where\footnote{The term is indeed slightly 'better' than what is indicated since there are in effect no logarithmic corrections $(\log R)^j$ for the terms of order $R^0$ in the expansion.} 
\begin{align*}
(t\lambda_2)^2v_{11} &= \frac{\theta_0(R)}{\sqrt{R}}\cdot \int_0^R\phi_0(R')\sqrt{R'}t^2\tilde{e}_0(R')\,dR' - \frac{\phi_0(R)}{\sqrt{R}}\cdot \int_0^R\theta_0(R')\sqrt{R'}t^2\tilde{e}_0(R')\,dR'\\
&\in |\log t|^{-1}IS^4(R^2) + |\log t|^{-2}IS^4(R^2). 
\end{align*}
The error $e_1$ generated by the approximation $u_1 = Q(\lambda_2(t)r) + v_1$ satisfies
\begin{align*}
t^2e_1 = t^2\partial_t^2v_1 + (t\lambda_2)^2\cdot\Big(\frac{4\cos(2Q)\cdot\big(\frac{\sin (2v_1)}{2} - v_1\big)}{R^2} + \frac{2\sin(2Q)\cdot\big(\cos(2v_1) - 1\big)}{R^2}\Big). 
\end{align*}
Here we abbreviate $Q = Q(\lambda_2(t)r)$. We expand the trigonometric functions 
\[
\sin (2v_1), \cos(2v_1)
\]
 into Taylor series up to order $N$. Then we easily infer that\footnote{These terms are again somewhat better in the sense of the preceding footnote.} 
\begin{equation}\label{eq:eone}
t^2e_1\in (t\lambda_2)^{-2}\Big(IS^4(1) + \frac{1}{|\log t|}IS^4(R^2) + \sum_{N\geq k\geq 1} \frac{1}{|\log t|^{k+1}}\cdot (t\lambda_2)^{-2k}IS^4(R^{2k})\Big) + t^2 E_1,
\end{equation}
where the remaining error term satisfies (with arbitrary $K\geq 0$)
\begin{align*}
\big||\chi_{r\lesssim t} t^2 E_1\big| + \big||\chi_{r\lesssim t} t^2 S E_1\big| +\ldots + \big||\chi_{r\lesssim t} t^2 S^K E_1\big|\lesssim_{N,K,\beta} |\log t|^{-N}. 
\end{align*}
To obtain this last estimate, we observe that $\chi_{r\lesssim t}\cdot  (t\lambda_2)^{-2k}\cdot R^{2k}\lesssim_k 1$. We can then neglect this term as it satisfies the error estimate postulated in Theorem~\ref{thm:approxoutersolnYManalogue}. 
\\

The preceding $S$-spaces need to be much refined to control the behavior of the approximate solution near the characteristic light cone $r = t$, or alternatively $a = \frac{r}{t} = 1$. This we do next.
\\

{\bf{Spaces II}}: Following Definition 3.2 of \cite{KST3}, we introduce the algebra $\mathcal{Q}_k$, $k\geq 0$ not necessarily an integer, and consisting of continuous functions 
\[
q: (0,1]\times \mathbb{R}_+\times \mathbb{R}_+\longrightarrow \mathbb{R},\,q = q(a,\alpha,\alpha_1), 
\]
with the following properties:
\begin{itemize}
\item $q$ is $C^\infty$ with respect to $a\in (0,1)$, and meromorphic and even around $a = 0$. The restriction to the diagonal 
\[
q(a,b,b)
\]
extends analytically to $a = 0$ and has an even expansion there. 
\item $q$ has the form 
\[
q(a,\alpha,\alpha_1) = \sum_{i+j\leq -k}^{j\leq 0, i\leq \frac{|j|}{2}}q_{ij}(a, \log\alpha, \log\alpha_1)\alpha^i \alpha_1^j
\]
with $q_{ij}$ polynomial in $\log\alpha, \log\alpha_1$. The sum only has finitely many terms. 
\item Near $a = 1$, $q$ can be expanded as 
\[
q(a,\alpha,\alpha_1) = q_0(a,\alpha) + (1-a^2)^{\frac12}\cdot q_1(a,\alpha,\alpha_1),
\]
where the coefficients $q_j$ are analytic in $a$ around $a = 1$ and $C^\infty$ with respect to $\alpha,\alpha_1$. 
\end{itemize}
We next introduce the space corresponding to Definition 3.3 in \cite{KST3}, which essentially corresponds to the effect of $\partial_a$ on the preceding space. Thus we let 
$\mathcal{Q}_k'$, $k\geq 0$, be the space of continuous functions 
\[
q: (0,1]\times \mathbb{R}_+\times \mathbb{R}_+\longrightarrow \mathbb{R},\,q = q(a,\alpha,\alpha_1), 
\]
with the following properties:
\begin{itemize}
\item $q$ is $C^\infty$ with respect to $a\in (0,1)$, and meromorphic and even around $a = 0$. The restriction to the diagonal 
\[
q(a,b,b)
\]
extends analytically to $a = 0$.
\item $q$ admits an expansion as in the preceding list of properties near $a = 0$. 
\item Near $a = 1$, $q$ can be expanded as 
\[
q(a,\alpha,\alpha_1) = q_0(a,\alpha) + (1-a^2)^{\frac12}\cdot q_1(a,\alpha,\alpha_1) + (1-a^2)^{-\frac12}\cdot q_2(a,\alpha,\alpha_1),
\]
where the coefficients $q_j$ are analytic in $a$ around $a = 1$. The coefficient $q_2$ has the same representation as the term $q$ in the previous item near $a = 0$, except that $k$ is replaced by $k+1$ and $j\leq -1$. 
\end{itemize}

Finally, we combine the asymptotic expansions with respect to $R$ as well as the structure with respect to $a$ in the following definition, which corresponds to Definition 3.5 in \cite{KST3}: 
\begin{defn}\label{def:finalSspaceinvolvingQ} (a) We denote by 
\[
S^m\big(R^k(\log R)^l, \mathcal{Q}_n\big)
\]
the space of analytic functions 
\[
v: [0,\infty)\times [0,1)\times \mathbb{R}_+^2\longrightarrow \mathbb{R}
\]
with the following properties:
\begin{itemize}
\item $v$ is analytic as a function of $R$, and $v: [0,\infty)\longrightarrow \mathcal{Q}_n$. 
\item $v$ vanishes of order $m$ at $R = 0$. 
\item $v$ has an absolutely convergent epxansion 
\[
v\big(R,\cdot,b, b_1\big) =  \sum_{0\leq j\leq l+2i}c_{ij}(\cdot, b, b_1)\cdot R^{k-2i}(\log R)^j
\]
 with coefficients $c_{ij}\in \mathcal{Q}_n$. 
 \item Letting $\tilde{S} = R\partial_R$, the functions 
 \[
 \tilde{S}^{l_1}\partial_b^{l_2}\partial_{b_1}^{l_3}\big(v\big(R,\cdot,b, b_1\big)\big),\,l_{1,2,3}\geq 0
 \]
 have analogous properties. 
\end{itemize}
\end{defn}

The variables $b, b_1$ will be chosen of the following specific form:
\begin{equation}\label{eq:bbone}
b = \big|\log t\big|,\,b_1 = \big|\log t\big| + \big|\log p(a)\big|. 
\end{equation}
Here $p$ is a real even polynomial satisfying the conditions 
\[
p(1) = 0,\,p'(1) = -1,\,p(a) = 1 + O(a^M),
\]
where $M = M(N)$ is sufficiently large, with $N$ as in the statement of the Theorem. With these substitutions, it is then seen that the scaled functions
\[
S^lv,\,S = t\partial_t + r\partial_r,\,l\geq 0,
\]
with $v$ as in the preceding admit similar expansions. 
\\
(b) We denote by 
\[
IS^m\big(R^k(\log R)^l, \mathcal{Q}_n\big)
\]
the space of functions $w(t, r)$ which can be written as 
\[
w(t, r) = v\big(R,a,b, b_1)
\]
with $v\in S^m\big(R^k(\log R)^l, \mathcal{Q}_n\big)$, provided $r\leq t$. 
\\

{\bf{Step 2}}: We now take advantage of the formal framework introduced above. First, we approximate the error $e_1$, see \eqref{eq:eone}, by the expression 
\begin{align*}
t^2 e_1^{00} := c_1\cdot \frac{a^2}{b} + \sum_{2\beta\geq k\geq 1}\frac{c_{2k}\cdot a^{2k}}{b^{k+1}},\,a = \frac{r}{t},\,b = \big|\log t\big|. 
\end{align*}
The first term on the right arises by replacing a function in $(\lambda_2 t)^{-2}\cdot \big|\log t\big|^{-1}\cdot IS^4(R^2)$ by its leading term, which for $e_1$ is of the form 
\[
c_1\cdot (\lambda_2 t)^{-2}\cdot \big|\log t\big|^{-1}\cdot R^2 = c_1\cdot  \big|\log t\big|^{-1}\cdot a^2
\]
for a suitable constant $c_1$, and analogously for the other terms on the right. 
With this choice, we have 
\begin{align*}
t^2e_1 - t^2E_1 - t^2e_1^{00}\in IS^2\big(1, \mathcal{Q}_{2\beta}\big),
\end{align*}
which will then be improved in the subsequent steps. 
\\

Returning to the original variables $(t, r)$, we determine $v_2$ by solving the following model equation approximately: 

\begin{equation}\label{eq:waveat1}
t^2\cdot\Big({-}\partial_t^2 + \partial_r^2 + \frac{1}{r}\partial_r - \frac{4}{r^2}\Big)v_2 = t^2e_1^{00}.
\end{equation}
As in \cite{KST3}, interpreting $v_2 = v_2(a, b)$, $a = \frac{r}{t}$, this can be rewritten in the form 
\begin{align*}
L_{ab}v_2 = t^2e_1^{00},\,L_{ab} = -\big(\partial_b + a\partial_a\big)^2 - \big(\partial_b + a\partial_a\big) + \partial_a^2 + \frac{1}{a}\partial_a - \frac{4}{a^2}. 
\end{align*}
In order to solve this, we introduce the general function space $\mathcal{F}_{k,m},\,2k\geq m\geq k\geq 1$ stemming from \cite{KST3}:
\begin{defn}\label{defn:Fkmspace} For $2k\geq m\geq k\geq 1$ integers, we define (recall \eqref{eq:bbone})
\begin{align*}
\mathcal{F}_{k,m} = \{f_k\,\big|\, f_k &= b^{-k}e_0(a, \log b) + (1-a)^{\frac12}\sum_{j=1}^m e_j^0(a, \log b, \log b_1)b^{j-k}b_1^{-j}\\
& + (1-a)^{-\frac12}\sum_{j=1}^m e_j^1(a, \log b, \log b_1)b^{j-k-1}b_1^{-j}\}
\end{align*}
Here the functions $e_0, e_j^0, e_j^1$ are smooth for $a\in (0,1)$, analytic around $a = 1$, and meromorphic and even around $a = 0$. They are furthermore polynomials in $\log b, \log b_1$. Finally, 
\[
f_k = O(a^2)
\]
for $0<a\ll 1$. 
\end{defn}

It is immediate that the function $t^2 e_1^{00}$ belongs to
\[
\sum_{1\leq k\leq 2\beta}\mathcal{F}_{k,m_k}
\]
where the $m_k$ can be chosen arbitrarily in $[k, 2k]$. In fact, the components of $t^2 e_1^{00}$ are of the much more restrictive type $b^{-k}\cdot e_0(a)$. We now require the following basic result proved in \cite{KST3}:
\begin{lem}\label{lem:abwaveeqnlemma}(\cite{KST3}) Letting $f_k\in \mathcal{F}_{k,m}$, the equation 
\[
L_{ab}v = f_k
\]
admits an approximate solution 
\[
v(a, b) = b^{-k}V_0(a, \log b) + (1-a)^{\frac12}\cdot\sum_{j=1}^{m}V_j(a,\log b, \log b_1)b^{j-k}b_1^{-j}, 
\]
where $V_0, V_j$ are smooth in $a\in (0, 1)$, analytic around $a = 1$, meromorphic around $a = 0$, and polynomial in the variables $\log b, \log b_1$. Furthermore, $v$ vanishes to fourth order at $a = 0$ and 
\[
L_{ab}v - f_k\in \mathcal{F}_{k+1,m} + \mathcal{F}_{k+2,m}. 
\]
\end{lem}
Iteratively applying Lemma 3.8 to the equation
\[
L_{ab}v_2 = t^2e_1^{00},
\]
we construct 
\[
v_2\in a^4IS\big(1, \mathcal{Q}_1\big)
\]
such that the corresponding error satisfies 
\begin{align*}
L_{ab}v_2 - t^2e_1^{00}\in a^2 IS^4\big(1, \mathcal{Q}_{2\beta}'\big) + IS^4\big(1, \mathcal{Q}_{2\beta}\big). 
\end{align*}
It remains to determine to what extent this approximation solves the original equation, which is of course only approximated by \eqref{eq:waveat1}. We encounter the following types of errors $E_j,\,j = 1,\ldots, 5$.
First, an error due to the modification of the potential term in \eqref{eq:waveat1}, which equals
\begin{align*}
E_1 = \big(\frac{4\cos(2Q)}{r^2} - \frac{4}{r^2}\big)\cdot v_2. 
\end{align*}
We determine the potential term explicitly using $R = \lambda_2(t)\cdot r$, as follows:
\begin{align*}
\big(\frac{4\cos(2Q)}{r^2} - \frac{4}{r^2}\big) =-32r^{-2}\cdot \frac{R^4}{(1+R^4)^2}. 
\end{align*}
Using the fact that $v_2\in a^4IS\big(1, \mathcal{Q}_1\big)$, we infer that 
\begin{align*}
t^2E_1&\in  t^2\cdot 32r^{-2}\cdot \frac{R^4}{(1+R^4)^2}\cdot a^4IS\big(1, \mathcal{Q}_1\big) = \frac{R^4}{(1+R^4)^2}\cdot a^2IS\big(1, \mathcal{Q}_1\big)\\
& =  \frac{R^6}{(1+R^4)^2}\cdot (t\lambda_2)^{-2}\cdot IS\big(1, \mathcal{Q}_1\big)\subset IS^4\big(1, \mathcal{Q}_{2\beta}\big).
\end{align*}
Another error accounts for the fact that we omitted the effect of $v_1$ in the linear potential term. This is the expression 
\begin{align*}
E_2 = \frac{4\cos(2Q)\cdot(\cos(2v_1) - 1)}{r^2}\cdot v_2. 
\end{align*}
Here the fact that $v_1$ only decays like $|\log t|^{-1}$ from Step 2 means that we only infer the weaker conclusion 
\[
t^2E_2\in  IS^4\big(1, \mathcal{Q}_2\big). 
\]
Similarly, for the remaining nonlinear error terms 
\begin{align*}
&E_3 =\frac{2\cos\big(2Q+2v_1\big)\cdot\big(\sin\big(2v_2\big) - 2v_2\big)}{r^2},\,E_4 = \frac{4\sin(2Q)\cdot \sin(2v_1)\cdot v_2}{r^2},\\
&E_5 = \frac{2\sin\big(2Q + 2v_1\big)\cdot [\cos(2v_2)-1]}{r^2},  
\end{align*}
we infer the inclusion 
\[
t^2E_j\in  IS^4\big(1, \mathcal{Q}_2\big),\,j = 3, 4, 5. 
\]
Thus the errors $E_j, j = 2,\ldots, 5$ are worse than $E_1$ provided $\beta>1$, and we have to re-iterate the process leading to $v_2$ to improve the remaining error to $t^{-2}IS^4\big(1, \mathcal{Q}_{2\beta}\big)$. Repeating this procedure finitely many times leads to our final $v_2$ with an error $e_2$ satisfying 
\[
v_2\in a^4IS\big(1, \mathcal{Q}_1\big),\,t^2e_2\in a^2 IS^4\big(1, \mathcal{Q}_{2\beta}'\big) + IS^4\big(1, \mathcal{Q}_{2\beta}\big).
\]
{\bf{Step 3}}: Here we perform the general inductive step. Our goal shall be to prove the following inclusions:
\begin{align}
&v_{2k-1}\in IS^4\big(R^2\big(\log R\big)^{3k-5},\,\mathcal{Q}_{2\beta k}\big),\,k\geq 2\\
&t^2 e_{2k-1}\in IS^2\big(R^2\big(\log R\big)^{3k-5},\,\mathcal{Q}_{2\beta k}'\big),\,k\geq 2\\
&v_{2k}\in a^4IS\big(\big(\log R\big)^{3(k-1)},\,\mathcal{Q}_{2\beta(k-1)}\big),\,k\geq 1,\\
&t^2e_{2k}\in  a^2IS\big(\big(\log R\big)^{3(k-1)},\,\mathcal{Q}_{2\beta k}'\big) +  IS^4\big(\big(\log R\big)^{3(k-1)},\,\mathcal{Q}_{2\beta k}\big),\,k\geq 1. 
\end{align}
We note right away that the starting point with the last two inclusions for $k = 1$ is satisfied for $v_2, e_2$ constructed above. 
\\

We distinguish between the construction of even and odd indices. Start with a given $e_{2k}$, $k\geq 1$. We only need to deal with the part $e_{2k}^0$ in 
\[
IS^4\big((\log R)^{3(k-1)}, \mathcal{Q}_{2\beta k}\big),
\]
as the remaining part can be included into $e_{2k+1}$ at the next stage of the iteration. 
We construct $v_{2k+1}$ by applying the procedure used for $v_1$, i. e. applying the simple variation of constants formula to $e_{2k}^0$, treating the variables $a, b, b_1$ as constants. Thus we set 
\begin{align*}
(t\lambda_2)^2v_{2k+1}(R,a,b,b_1) &=  \frac{\theta_0(R)}{\sqrt{R}}\cdot \int_0^R\phi_0(R')\sqrt{R'}t^2e_{2k}^0(R',a,b,b_1)\,dR'\\& - \frac{\phi_0(R)}{\sqrt{R}}\cdot \int_0^R\theta_0(R')\sqrt{R'}t^2e_{2k}^0(R',a,b,b_1)\,dR'.
\end{align*}
From this we conclude directly that 
\[
v_{2k+1}\in IS^4\big(R^2(\log R)^{3k-2},\,\mathcal{Q}_{2\beta(k+1)}\big).
\]
We generate errors due to having treated $a, b, b_1$ as constant parameters and to having neglected the time derivative term in the wave operator, as well as errors due to the nonlinear interactions, and having used $\mathcal{L}$ for the linearisation around the approximate solution $u_{2k} = Q + \sum_{j=1}^{2k}v_j$. We deal with the latter two types of errors, as the former can be handled exactly as in \cite{KST3}. 
\\

In fact, we can argue exactly as for the error terms $E_j$ in step 2 but with $v_1$ replaced by $ \sum_{j=1}^{2k}v_j$ and $v_2$ replaced by $v_{2k+1}$, since 
\[
\sum_{j=1}^{2k}v_j\in \sum_{j=1}^{2k}IS^4\big(R^2(\log R)^{l_j}, \mathcal{Q}_{2\beta \lfloor \frac{j+1}{2}\rfloor}\big),\,l_{2k} = 3(k-1),\,l_{2k-1} = (3k-5)_+. 
\]
Furthermore we observe that (for $j\geq 1$)
\[
IS^4\big(R^2(\log R)^{l_j}, \mathcal{Q}_{2\beta \lfloor \frac{j+1}{2}\rfloor}\big)\subset a^2IS^4\big((\log R)^{l_j}, \mathcal{Q}_{2\beta (\lfloor \frac{j+1}{2}\rfloor-1)}\big)
\]

We then have for $j\geq 1$
\begin{align*}
&a^2IS^4\big((\log R)^{l_j}, \mathcal{Q}_{2\beta (\lfloor \frac{j+1}{2}\rfloor-1)}\big)\cdot  IS^4\big(R^2(\log R)^{3k-2},\,\mathcal{Q}_{2\beta(k+1)}\big)\\
&\subset  a^2IS^4\big(R^2(\log R)^{3(k+\lfloor \frac{j+1}{2}\rfloor)-5},\,\mathcal{Q}_{2\beta(k+\lfloor \frac{j+1}{2}\rfloor)}\big)
\end{align*}
Expanding the trigonometric functions into power series (only up to errors of size $|\log t|^{-N}$ is required), this yields 
\begin{align*}
t^2 E_2 &=  t^2\cdot \frac{4\cos(2Q)\cdot\big(\cos\big(2 \sum_{j=1}^{2k}v_j\big) - 1\big)}{r^2}\cdot v_{2k+1}\\
&\in \sum_{1\leq j\leq N}IS^4\big(R^2(\log R)^{3(k+\lfloor \frac{j+1}{2}\rfloor)-5},\,\mathcal{Q}_{2\beta(k+\lfloor \frac{j+1}{2}\rfloor)}\big) + O\big(|\log t\big|^{-N}\big), 
\end{align*}
where the error term satisfies the same estimate as the term $E_1$ in step 1. Each of the term in the finite sum can be attributed to a source term
\[
t^2e_{2(k+\lfloor \frac{j+1}{2}\rfloor)-1},\,j\geq 1,
\]
which gets reduced at a later stage of the induction. 
\\
Next consider the analogue of the error terms $E_3$ in step 2. Here we use that 
\begin{align*}
v_{2k+1}\in a^2\cdot IS^4\big((\log R)^{3k-2},\,\mathcal{Q}_{2\beta k}\big)
\end{align*}
and further (with $k, k'\geq 1$)
\begin{align*}
&a^2\cdot IS^4\big((\log R)^{3k-2},\,\mathcal{Q}_{2\beta k}\big)\cdot IS^4\big(R^2(\log R)^{3k'-2},\,\mathcal{Q}_{2\beta(k'+1)}\big)\\
&\subset a^2\cdot IS^4\big(R^2(\log R)^{3(k'+k)-2},\,\mathcal{Q}_{2\beta(k'+k+1)}\big)
\end{align*}
Using Taylor expansion for $\sin(2 v_2)$, expanding 
\[
\cos\big(2Q+2\sum_{j=1}^{2k}v_j\big) = \cos(2Q)\cdot \cos\big(2\sum_{j=1}^{2k}v_j\big) -  \sin(2Q)\cdot \sin\big(2\sum_{j=1}^{2k}v_j\big),
\]
using Taylor expansion for the factors involving $\sum_{j=1}^{2k}v_j$ and  applying the preceding inclusion inductively, we infer that
\begin{align*}
&t^2\frac{2\cos\big(2Q+2\sum_{j=1}^{2k}v_j\big)\cdot\big(\sin\big(2v_{2k+1}\big) - 2v_{2k+1}\big)}{r^2}\\
&\in \sum_{1\leq k'\leq N}IS^4\big(R^2(\log R)^{3(k'+k)-2},\,\mathcal{Q}_{2\beta(k'+k+1)}\big) + O\big(|\log t\big|^{-N}\big),
\end{align*}
with the remaining error term analogous to the one for $t^2E_2$. The analogues of the remaining terms $E_{4,5}$ from step 2 are handled similarly. 
\\

Next, we start with an error term of type $e_{2k-1}$, $k\geq 2$, and construct the corresponding correction $v_{2k}$. Here, proceeding in exact analogy to \cite{KST3}, we select the leading order term $e_{2k-1}^0$ in $e_{2k-1}$ which satisfies 
\[
t^2e_{2k-1}^0 = R^2\sum_{j=0}^{3k-5}g_j(a,b,b_1)\big(\log R\big)^j,\,g_j(a,\cdot,\cdot)\in \mathcal{Q}'_{2\beta k}. 
\]
The remaining error $e_{2k-1}^1 = e_{2k-1} - e_{2k-1}^0$ then satisfies 
\begin{align*}
t^2e_{2k-1}^1 &= a^2\cdot t^2e_{2k-1}^1 + (1-a^2)\cdot t^2e_{2k-1}^1\\&\in a^2IS\big(\big(\log R\big)^{3(k-1)},\,\mathcal{Q}_{2\beta k}'\big) +  IS^4\big(\big(\log R\big)^{3(k-1)},\,\mathcal{Q}_{2\beta k}\big),
\end{align*}
hence can be incorporated into $t^2 e_{2k}$. 
Re-write the preceding expression for $t^2e_{2k-1}^0$ as 
\[
\sum_{j=0}^{3k-5}h_j(a,b,b_1)\big(\log R\big)^j,\,h_j = a^2\cdot (\lambda_2 t)^2\cdot g_j\in a^2 \mathcal{Q}'_{2\beta (k-1)}
\]
The correction $v_{2k}$ will then be the result of an iterative process. To begin with, we start with a first correction $w_{2k}$ for which we make the ansatz 
\[
w_{2k} = \sum_{j=0}^{3k-5}z_j(a,b,b_1)\big(\log R\big)^j,\,z_j\in a^4\mathcal{Q}_{2\beta(k-1)}
\]
Then we determine the coefficient functions $z_j$ inductively by formulating equations analogously to \eqref{eq:waveat1}. Proceeding as in \cite{KST3} indeed leads to 
\begin{equation}\label{eq:zjsystem}\begin{split}
&L_{ab}z_j = h_j + (j+1)\big(1+2(a-a^{-1})\partial_a\big)z_{j+1} - (j+1)(j+2)(1+a^{-2})z_{j+2},\\&z_{3k-4} = z_{3k-3} = 0,\,0\leq j\leq 3k-5. 
\end{split}\end{equation}
As in \cite{KST3}, section 3, this system can be solved inductively by applying Lemma~\ref{lem:abwaveeqnlemma}. We thereby obtain coefficient functions 
\[
z_j\in a^4\cdot\mathcal{Q}_{2\beta(k-1)}. 
\]

Again in analogy to step 2 we now consider the error terms generated by this procedure, which are like the terms $E_j, j = 1,\ldots 5$ there, but with $v_1$ replaced by 
\[
\sum_{j=1}^{2k-1}v_j\in \sum_{j=1}^{2k-1}IS^4\big(R^2(\log R)^{l_j}, \mathcal{Q}_{2\beta \lfloor \frac{j+1}{2}\rfloor}\big),\,l_{2k} = 3(k-1),\,l_{2k-1} = (3k-5)_+. 
\]
We need to show that these errors can be placed into $e_{2k}$. 
To begin with, we have (arguing similarly to the beginning of step 2)
\begin{align*}
t^2E_1 &= t^2\big(\frac{4\cos(2Q)}{r^2} - \frac{4}{r^2}\big)\cdot w_{2k}\\
&\in a^{-2}\cdot IS^4(R^{-4})\cdot a^4\sum_{j=1}^{3k-5}IS^4\big((\log R)^{j}, \mathcal{Q}_{2\beta(k-1)}\big)\\
&\in IS^4\big((\log R)^{3k-5}, \mathcal{Q}_{2\beta k}\big)
\end{align*}
For the term $E_2$, we use that 
\[
\sum_{j=1}^{2k-1}v_j\in  \sum_{j=1}^{2k-1}a^2IS^2\big((\log R)^{l_j}, \mathcal{Q}_{2\beta (\lfloor \frac{j+1}{2}\rfloor-1)}\big)
\]
Furthermore\footnote{Recall that for our construction of $w_{2k}$ we only required $3k - 5$ powers of $\log R$.}, we have that with $j\geq 3, k\geq 1$, 
\begin{align*}
&IS^2\big((\log R)^{l_j}, \mathcal{Q}_{2\beta (\lfloor \frac{j+1}{2}\rfloor-1)}\big)\cdot a^4IS\big(\big(\log R\big)^{3k- 5},\,\mathcal{Q}_{2\beta(k-1)}\big)\\
&\subset a^2\cdot \sum_{k'\geq 2}IS^4\big(\big(\log R\big)^{3(k+[k'-1]-1)},\,\mathcal{Q}_{2\beta(k+k'-1)}\big)
\end{align*}
Further, from steps 1 and 2 we recall that $v_1\in a^2 \cdot IS^4\big(1, \mathcal{Q}_1\big), v_2\in a^4IS\big(1, \mathcal{Q}_1\big)$, and we have 
\begin{align*}
&a^2 \cdot IS^4\big(1, \mathcal{Q}_1\big)\cdot a^4IS\big(\big(\log R\big)^{3k- 5},\,\mathcal{Q}_{2\beta(k-1)}\big)\\
&\subset a^6IS\big(\big(\log R\big)^{3k- 5},\,\mathcal{Q}_{2\beta(k-1)+1}\big)
\end{align*}
Using Taylor expansion and the preceding observations, we infer that 
\begin{align*}
t^2 E_2 &=  t^2\cdot \frac{4\cos(2Q)\cdot\big(\cos\big(2 \sum_{j=1}^{2k}v_j\big) - 1\big)}{r^2}\cdot w_{2k}\\
&\in  \sum_{k'\geq 2}IS^4\big(\big(\log R\big)^{3(k+[k'-1]-1)},\,\mathcal{Q}_{2\beta(k+k'-1)}\big)\\
& + IS^4\big(\big(\log R\big)^{3k- 5},\,\mathcal{Q}_{2\beta(k-1)+1}\big)
\end{align*}
If a term is in the first sum of spaces it can be included in $\sum_{l\geq 0}t^2e_{2k+l}$, but terms in the last space on the right are not good enough and we have to re-iterate application of Lemma~\ref{lem:abwaveeqnlemma} to its principal part, just as before. Sufficiently many iterations of this procedure lead to the desired correction $v_{2k}$ with an error term of the form $\sum_{l\geq 0}t^2e_{2k+l}$.
\\

The remaining error terms $t^2E_j, j = 3, 4, 5$, which are defined as in step 2, except $v_2$ is replaced by $w_{2k}$ and $v_1$ by $\sum_{j=1}^{2k-1}v_j$, are handled similarly. 
Consider for example $E_4$, which satisfies 
\[
t^2 E_4 = t^2\frac{2\sin(2Q)\cdot \sin\big(\sum_{j=1}^{2k-1}v_j\big)\cdot v_{2k}}{r^2}.
\]
Here we can use that 
\begin{align*}
\sin(2Q)\in IS^2\big(R^{-2}\big),
\end{align*}
and further in light of the preceding comments on the structure of $v_1, v_2, v_{k},\,k\geq 3$, we have 
\begin{align*}
 t^2\frac{2\sin(2Q)\cdot \sin\big(\sum_{j=1}^{2k-1}v_j\big)}{r^2}&\in a^{-2}\cdot  IS^2\big(R^{-2}\big)\cdot  a^2 \cdot IS^4\big(1, \mathcal{Q}_1\big)\\
 &\subset a^{-2}\cdot IS^4\big(1, \mathcal{Q}_{2\beta + 1}\big)
\end{align*}
We conclude that with 
\[
v_{2k}\in a^4IS\big(\big(\log R\big)^{3(k-1)},\,\mathcal{Q}_{2\beta(k-1)}\big),\,k\geq 1,
\]
we have 
\begin{align*}
t^2\frac{2\sin(2Q)\cdot \sin\big(\sum_{j=1}^{2k-1}v_j\big)\cdot v_{2k}}{r^2}\in a^2IS^4\big(\big(\log R\big)^{3(k-1)},\,\mathcal{Q}_{2\beta k + 1}\big),
\end{align*}
which is better than required for terms of the type $t^2 e_{2k},\,k\geq 1$. 
\\

The proof of Theorem~\ref{thm:approxoutersolnYManalogue} is now completed by iterating the preceding step 3 sufficiently many times such that the remaining error is a sum of terms of the type $e_j$ with $j$ sufficiently large. 
The fact that the differentiated bounds (applying powers of $S$) apply to both $v_N, e_N$ follow from the final property of the spaces in Definition~\ref{def:finalSspaceinvolvingQ}.

\end{proof}

Completion of the approximate solution to the exact solution asserted in Theorem~\ref{thm:outersolnYManalogue} requires a number of facts from spectral theory. Recall the fundamental system \eqref{eq:phi0theta0}.
The following proposition is the exact analogue of Proposition 4.5 in \cite{KST3}. Recall the operator $\tilde{\mathcal{L}}$ from \eqref{eq:operatortildemathcalL}. 
\begin{prop}\label{prop:Fourierbasis} 
The operator $\tilde{\mathcal{L}}$ restricted to $(0,\infty)$ is self-adjoint with domain
\[
\text{Dom}\big(\tilde{\mathcal{L}}\big) = \{f\in L^2\big((0,\infty)\big)\,\,|\,\,f,\,f'\in \text{AC}\big((0,\infty)\big),\,\tilde{\mathcal{L}}f\in L^2\big((0,\infty)\big)\}. 
\]
The spectrum of $\tilde{\mathcal{L}}$ equals $[0,\infty)$, and is purely absolutely continuous. \\
There is a fundamental system $\phi(R,z), \theta(R,z)$ for $\tilde{\mathcal{L}} - z$ and arbitrary $z\in \mathbb{C}$, which depends analytically on $(z, R)$ for $R>0$, and has the asymptotic behavior 
\[
\phi(R;z)\sim R^{\frac52},\,\theta(R,z)\sim \frac14R^{-\frac32}
\]
as $R\rightarrow 0$. 

For any $z\in \mathbb{C}$, the function $\phi(R, z)$ admits an absolutely convergent asymptotic expansion 
\[
\phi(R, z) = \phi_0(R) + R^{-\frac32}\cdot\sum_{j=1}^\infty \big(R^2z\big)^j\tilde{\phi}_j(R^2).
\]
The functions $\tilde{\phi}_j$ are holomorphic on the region $\Omega: = \{\big|\Im u\big|<\frac12\}$, and satisfy the point wise bounds 
\begin{align*}
\big|\tilde{\phi}_j(u)\big|\leq \frac{C^j}{j!}\cdot \frac{|u|^{2}}{\langle u\rangle},\,j\geq 1.
\end{align*}

\end{prop}
\begin{proof} It follows closely the one in \cite{KST3}. Write 
\[
\phi(R,z) = \sum_{j=0}^\infty z^j\cdot \phi_j(R),
\]
where the coefficient functions $\phi_j$ solve the iterative system 
\[
\tilde{\mathcal{L}}\phi_j = \phi_{j-1},\,j\geq 1. 
\]
Write $\phi_j = R^{-\frac32}f_j$, $j\geq 0$, whence in particular $f_0 = \frac{R^4}{1+R^4}$. Using the fundamental system \eqref{eq:phi0theta0} to construct the Green's function, we infer the iterative relations 
\begin{align*}
&f_j(R) = \int_0^R R^{\frac32}\big(R'\big)^{-\frac32}\cdot \big(\phi_0(R)\theta_0(R') - \phi_0(R')\theta_0(R)\big)f_{j-1}(R')\,dR',\\
&f_0(R) = \frac{R^4}{1+R^4}.
\end{align*}
The inductive formula can be rendered explicit as 
\begin{align*}
f_j(R) &= \int_0^R\Big[R^4\big({-}1 + 8(R')^4\log R' + (R')^8\big)\\&\hspace{1.5cm} - (R')^4\big({-}1+8R^4\log R + R^8\big)\Big)\cdot \frac{f_{j-1}(R')R'}{(R')^4\big(1+R^4\big)(1+(R')^4\big)}\,dR'. 
\end{align*}
By the exact same argument as in \cite{KST3} one infers by induction that $f_j$ extends to a holomorphic function on any simply connected region which does not contain any fourth root of $-1$. Writing 
$f_j(R) = \tilde{f}_j(R^2)$, whence in particular $\tilde{f}_0(u) = \frac{u^2}{1+u^2}$, we obtain the inductive relation 
\begin{align*}
\tilde{f}_j(u) &= \int_0^u \Big(u^2\big({-}1 + 8v^2\log v + v^4\big)\\&\hspace{2cm} - v^2\big({-}1 + 8u^2\log u + u^4\big)\Big)\frac{\tilde{f}_{j-1}(v)}{2v^2(1+u^2)(1+v^2)}\,dv
\end{align*}
Restricting to $\big|\Im u\big|<\frac12$ in order to avoid a singularity, we then infer the bounds 
\begin{align*}
\big|\tilde{f}_j(u)\big|\leq \frac{C^j}{j!}\cdot \frac{|u|^{2+j}}{\langle u\rangle}, 
\end{align*}
for a suitable constant $C$, where $\langle u\rangle: = \sqrt{1+|u|^2}$, via induction. Finally using that 
\[
\tilde{\phi}_j(u) = u^{-j}\cdot \tilde{f}_j(u), 
\]
the bounds asserted in the proposition follow. 
\end{proof}

The preceding proposition is useful in the non-oscillatory region $R^2\xi\lesssim 1$, while in the oscillatory regime $R^2\xi\gg 1$, a different asymptotic expansion is required, as in the following Proposition, whose proof is identical to the one given in \cite{KST3}: 
\begin{prop}\label{prop:WeylTitchmarsh} For any $\xi>0$, the equation $(\tilde{\mathcal{L}} - \xi)f = 0$ admits a solution  $\psi^+(\cdot, \xi) = \psi^{+}(\cdot, \xi + i 0)$ arising as limit of the Weyl-Tichmarsh solutions on the upper half plane, and of the form 
\[
\psi^+(R,\xi) = \xi^{-\frac14}e^{iR\xi^{\frac12}}\sigma(R\xi^{\frac12}, R),\,R^2\xi\gtrsim 1, 
\]
where $\sigma = \sigma(q, R)$ admits an asymptotic series expansion of the form 
\[
\sum_{j=0}^\infty q^{-j}\psi^+_j(R),\,\psi_0^+ = 1,\,\psi_1^+ = \frac{15i}{8} + O\big(\langle R\rangle^{-2}\big),
\]
with zero order symbols $\psi_j^+(R)$ that are analytic at infinity, in the sense that 
\begin{align*}
\sup_{R>0}\Big|\big(R\partial_R\big)^{\alpha}\big(q\partial_q\big)^{\beta}\big[\sigma(q, R) - \sum_{j=0}^{j_0}q^{-j}\psi_j^+(R)\big]\Big|\leq c_{\alpha,\beta,j_0}q^{-j_0 - 1}. 
\end{align*}
Writing 
\begin{align*}
\phi(R;\xi) = a(\xi)\psi^+(R,\xi) + \overline{a(\xi)\psi^+(R,\xi)},
\end{align*}
the function $a(\cdot)$ is smooth on $(0,\infty)$,  always non zero, and we have 
\begin{align*}
&\big|a(\xi)\big|\sim 1,\,0<\xi\lesssim 1,\\
&\big|a(\xi)\big|\sim \xi^{-1},\,\xi\gtrsim 1.
\end{align*}
The absolutely continuous part of the spectral measure has density 
\[
\rho(\xi) = \frac{1}{\pi}\big|a(\xi)\big|^{-2},
\]
and hence obeys the asymptotic relations 
\begin{align*}
&\big|\rho(\xi)\big|\sim 1,\,0<\xi\lesssim 1,\\
&\big|\rho(\xi)\big|\sim \xi^{2},\,\xi\gtrsim 1.
\end{align*}
\end{prop}

The proof of this is exactly like the one in \cite{KST3}. We note that $\psi_1^+$ has the same leading order asymptotic as in the Yang-Mills case since the leading order part of the operator $\tilde{L}$ agrees with the one of the operator occurring in \cite{KST3}.  
\\

Using the distorted Fourier representation 
\begin{align*}
\tilde{\epsilon}(\tau, R) = x_0(\tau)\phi_0(R) + \int_0^\infty x(\tau,\xi)\phi(R;\xi)\rho(\xi)\,d\xi, 
\end{align*}
we need to express the operator $R\partial_R$ in terms of the distorted Fourier coefficients. The fact that there is both continuous as well as discrete spectrum requires the introduction of a matrix valued operator $\mathcal{K}$: write 
\begin{align*}
&\widehat{(R\partial_R)u} = -2\xi\partial_{\xi}\hat{u} + \mathcal{K}\hat{u},\,\hat{u} = \left(\begin{array}{c}x_0\\ x\end{array}\right),\\
&\mathcal{K} = \left(\begin{array}{cc}\mathcal{K}_{ee}& \mathcal{K}_{ec}\\\mathcal{K}_{ce}&\mathcal{K}_{cc}\end{array}\right),
\end{align*}
where the entries are either operators or functions (which act by multiplication), as in the following list, keeping in mind \eqref{eq:phi0theta0}:
\begin{align*}
&\mathcal{K}_{ee} = \langle R\partial_R\phi_0(R),\,\phi_0(R)\rangle_{L^2_{dR}}\\
& \mathcal{K}_{ec}f = \langle\int_0^\infty f(\xi)R\partial_R\phi(R;\xi)\rho(\xi)\,d\xi,\,\phi_0(R)\rangle_{L^2_{dR}}\\
& \mathcal{K}_{ce}(\eta) = \langle  R\partial_R\phi_0(R),\,\phi(R;\eta)\rangle_{L^2_{dR}}\\
&\mathcal{K}_{cc}f(\eta) =  \langle\int_0^\infty f(\xi)R\partial_R\phi(R;\xi)\rho(\xi)\,d\xi,\,\phi(R;\eta)\rangle_{L^2_{dR}}\\
&\hspace{1.3cm}+  \langle\int_0^\infty 2\xi\partial_{\xi}f(\xi)R\partial_R\phi(R;\xi)\rho(\xi)\,d\xi,\,\phi(R;\eta)\rangle_{L^2_{dR}}\\
\end{align*}
Using integration by parts, we can directly compute 
\begin{align*}
 \langle R\partial_R\phi_0(R),\,\phi_0(R)\rangle_{L^2_{dR}} = -\frac12\big\|\phi_0\big\|_{L^2_{dR}}^2 = -\frac{3\pi}{32}. 
\end{align*}
On the other hand, the operator $ \mathcal{K}_{ec}: L^2_{\rho\,d\xi}\longrightarrow \mathbb{C}$ is given by integration against the kernel $K_e(\xi) = \rho(\xi)\cdot K_e(\xi)$, where 
\[
K_e(\xi) = -\langle R\partial_R\phi_0(R),\,\phi(R;\xi)\rangle_{L^2_{dR}}. 
\]
Using the commutation relation 
\begin{align*}
\big[\tilde{\mathcal{L}}, R\partial_R\big] = 2\tilde{\mathcal{L}} - \big(R\partial_R\big)\Big(\frac{32 R^2}{(1+R^4)^2}\Big), 
\end{align*}
as well as repeated integration by parts, one infers that $K_e$ is rapidly decaying towards $\xi = +\infty$. Of course we also have 
\[
 \mathcal{K}_{ce} = -K_e.
 \]
 The operator $\mathcal{K}_{cc}$, which is the most delicate, can be analyzed in close analogy to \cite{KST3}:
 \begin{prop}\label{prop:Kcctransference} We can write 
 \begin{align*}
 \mathcal{K}_{cc} = -\big(\frac32 + \frac{\eta\rho'}{\rho}\big)\delta(\xi - \eta) + \mathcal{K}_0,
 \end{align*}
 where the operator $\mathcal{K}_0$ is given by integration against the kernel
 \[
 K_0(\eta, \xi) = \frac{\rho(\xi)}{\eta - \xi}F(\xi, \eta),
 \]
 where the function $F$ is of regularity $C^2$ on $(0, \infty)\times (0,\infty)$, continuous on the closure of this set, and satisfies the bounds stated in \cite{KST3}. 
  \end{prop}
  
  The proof of this follows exactly the steps in [KST], Theorem 5.1, except that one uses the relation 
  \[
  \big[\tilde{\mathcal{L}}, R\partial_R\big] = 2\tilde{\mathcal{L}} - \frac{64 R^2}{(1+R^4)^2} + \frac{256 R^6}{(1+R^4)^3}. 
  \]
  The last two terms here have in fact a better decay rate than the corresponding terms in \cite{KST3}. 
\\

Following loc.cit., we then deduce the following important mapping bounds:
\begin{prop}\label{prop:transferencemappingbounds} The operators 
\[
\mathcal{K}_0, \mathcal{K}
\]
have the following mapping properties 
\[
\mathcal{K}_0: L^{2,\alpha}_{\rho}\longrightarrow L^{2,\alpha+\frac12}_{\rho},\,\mathcal{K}: L^{2,\alpha}_{\rho}\longrightarrow L^{2,\alpha}_{\rho}.
\]
for arbitrary $\alpha\in \mathbb{R}$. We also have the commutator bound 
\[
\big[\mathcal{K}, \xi\partial_{\xi}\big]: L^{2,\alpha}_{\rho}\longrightarrow L^{2,\alpha}_{\rho},
\]
where $\xi\partial_\xi$ only acts non-trivially on the continuous spectral part. 
\end{prop}

Finally we need to implement the fixed point argument to construct the final perturbation. Write the equation for to $\epsilon$ as follows:
\begin{equation}\label{eq:epsfinalequation}\begin{split}
\big({-}\partial_{tt} + \partial_{rr} +\frac1r \partial_r - \frac{4\cos\big(2Q(R)\big)}{r^2} + \frac{8\sin(2Q(R))}{r^2}v_{10}\big)\epsilon = \mathcal{N}(\epsilon) + \chi_{r\lesssim t}e_N, 
\end{split}\end{equation}
where we write the source term as (here $v_{app} = v_{10} + v_N$, $Q = Q(R) = Q(\lambda_2(t)r)$)
\begin{align*}
&\mathcal{N}(\epsilon) = \sum_{j=1}^3\mathcal{N}_j(\epsilon),\\
&\mathcal{N}_1(\epsilon) = \frac{2\sin[2(Q+v_{app})]\cdot \big(\cos(2\epsilon) - 1\big)}{r^2},\\
& \mathcal{N}_2(\epsilon) = \frac{2\cos[2(Q+v_{app})]\cdot \big(\sin(2\epsilon) - 2\epsilon\big)}{r^2},\\
& \mathcal{N}_3(\epsilon) = \frac{4\cos\big(2(Q+v_{app})\big)+8\sin(2Q)v_{10}-4\cos\big(2Q(R)\big)}{r^2}\epsilon. 
\end{align*}
We note that the time dependent term 
\[
 \frac{8\sin(2Q(R))}{r^2}v_{10}\epsilon,
\]
when multiplied by $\lambda_2^{-2}$, admits the explicit form 
\begin{align*}
(\lambda_2 t)^{-2}\cdot \frac{32(1-R^2)}{(1+R^4)^2}\cdot \epsilon. 
\end{align*}
This term plays a role analogous to the one of 
\[
12\omega^2\frac{R^2(1-R^2)}{(1+R^2)^3}\epsilon
\]
in \cite{KST3}. 
\\

In the following, we take advantage of the new time variable 
\begin{equation}\label{eq:tautimeinlambda2setting}
\tau: = \int_{t}^{t_0} \lambda_2(s)\,ds\sim \big|\log t\big|^{\beta+1}, 
\end{equation}
as well as the parameter 
\begin{equation}\label{eq:omegainlambda2setting}
\omega = \frac{\lambda_{2,\tau}}{\lambda_2}. 
\end{equation}
We shall restrict $t_0 = t_0(\beta)>0$ to be sufficiently small at the end. 
We pass to the variable $\tilde{\epsilon} = R^{\frac12}\cdot \epsilon$, where we think of $\epsilon$ as a function of $\tau, R$, to derive the equation 
\begin{equation}\label{eq:tildeepsagain}
\Big[{-}\big(\partial_\tau + \frac{\lambda_{\tau}}{\lambda}R\partial_R\big)^2 + \frac12\dot{\omega} + \frac14 \omega^2 - \tilde{\mathcal{L}}+(\lambda_2 t)^{-2}\cdot \frac{32(1-R^2)}{(1+R^4)^2}\Big]\tilde{\epsilon} = \lambda_2^{-2}R^{\frac12}\mathcal{N}\big(\epsilon\big)
\end{equation}
In order to complete a setup which allows us to replicate the method from \cite{KST3}, we need to introduce 
\begin{align*}
\mathcal{F}\big((\lambda_2 t)^{-2}\cdot \frac{32(1-R^2)}{(1+R^4)^2}\tilde{\epsilon}\big) =:\mathcal{J}\big(\mathcal{F}\tilde{\epsilon}\big),\,\mathcal{J} = \left(\begin{array}{cc}\mathcal{J}_{ee}& \mathcal{J}_{ec}\\
 \mathcal{J}_{ce}&  \mathcal{J}_{cc}\end{array}\right).
\end{align*}
In analogy to the operator $\mathcal{K}_{nd}$ before, one verifies that $\mathcal{J}_{ee} = \gamma_*$ for some $\gamma_*\in \mathbb{R}$, 
\[
 \mathcal{J}_{ec}x = \int_0^\infty \rho(\xi)J_e(\xi)x(\xi)\,d\xi,
 \]
 where the function $J_e(\xi)$ is smooth and decays rapidly for large $\xi$, while $ \mathcal{J}_{ce}$ is given by multiplication with $J_e$. Finally, in analogy to $\mathcal{K}_0$, we have the integral representation 
 \begin{align*}
  \mathcal{J}_{cc} = \int_0^\infty \rho(\xi)J_{cc}(\xi,\eta)x(\eta)\,d\eta, 
 \end{align*}
 where the function $J_{cc}(\xi,\eta)$ is of regularity $C^2$ and enjoys the off-diagonal decay 
 \[
 \big|J_{cc}(\xi,\eta)\big|\lesssim (1+\xi)^{-\frac12}\cdot \big(1+|\xi^{\frac12} - \eta^{\frac12}|\big)^{-N}. 
 \]
 We have now all the ingredients that are required, precisely following the method detailed in \cite{KST3}, to prove the required existence result for a solution of \eqref{eq:epsfinalequation}. For this we need to specify the space in which the solution $\epsilon$ lives, which we do via the following norm: 
\begin{equation}\label{eq:epsnorm}
\big\|\epsilon\big\|_{H^1_N}: = \sup_{0<t<t_0}\big|\log t\big|^{N-\beta-1}\big(\big\|L_t^{\frac12}\epsilon\big\|_{L^2_{r\,dr}} + \big\|\partial_t\epsilon\big\|_{L^2_{r\,dr}} + \lambda_2(t)\big|\log t\big|^{-\beta}\big\|\epsilon\big\|_{L^2_{r\,dr}}\big). 
\end{equation}
We also have the $L^2$-based norm 
\begin{equation}\label{eq:L2typenorm}
\big\|f\big\|_{L^2_N}: = \sup_{0<t<t_0}\lambda_2^{-1}(t)\cdot \big|\log t\big|^N\cdot \big\|f(t,\cdot)\big\|_{L^2_{r\,dr}}. 
\end{equation}

We can then formulate the following main result 
\begin{thm}\label{thm:epssolutioncompletion} The equation \eqref{eq:epsfinalequation} admits a solution which is $C^\infty$ in the interior of the light cone centered at $(0,0)$, and which satisfies the bound 
\begin{align*}
\big\|\epsilon\big\|_{X}: = \big\|\epsilon\big\|_{H^1_{N_0}} + \big\|S\epsilon\big\|_{H^1_{N_1}} + \big\|S^2\epsilon\big\|_{H^1_{N_2}}\ll 1,
\end{align*}
where $S = t\partial_t + r\partial_r$, $N_0\gg N_1\gg N_2$, and the parameter $N$ from Theorem~\ref{thm:approxoutersolnYManalogue} is chosen sufficiently large, $N\gg N_0$. 
\end{thm}

The proof if this theorem follows exactly as in \cite{KST3} via the following two ingredients:
\begin{itemize}
\item Solution of the linear problem: here we generalize \eqref{eq:tildeepsagain} to the inhomogeneous linear wave equation with suitable source term, and establish bounds for the solution. 
\item Nonlinear estimates: here we show that the specific source term in \eqref{eq:tildeepsagain} satisfies the required bounds. 
\end{itemize}

For the first item, we first reformulate the general inhomogeneous problem as in \eqref{eq:tildeepsagain} with the source $ \lambda^{-2}R^{\frac12}\mathcal{N}\big(\epsilon\big)$ replaced by $\tilde{f}$ as follows: letting $\omega = \frac{\lambda_{\tau}}{\lambda}$, 
\begin{equation}\label{eq:generalinhomlin}
\Big[{-}\big(\partial_\tau + \omega R\partial_R\big)^2 - \omega\big(\partial_\tau + \omega R\partial_R\big) +  \frac12\dot{\omega} + \frac14 \omega^2 - \tilde{\mathcal{L}}\Big]\tilde{\epsilon} = \tilde{f} - (\lambda_2 t)^{-2}\cdot \frac{32(1-R^2)}{(1+R^4)^2}\tilde{\epsilon}. 
\end{equation}
Then we have the following result: 
\begin{prop}\label{prop:linearinhombounds} Recalling that $\tilde{\epsilon} = R^{\frac12}\epsilon$, $\tilde{f} = \lambda_2^{-2}R^{\frac12}f$, there is a solution of \eqref{eq:generalinhomlin} satisfying the bound 
\begin{align*}
\big\|\epsilon\big\|_{H^1_{N}}\lesssim N^{-1}\cdot \big\|f\big\|_{L^2_{N}}. 
\end{align*}
We also have the estimates 
\begin{align*}
&\big\|S\epsilon\big\|_{H^1_{N}}\lesssim N^{-1}\cdot \big(\big\|Sf\big\|_{L_N^2} + \big\|f\big\|_{L^2_N}\big),\\
&\big\|S^2\epsilon\big\|_{H^1_{N}}\lesssim N^{-1}\cdot \big(\big\|S^2f\big\|_{L_N^2} +  \big\|Sf\big\|_{L_N^2} + \big\|f\big\|_{L^2_N}\big)
\end{align*}
The implied constant depends only on $\beta$. 
\end{prop}
The proof of this is essentially identical to the one given in \cite{KST3}, thanks to the preceding comments on $\mathcal{K}, \mathcal{J}$. In the sequel, we shall sometimes denote the solution described in the preceding proposition by 
\[
\epsilon = \Phi(f).
\]

We next turn to the second item, for which the list of terms after \eqref{eq:epsfinalequation} need to be bounded in the norm 
\begin{align*}
\big\|f\big\|_{Y}: = \big\|f\big\|_{L^2_{N_0}} +  \big\|Sf\big\|_{L^2_{N_0}} +  \big\|S^2f\big\|_{L^2_{N_0}}
\end{align*}

We state the 
\begin{lem}\label{lem:sourcetermnonlinbounds} We have the estimates
\[
\big\|\mathcal{N}_j(\epsilon)\big\|_{Y}\lesssim \big\|\epsilon\big\|_{X},\,j = 1, 2, 3. 
\]
\end{lem}
\begin{proof} We use the bounds 
\begin{align*}
&\big\|S^k\epsilon\big\|_{L^\infty}\lesssim \big|\log t\big|^{1+2\beta - N_k}\big\|\epsilon\big\|_{X},\,k = 0, 1, 2,\\
&\big\|r^{-2}\epsilon\big\|_{L^2_{r\,dr}}\lesssim \lambda_2(t)\big|\log t\big|^{-N_2+1}\big\|\epsilon\big\|_{X} + \lambda_2(t)\cdot \big\|f\big\|_{L^2_{r\,dr}}. 
\end{align*}
The first one can be rendered more precise when interpreting $\epsilon = \Phi f$ where $\Phi$ is the linear operator assigning the solution $\epsilon$ to $f$ as described in Proposition~\ref{prop:linearinhombounds}: 
\begin{align*}
\big\|S^k\epsilon\big\|_{L^\infty}\lesssim N_k^{-1} \big|\log t\big|^{1+2\beta - N_k}\big\|f\big\|_{Y},\,k = 0, 1, 2. 
\end{align*}
Using interpolation between the preceding bounds, we can also infer the following estimate:
\begin{align*}
\lambda_2^{-\frac12}\cdot \big\|r^{-1}S\epsilon\big\|_{L^4_{r\,dr}}&\lesssim \big|\log t\big|^{-N_2 + 1 + \beta}\cdot \big(\big\|\epsilon\big\|_{X} + \big\|f\big\|_{Y}\big)^{\frac12}\cdot \big\|\epsilon\big\|_{X}^{\frac12}.\\
&\lesssim N_2^{-\frac12}\cdot \big|\log t\big|^{-N_2 + 1 + \beta}\cdot \big\|f\big\|_{Y},
\end{align*}
the last bound holding provided $\epsilon = \Phi(f)$.
\\

We now estimate each of the terms $\mathcal{N}_j(\epsilon)$: 
\\

{\it{The estimate for $\mathcal{N}_1(\epsilon)$.}}  We infer from Theorem~\ref{thm:approxoutersolnYManalogue} that we have the bound 
\begin{align*}
\sum_{k=0}^2\big|S^k\Big(\frac{2\sin[2(Q+v_{app})}{r^2}\Big)\big|\lesssim \lambda_2^2. 
\end{align*}
Taking advantage of the Leibniz rule for differentiation of products, we infer the estimate 
\begin{align*}
\big\|\mathcal{N}_1(\epsilon)\big\|_{Y}\lesssim \sup_{0<t<t_0}\lambda_2(t)\cdot \sum_{k=0}^2\big|\log t\big|^{N_k}\cdot\sum_{k_1+k_2=k}\big\|S^{k_1}\epsilon\cdot S^{k_2}\epsilon\big\|_{L^2_{r\,dr}}.
\end{align*}
Then we take advantage of the bounds stated at the beginning of this proof, as well as \eqref{eq:epsnorm}, to infer the estimate 
\begin{align*}
\big\|\mathcal{N}_1(\epsilon)\big\|_{Y}&\lesssim \sup_{0<t<t_0} \sum_{k=0}^2\sum_{k_1+k_2=k}\big|\log t\big|^{N_k}\cdot\big\|S^{k_1}\epsilon\big\|_{L^\infty}\cdot \lambda_2(t)\big\|S^{k_2}\epsilon\big\|_{L^2_{r\,dr}}\\
&\lesssim \sup_{0<t<t_0} \sum_{k=0}^2\sum_{k_1+k_2=k}\big|\log t\big|^{2+4\beta + N_k - N_{k_1} - N_{k_2}}\cdot \big\|\epsilon\big\|_{X}^2.
\end{align*}
Since our assumptions on the $N_j$ imply 
\begin{align*}
2+4\beta + N_k - N_{k_1} - N_{k_2}\ll -1
\end{align*}
provided $k = k_1+k_2$, $0\leq k, k_{1,2}\leq 2$, we infer the better than required bound 
\[
\big\|\mathcal{N}_1(\epsilon)\big\|_{Y}\ll\big\|\epsilon\big\|_{X},
\]
provided $\big\|\epsilon\big\|_{X}\ll 1$ 
\\

{\it{The estimate for $\mathcal{N}_2(\epsilon)$.}} Again taking advantage of Theorem~\ref{thm:approxoutersolnYManalogue} and the Leibniz rule, we estimate this contribution by 
\begin{align*}
\big\|\mathcal{N}_2(\epsilon)\big\|_{Y}\lesssim  \sup_{0<t<t_0}\lambda_2^{-1}(t)\cdot \sum_{k=0}^2\sum_{k_1+k_2 = k}\big|\log t\big|^{N_k}\prod_{j=1,2}\big\|S^{k_j}\epsilon\big\|_{L^\infty}\cdot \big\|r^{-2}\epsilon\big\|_{L^2_{r\,dr}}.
\end{align*}
Again taking advantage of the bounds stated at the beginning of this proof, we can bound the preceding by 
\begin{align*}
 \sup_{0<t<t_0}\sum_{k=0}^2\sum_{k_1+k_2 = k}\big|\log t\big|^{3 + 4\beta +N_k - N_{k_1} - N_{k_2}-N_2}\cdot \big\|\epsilon\big\|_{X}^2\cdot \big(\big\|\epsilon\big\|_{X}+\big\|f\big\|_{Y}\big) \ll  \big\|\epsilon\big\|_{X} + \big\|f\big\|_{Y}.
\end{align*}

{\it{The estimate for $\mathcal{N}_3(\epsilon)$.}} Again in light of Theorem~\ref{thm:approxoutersolnYManalogue} we infer the estimate 
\begin{align*}
\Big|S^k\Big(\frac{4\cos\big(2(Q+v_{app})\big)+8\sin(2Q)v_{10} - 4\cos\big(2Q\big)}{r^2}\Big)\Big|\lesssim \frac{1}{t^2\cdot \big|\log t\big|}. 
\end{align*}
We conclude that 
\begin{align*}
\big\|\mathcal{N}_3(\epsilon)\big\|_{Y}\lesssim  \sup_{0<t<t_0}\lambda_2^{-1}(t)\cdot\sum_{0\leq k\leq 2}\big|\log t\big|^{N_k}\cdot \frac{1}{t^2\cdot \big|\log t\big|}\cdot \big\|S^k\epsilon\big\|_{L^2_{r\,dr}}
\end{align*}
Since we have 
\begin{align*}
\lambda_2^{-1}(t)\cdot  \frac{1}{t^2\cdot \big|\log t\big|}\lesssim \lambda_2(t)\cdot \big|\log t\big|^{-2\beta - 1}, 
\end{align*}
recalling the definition of the norm $\big\|\cdot\big\|_{X}$, we deduce the bound 
\begin{align*}
\big\|\mathcal{N}_3(\epsilon)\big\|_{Y}\lesssim \big\|\epsilon\big\|_{X}. 
\end{align*}
Here we do not gain any smallness, but this gets rectified by using the norm $\big\|f\big\|_{Y}$ instead on the right hand side, provided $\epsilon = \Phi f$: 
\begin{align*}
\big\|\mathcal{N}_3(\epsilon)\big\|_{Y}\lesssim N_2^{-1}\big\|f\big\|_{Y}. 
\end{align*}
\end{proof}

Combining the preceding lemma with Proposition~\ref{prop:linearinhombounds}, we deduce the 
\begin{prop}\label{prop:contraction} The map 
\[
f\longrightarrow \mathcal{N}\big(\Phi f\big)
\]
is a contraction in $Y$ with Lipschitz constant $\sim N_2^{-1}$. In particular, it has a fixed point provided $N_2$ is sufficiently large (in relation to $\beta$). 
\end{prop}

The proof of Theorem~\ref{thm:epssolutioncompletion} is completed by setting 
\[
\epsilon = \Phi f. 
\]
In order to derive Theorem~\ref{thm:outersolnYManalogue}, we also requite control over $S^k\epsilon$ for $k\leq K$. This follows by applying $S$ sufficiently many times to the equation for $\epsilon$ and using the same estimates as before, together with Theorem~\ref{thm:approxoutersolnYManalogue}.

\section{Technical preparations for construction of an approximate two bubble solution: bounds for the linearisation around the outer profile}

The preceding section depended importantly on bounds for the solution of the linear problem \eqref{eq:generalinhomlin}. In the next section, we shall have to rely on a priori bounds for the propagator of the linearisation around the full outer profile 
\begin{equation}\label{eq;outerprofile}
\tilde{Q}_2: = Q(\lambda_2(t)r) + v(t, r),
\end{equation}
which is as described in Theorem~\ref{thm:outersolnYManalogue}. More specifically, we shall need bounds for the inhomogeneous linear equation with suitably decaying source terms. Consider the equation
\begin{equation}\label{eq:linearisationaroundtildeQtwo}
-h_{tt} +  h_{rr} + \frac{1}{r}h_r - 4\frac{\cos\big(2\tilde{Q}_2\big)}{r^2}\cdot h = F. 
\end{equation}
We next introduce the following time variable, which re-defines the meaning of $\tau$ for the rest of this section
\begin{equation}\label{eq:newtaudefi}
\tau := e^{\frac{c}{(\beta+1)}\cdot\big|\log t\big|^{\beta+1}},\,c>0. 
\end{equation}
The old time variable $\int_t^{t_0}\lambda_2(s)\,ds$ which we used in the previous section shall be labelled $\tau_2$ below. 
The reason for introducing $\tau$ shall become clear when we introduce the leading order of the inner frequency scale, in terms of a suitable scaling parameter $\lambda_1(t)$. 
Then we have 
\begin{lem}\label{lem:outerpropagatorinhombounds}
Assume the function $F$ satisfies the bounds 
\begin{align*}
\lambda_2^{-1}\cdot \Big\|S^lF(t,\cdot)\Big\|_{L^2_{r\,dr}}\lesssim \tau^{-p},\,p>0,\,0\leq l\leq M\geq 2. 
\end{align*}
Then there is $t_0 = t_0(p,M,\beta)$, such that \eqref{eq:linearisationaroundtildeQtwo} admits a solution on $(0, t_0]$ satisfying the bounds 
\begin{align*}
\big\|S^lh(t,\cdot)\big\|_{H^1_{r\,dr}}\lesssim_{l,p}\tau^{-p+},0\leq l\leq M,
\end{align*}
where we define the energy norm at the end by 
\[
\big\|h\big\|_{H^1_{r\,dr}}: = \big\|L_t^{\frac12}h\big\|_{L^2_{r\,dr}} + \big\|h_t\big\|_{L^2_{r\,dr}} + \lambda_2(t)\cdot\big\|h\big\|_{L^2_{r\,dr}},
\]
with $L_t = \partial_{rr} + \frac{1}{r}\partial_r - \frac{4\cos\big(2Q_2\big)}{r^2}$. 
Furthermore, we have the estimates 
\begin{align*}
\big\|\frac{S^lh}{r^2}\big\|_{L^2_{r\,dr}} + \big\|\frac{S^l h}{r}\big\|_{L^4_{r\,dr}}\lesssim_{l,\beta,p}\tau^{-p+},\,0\leq l\leq M-2. 
\end{align*}
\end{lem}
\begin{proof} We rewrite \eqref{eq:linearisationaroundtildeQtwo} in the form 
\begin{equation}\label{eq:heqnreformulated}
-h_{tt} +  h_{rr} + \frac{1}{r}h_r - 4\frac{\cos\big(2Q_2\big)}{r^2}\cdot h = F +  4\frac{\cos\big(2\tilde{Q}_2\big) - \cos\big(2Q_2\big)}{r^2}\cdot h. 
\end{equation}
We shall first establish the desired bounds for the case $l = 0$ and then pass to the derivatives inductively. Passing to the variable $R = \lambda_2(t)\cdot r$, and letting 
\[
\tilde{h} = R^{\frac12}\cdot h,
\]
we describe $\tilde{h}$ in terms of its distorted Fourier components
\begin{equation}\label{eq:tildehFourier}
\mathcal{F}\tilde{h} = \left(\begin{array}{c}\langle \tilde{h}, \phi_0(R)\rangle_{L^2_{dR}}\\ \langle \tilde{h}, \phi(R, \xi)\rangle_{L^2_{dR}}\end{array}\right). 
\end{equation}
We interpret $\mathcal{F}\tilde{h}$ as a vector valued function of $\big(\tau_2, \xi\big)$, where we set 
\[
\tau_2: = \int_t^1\lambda_2(s)\,ds = \frac{|\log t|^{\beta+1}}{\beta+1}. 
\]
In particular, we have that 
\[
\tau^{-p} = e^{-cp\cdot \tau_2}.
\]
for suitable $c>0$. 
Defining 
\begin{equation}\label{eq:omegalambdatwo}
\omega = \frac{\lambda_{2,\tau_2}}{\lambda_2} \sim \tau_2^{-\frac{\beta}{\beta+1}},\,\dot{\omega} = \partial_{\tau_2}\omega, 
\end{equation}
and further setting 
\[
D_{\tau_2} = \partial_{\tau_2} - \omega(1+\mathcal{K}_d), 
\]
where we use the notation (recall the discussion after Proposition~\ref{prop:WeylTitchmarsh})
\begin{equation}\label{eq:KdandKnd}
\mathcal{K}_d = \left(\begin{array}{cc}0 & 0 \\ 0& 2\xi\partial_{\xi} + 1 + \frac{\xi\rho'(\xi)}{\rho(\xi)}\end{array}\right),\,\mathcal{K}_{nd} = -\left(\begin{array}{cc}0 & \mathcal{K}_{ec}\\ \mathcal{K}_{ce}& \mathcal{K}_0\end{array}\right),
\end{equation}
we infer the following system of equations:
\begin{equation}\label{eq:distFouriereqnKST0}\begin{split}
\Big[{-}D_{\tau_2}^2 - \omega D_{\tau_2} - \xi\Big]\mathcal{F}\tilde{h} &= \mathcal{F}\tilde{G} - 2\omega\mathcal{K}_{nd}D_{\tau_2}\mathcal{F}\tilde{h} + \omega^2\big[\mathcal{K}_{nd}, \mathcal{K}_d\big]\mathcal{F}\tilde{h} \\
& + \omega^2\cdot\big(\mathcal{K}_{nd}^2 - \mathcal{K}_{nd}\big)\mathcal{F}\tilde{h} - \dot{\omega}\mathcal{K}_{nd}\mathcal{F}\tilde{h}.
\end{split}\end{equation} 
The term $\tilde{G}$ is defined as 
\[
\tilde{G} = R^{\frac12}\cdot \big(\lambda_2^{-2}F +  4\frac{\cos\big(2\tilde{Q}_2\big) - \cos\big(2Q_2\big)}{R^2}\cdot h\big). 
\]
We shall solve \eqref{eq:distFouriereqnKST0} by means of a suitable fixed point argument, which will then furnish the desired bounds of the lemma with $l = 0$. For this first introduce the norm (here we denote $\underline{x} = \left(\begin{array}{c}x_1\\ x_0\end{array}\right)$)
\begin{equation}\label{eq:normforxcomponents0}\begin{split}
&\big\|\underline{x}\big\|_{\tilde{X}_0}:\\& =  \sup_{0<t\leq t_0}e^{cq\cdot\tau_2}\cdot\big\|\rho^{\frac12}\cdot x_1(\tau_2,\cdot)\big\|_{L^2_{d\xi}} +   \sup_{0<t\leq t_0} \tau_2^{\frac{\beta}{\beta+1}}e^{cq\cdot\tau_2}\cdot\big\|\rho^{\frac12}\cdot D_{\tau_2}x_1(\tau_2,\cdot)\big\|_{L^2_{d\xi}}\\
& +  \sup_{0<t\leq t_0} \tau_2^{\frac{\beta}{\beta+1}}e^{cq\cdot\tau_2}\cdot\big\|\rho^{\frac12}\cdot \xi^{\frac12}x_1(\tau_2,\cdot)\big\|_{L^2_{d\xi}} + \sup_{0<t\leq t_0}e^{cq\cdot \tau_2}\cdot\big(\big|x_0(\tau_2)\big| +  \tau_2^{\frac{\beta}{\beta+1}}\big|\dot{x}_{0}(\tau_2)\big|\big).
\end{split}\end{equation}
Here $q<p$ can be chosen arbitrarily close to $p$. 
From Plancherel's theorem for the distorted Fourier transform, we have that (with $\underline{x} = \mathcal{F}\big(\tilde{h}\big)$)
\begin{align*}
 \sup_{0<t\leq t_0}e^{cq\cdot\tau_2}\cdot\big\|h(t,\cdot)\big\|_{H^1_{r\,dr}}\lesssim \big\|\underline{x}\big\|_{\tilde{X}_0}.
\end{align*}
For the source terms we use the simpler norm 
\begin{equation}\label{eq:sourceforfnorm0}
\big\|\underline{f}\big\|_{\tilde{Y}_0}: =  \sup_{0<t\leq t_0}\tau_2^{\frac{\beta}{\beta+1}}\cdot e^{cq\cdot\tau_2}\big\|\rho^{\frac12}\cdot f_1(\tau_2,\cdot)\big\|_{L^2_{d\xi}} +  \sup_{0<t\leq t_0}\tau_2^{\frac{\beta}{\beta+1}}\cdot e^{cq\cdot\tau_2}\big|f_0(\tau_2)\big|. 
\end{equation}
In order to solve \eqref{eq:distFouriereqnKST0}, let us write the right hand side as $\left(\begin{array}{c}g_1\\ g_0\end{array}\right)$, and setting $\mathcal{F}\tilde{h}  = \left(\begin{array}{c}x_1\\ x_0\end{array}\right)$, we infer the system  
\begin{equation}\label{eq:xzeroonedecoupled0}\begin{split}
&{-}\partial_{\tau_2}\big(\partial_{\tau_2} - \omega\big)x_0 = g_0\\
&\big({-}D_{\tau_2}^2 - \omega D_{\tau_2} - \xi\big)x_1 = g_1. 
\end{split}\end{equation}
In order to solve the first equation on the right, we can use a direct analogue of \eqref{eq:xzeropropagator} where we replace $\tau_1$ by $\tau_2$ and $\lambda_1$ by $\lambda_2$. Using the bounds 
\begin{align*}
\lambda_2(\tau_2)\cdot\int_{\tau_2}^{\sigma_2}\lambda_2^{-1}(s)\,ds\lesssim \tau_2^{\frac{\beta}{\beta+1}},\,\Big|\partial_{\tau_2}\Big(\lambda_2(\tau_2)\cdot\int_{\tau_2}^{\sigma_2}\lambda_2^{-1}(s)\,ds\Big)\Big|\lesssim 1,\,\sigma_2\geq \tau_2, 
\end{align*}
we find that the first equation of \eqref{eq:xzeroonedecoupled0} admits a solution which satisfies the bound 
\begin{equation}\label{eq:xzeroboundintautwo}
 \sup_{0<t\leq t_0}e^{cq\cdot \tau_2}\cdot\big(\big|x_0(\tau_2)\big| + \tau_2^{\frac{\beta}{\beta+1}}\big|\dot{x}_{0}(\tau_2)\big|\big)\lesssim_q \sup_{0<t\leq t_0}\tau_2^{\frac{\beta}{\beta+1}}\cdot e^{cq\cdot\tau_2}\big|g_0(\tau_2)\big|.
\end{equation}
For the second equation in \eqref{eq:xzeroonedecoupled0} we take advantage of the propagator \eqref{eq:Duhamellambdaonetauone} with $\lambda_1, \tau_1$ replaced $\lambda_2, \tau_2$. Calling this $\tilde{U}\big(\tau_2,\sigma_2,\xi\big)$, we have the estimates 
\begin{equation}\label{eq:tildeUbounds}\begin{split}
&\big|\tilde{U}\big(\tau_2,\sigma_2,\xi\big)\big|\lesssim \frac{\rho^{\frac12}\big(\xi\frac{\lambda_2^2(\tau_2)}{\lambda_2^2(\sigma_2)}}{\rho^{\frac12}(\xi)}\cdot \frac{\lambda_2(\tau_2)}{\lambda_2(\sigma_2)}\cdot \min\{\xi^{-\frac12}, \tau_2^{\frac{\beta}{\beta+1}}\},\\
&\big|\partial_{\tau_2}\tilde{U}\big(\tau_2,\sigma_2,\xi\big)\big|\lesssim  \frac{\rho^{\frac12}\big(\xi\frac{\lambda_2^2(\tau_2)}{\lambda_2^2(\sigma_2)}}{\rho^{\frac12}(\xi)}\cdot \frac{\lambda_2(\tau_2)}{\lambda_2(\sigma_2)}.
\end{split}\end{equation}
Then the solution $x_1$ which vanishes at $\tau_2 = +\infty$ is given by the Duhamel propagator 
\begin{equation}\label{eq:tildeUpropagator}
x_1(\tau_2, \xi) =  \int_{\tau_2}^{\infty}\tilde{U}(\tau_2, \sigma_2, \xi)\cdot g_1\big(\sigma_2, \frac{\lambda_2^2(\tau_2)}{\lambda_2^2(\sigma_2)}\xi\big)\,d\sigma_2. 
\end{equation}
In light of the bounds \eqref{eq:tildeUbounds}, we deduce the following propagator bound:
\begin{equation}\label{eq:tildeUpropagatorbounds}\begin{split}
&\sup_{0<t\leq t_0}e^{cq\cdot\tau_2}\cdot\big\|\rho^{\frac12}\cdot x_1(\tau_2,\cdot)\big\|_{L^2_{d\xi}} + \sup_{0<t\leq t_0}\tau_2^{\frac{\beta}{\beta+1}}\cdot e^{cq\cdot\tau_2}\cdot\big\|\rho^{\frac12}\cdot D_{\tau_2}x_1(\tau_2,\cdot)\big\|_{L^2_{d\xi}}\\
& +  \sup_{0<t\leq t_0}\tau_2^{\frac{\beta}{\beta+1}}\cdot e^{cq\cdot\tau_2}\cdot\big\|\rho^{\frac12}\cdot \xi^{\frac12}x_1(\tau_2,\cdot)\big\|_{L^2_{d\xi}}\lesssim   \sup_{0<t\leq t_0}\tau_2^{\frac{\beta}{\beta+1}}\cdot e^{cq\cdot\tau_2}\big\|\rho^{\frac12}\cdot g_1(\tau_2,\cdot)\big\|_{L^2_{d\xi}}. 
\end{split}\end{equation}
We can rephrase \eqref{eq:xzeroboundintautwo} and \eqref{eq:tildeUpropagatorbounds} more succinctly as
\begin{equation}\label{eq:tautwoinhomeqnbounds}
\big\|\underline{x}\big\|_{\tilde{X}_0}\lesssim\big\|\underline{g}\big\|_{\tilde{Y}_0}. 
\end{equation}
We shall now solve the fixed point problem \eqref{eq:distFouriereqnKST0} by relying on these bounds. For this we need to estimate the terms on the right hand side and gain smallness for those depending on $\tilde{h}$. We do this for the various terms appearing there:
\\

{\it{(1): Estimating the term $\mathcal{F}\tilde{G}$.}} We observe that 
\begin{align*}
\big\|\mathcal{F}\tilde{G}\big\|_{\tilde{Y}_0}&\lesssim  \sup_{0<t\leq t_0}\tau_2^{\frac{\beta}{\beta+1}}\cdot e^{cq\cdot\tau_2}\cdot \lambda_2^{-1}\big\|F(t,\cdot)\big\|_{L^2_{r\,dr}}\\&\hspace{4cm} +  \sup_{0<t\leq t_0}\tau_2^{\frac{\beta}{\beta+1}}\cdot e^{cq\cdot\tau_2}\cdot\big\|G_1(t,\cdot)\big\|_{L^2_{R\,dR}}.
\end{align*}
Here we let $G_1 = 4\frac{\cos\big(2\tilde{Q}_2\big) - \cos\big(2Q_2\big)}{R^2}\cdot \tilde{h}$. For the first term on the right we have the estimate 
\[
\sup_{0<t\leq t_0}\tau_2^{\frac{\beta}{\beta+1}}\cdot e^{cq\cdot\tau_2}\cdot \lambda_2^{-1}\big\|F(t,\cdot)\big\|_{L^2_{r\,dr}}\lesssim_{t_0} \sup_{0<t\leq t_0}e^{cp\cdot\tau_2}\cdot \lambda_2^{-1}\big\|F(t,\cdot)\big\|_{L^2_{r\,dr}}.
\]
For the second term we use the estimate
\begin{align*}
\Big|4\frac{\cos\big(2\tilde{Q}_2\big) - \cos\big(2Q_2\big)}{R^2}\Big|\lesssim \big(\lambda_2 t\big)^{-2}\sim \tau_2^{-\frac{2\beta}{\beta+1}},
\end{align*}
which follows from Theorem~\ref{thm:approxoutersolnYManalogue}. We conclude that 
\begin{align*}
\sup_{0<t\leq t_0}\tau_2^{\frac{\beta}{\beta+1}}\cdot e^{cq\cdot\tau_2}\cdot\big\|G_1(t,\cdot)\big\|_{L^2_{R\,dR}}\ll_{t_0}\sup_{0<t\leq t_0}e^{cq\cdot\tau_2}\cdot\big\|\tilde{h}(t,\cdot)\big\|_{L^2_{R\,dR}}\lesssim \big\|\mathcal{F}\tilde{h}\big\|_{\tilde{X}_0}. 
\end{align*}
{\it{(2): Estimates for the terms involving the transference operators $\mathcal{K}_{*}$.}} Here we rely on Proposition~\ref{prop:Kcctransference}, Proposition~\ref{prop:transferencemappingbounds}. Also recalling \eqref{eq:omegalambdatwo}, we infer that 
\begin{equation}\label{eq:transferenceboundslambdatwo1}\begin{split}
\big\| 2\omega\mathcal{K}_{nd}D_{\tau_2}\mathcal{F}\tilde{h} \big\|_{\tilde{Y}_0}&\lesssim \sup_{0<t\leq t_0}\tau_2^{-\frac{\beta}{\beta+1}}\cdot e^{cq\cdot\tau_2}\cdot\tau_2^{\frac{\beta}{\beta+1}}\cdot \big\|D_{\tau_2}\mathcal{F}\tilde{h}(t,\cdot)\big\|_{L^2_{\rho\,d\xi}}\ll_{t_0}\big\|\mathcal{F}\tilde{h}\big\|_{\tilde{X}_0}. 
\end{split}\end{equation}
Similarly we deduce the bound
\begin{equation}\label{eq:transferenceboundslambdatwo2}\begin{split}
\big\|\omega^2\big[\mathcal{K}_{nd}, \mathcal{K}_d\big]\mathcal{F}\tilde{h}\big\|_{\tilde{Y}_0} + \big\| \omega^2\cdot\big(\mathcal{K}_{nd}^2 - \mathcal{K}_{nd}\big)\mathcal{F}\tilde{h}\big\|_{\tilde{Y}_0} + \big\|\dot{\omega}\mathcal{K}_{nd}\mathcal{F}\tilde{h}\big\|_{\tilde{Y}_0}\ll_{t_0}\big\|\mathcal{F}\tilde{h}\big\|_{\tilde{X}_0}. 
\end{split}\end{equation}
Taking advantage of the propagator bound \eqref{eq:tautwoinhomeqnbounds} for the solution of \eqref{eq:xzeroonedecoupled0} vanishing at $\tau_2 = +\infty$, we infer the existence of a solution of \eqref{eq:distFouriereqnKST0} and satisfying the bound
\begin{equation}\label{eq:tildehboundundifferentiated}
\big\|\mathcal{F}\tilde{h}\big\|_{\tilde{X}_0}\lesssim \sup_{0<t\leq t_0}e^{cp\tau_2}\cdot\lambda_2^{-1}(t)\cdot\big\|F(t,\cdot)\big\|_{L^2_{r\,dr}}.
\end{equation}
Again using Plancherel's theorem for the distorted Fourier transform, this implies the estimate 
\begin{equation}\label{eq:hboundsl=0noweights}
\big\|h(t,\cdot)\big\|_{H^1_{r\,dr}}\lesssim_p \tau^{-p+}.  
\end{equation}
In order to complete the proof of the lemma, we need to establish analogous bounds for the differentiated functions $S^lh$, as well as the bounds with singular weights $r^{-\kappa}, \kappa = 1, 2$. 
\\

To begin with, we apply the operator $S = t\partial_t + r\partial_r$ to the equation \eqref{eq:heqnreformulated}. This results in 
\begin{equation}\label{eq:Sequationlambdatwo}\begin{split}
&{-}(Sh)_{tt} +  (Sh)_{rr} + \frac{1}{r}(Sh)_r - 4\frac{\cos\big(2Q_2\big)}{r^2}\cdot Sh\\& = 2F + SF +  4S\Big(\frac{\cos\big(2\tilde{Q}_2\big) - \cos\big(2Q_2\big)}{r^2}\Big)\cdot h + 4\frac{S\big(\cos\big(2Q_2\big)\big)}{r^2}\cdot h
\end{split}\end{equation}
Thanks to the easily verified bound 
\begin{align*}
\Big|4\frac{S\big(\cos\big(2Q_2\big)\big)}{r^2}\Big|\lesssim \lambda_2^2, 
\end{align*}
together with the bound 
\begin{align*}
\Big|4S\Big(\frac{\cos\big(2\tilde{Q}_2\big) - \cos\big(2Q_2\big)}{r^2}\Big)\Big|\lesssim t^{-2}\cdot |\log t|^{-1}, 
\end{align*}
we conclude that if we denote the right hand side of \eqref{eq:Sequationlambdatwo} by $H$, then we have the estimate 
\begin{align*}
\sup_{0<t\leq t_0}\tau_2^{\frac{\beta}{\beta+1}}\cdot e^{c q\cdot \tau_2}\cdot \lambda_2^{-1}\big\|H(t,\cdot)\big\|_{L^2_{r\,dr}}\lesssim_q\sum_{l = 0, 1}\sup_{0<t\leq t_0}e^{cp\tau_2}\cdot\big\|S^l F(t,\cdot)\big\|_{L^2_{r\,dr}}, 
\end{align*}
where $q<p$ can be chosen arbitrarily, This is indeed a consequence of the already established bounds on $h$. Applying the already established undifferentiated bounds with $Sh$ taking the role of $h$, $H$ the role of $F$, $q_1<q$ the role of $q$ and $q$ the role of $p$, we infer the bound 
\begin{equation}\label{eq:hboundsl=1noweights}
\big\|Sh(t,\cdot)\big\|_{H^1_{r\,dr}}\lesssim_p \tau^{-p+}.  
\end{equation}
Using induction on $l$, we similarly infer the estimate 
\begin{equation}\label{eq:hboundsl=lnoweights}
\big\|S^lh(t,\cdot)\big\|_{H^1_{r\,dr}}\lesssim_{p,l} \tau^{-p+}.  
\end{equation}
It remains to establish the weighted bounds stated at the end of Lemma~\ref{lem:outerpropagatorinhombounds}. We prove the case $l = 0$, and the remaining estimates follow by differentiating the equation as before. The key is to invoke identity \eqref{eq:trickidentity} as well as estimate \eqref{eq:rminustwoestimate}, with $h$ replacing $\epsilon$. Thanks to the latter, we infer the bound 
\begin{equation}\label{eq:rminustwohbound}
\big\|r^{-2}h\big\|_{L^2_{r\,dr}}\lesssim\big\|t^{-1}\nabla h\big\|_{L^2_{r\,dr}} + \big\|t^{-1}r^{-1}h\big\|_{L^2_{r\,dr}} + \Big\|\Big(\frac{t^2 - r^2}{t^2}\partial_{rr} + \frac{1}{r}\partial_r - \frac{4}{r^2}\Big)h\Big\|_{L^2_{r\,dr}}.
\end{equation}
Thanks to the estimate 
\begin{align*}
\big\|\nabla h\big\|_{L^2_{r\,dr}} + \big\|\frac{h}{r}\big\|_{L^2_{r\,dr}}\lesssim \big\|L_t^{\frac12}h\big\|_{L^2_{r\,dr}} + \lambda_2(t)\cdot \big\|h\big\|_{L^2_{r\,dr}}, 
\end{align*}
the already established bounds imply that 
\begin{align*}
\big\|t^{-1}\nabla h\big\|_{L^2_{r\,dr}} + \big\|t^{-1}r^{-1}h\big\|_{L^2_{r\,dr}} \lesssim \tau^{-p+}. 
\end{align*}
It remains to control the degenerate second order elliptic operator above, for which we use \eqref{eq:trickidentity}. This implies in particular that 
\begin{align*}
\lambda_2^{-1}\cdot \Big\|\Big(\frac{t^2 - r^2}{t^2}\partial_{rr} + \frac{1}{r}\partial_r - \frac{4}{r^2}\Big)h\Big\|_{L^2_{r\,dr}}&\lesssim \big(\lambda_2\cdot t\big)^{-2}\cdot\big[\lambda_2\big\|S^2h\big\|_{L^2_{r\,dr}} + \lambda_2\big\|Sh\big\|_{L^2_{r\,dr}}\big]\\
& +  \big(\lambda_2\cdot t\big)^{-1}\cdot\big[\big\|\partial_t Sh\big\|_{L^2_{r\,dr}} + \big\|\partial_t h\big\|_{L^2_{r\,dr}}\big]\\
& + \lambda_2^{-1}\cdot\big\| \Box h - \frac{4h}{r^2}\big\|_{L^2_{r\,dr}}. 
\end{align*}
The first two terms are bounded in terms of the already established estimates, while for the last term on the right we take advantage of \eqref{eq:linearisationaroundtildeQtwo}. We conclude that 
\begin{align*}
\lambda_2^{-1}\cdot\big\| \Box h - \frac{4h}{r^2}\big\|_{L^2_{r\,dr}}&\lesssim \lambda_2^{-1}\cdot\big\|F\big\|_{L^2_{r\,dr}} +  \lambda_2^{-1}\cdot\big\|4\frac{\cos\big(2\tilde{Q}_2\big)-1}{r^2}\cdot h\big\|_{L^2_{r\,dr}}\\
&\lesssim \lambda_2^{-1}\cdot\big\|F\big\|_{L^2_{r\,dr}} + \lambda_2\cdot \big\|h\big\|_{L^2_{r\,dr}}\lesssim \tau^{-p+}. 
\end{align*}
Noting that $\lambda_2\lesssim \tau^{0+}$, we then infer that 
\[
\Big\|\Big(\frac{t^2 - r^2}{t^2}\partial_{rr} + \frac{1}{r}\partial_r - \frac{4}{r^2}\Big)h\Big\|_{L^2_{r\,dr}}\lesssim \tau^{-p+}. 
\]
Using \eqref{eq:rminustwohbound} and the other estimates above, we then obtain the desired estimate 
\begin{align*}
\big\|r^{-2}h\big\|_{L^2_{r\,dr}}\lesssim\tau^{-p+}. 
\end{align*}
Interpolating between this bound and 
\[
\big\|h\big\|_{L^\infty_{dr}}\lesssim \big\|\nabla h\big\|_{L^2_{r\,dr}} + \big\|\frac{h}{r}\big\|_{L^2_{r\,dr}}\lesssim \tau^{-p+}, 
\]
the estimate 
\begin{align*}
\big\|\frac{h}{r}\big\|_{L^4_{r\,dr}}\lesssim \tau^{-p+}
\end{align*}
follows.
\end{proof}

\section{Construction of two bubble solution; the main mechanism for the dynamics of inner bubble}

Here we briefly describe a key detail of our method to construct an approximate finite time blow up solution of the form 
\begin{equation}\label{eq:twobubbleblowup}
u_N(t, r) = Q\big(\lambda_1(t)r) - \tilde{Q}_2(t, r)  + v_N(t, r),\,v_N = \sum_{j=0}^N h_j. 
\end{equation}
The exact scaling parameter $\lambda_1(t)$ will be chosen in tandem with the correction $v_N$. Its leading order already emerges with the choice of the first approximation 
\[
h_0.
\]
This correction shall be chosen by essentially approximating the first equation in \eqref{eq:exactvNequation} by 
\begin{equation}\label{eq:firstapproximationcrude}
h_{0,rr} + \frac{1}{r}h_{0,r} - \frac{4\cos\big(2Q_1\big)}{r^2}h_0 = E_2,\,Q_1 = Q(R),\,R = \lambda_1(t)\cdot r.
\end{equation}
Thus we neglect the effect of the lower frequency potential term $\tilde{Q}_2$, and think of $h_0$ as essentially localised to the inner region $r\lesssim \lambda_1^{-1}$, where we shall pick $\lambda_1\gg \lambda_2$. 
In fact, we shall include a further correction term on the right to accommodate small adjustments of the scaling parameter $\lambda_1$, but to leading order, the preceding equation gives the precise asymptotics. We infer the precise form of $E_2$ from \eqref{eq:exactvNequation}. We shall use the variation of constants formula associated to the operator $\mathcal{L}$ in \eqref{eq:operatormathcalL} with its fundamental system 
\[
\Phi(R): = R^{-\frac12}\cdot \phi_0(R), \Theta(R): = R^{-\frac12}\cdot \theta_0(R),
\]
recalling \eqref{eq:phi0theta0}, resulting in 
\begin{equation}\label{eq:simplevariationofconstantshzero}
h_0 = \Theta(R)\cdot \int_0^R \lambda_1^{-2}\cdot E_2\Phi(s)\,sds - \Phi(R)\cdot \int_0^R \lambda_1^{-2}\cdot E_2\Theta(s)\,sds.
\end{equation}
The term with an outer factor $\Theta(R)$ leads generically to quadratic growth toward $R = +\infty$, which forces the vanishing condition 
\begin{equation}\label{eq:Etwoorthogonality}
\int_0^\infty \lambda_1^{-2}\cdot E_2\cdot\Phi(s)\,sds = 0. 
\end{equation}
For technical reasons, it is better to approximate the function $\lambda_1^{-2}\cdot E_2$ by the following one:
\begin{equation}\label{eq:tildeE2}
 \tilde{E}_2: = \frac{\lambda_1''}{\lambda_1^3}\cdot\Phi(R) + \big(\frac{\lambda_1'}{\lambda_1^2}\big)^2\cdot \big(R\Phi'(R) - \Phi(R)\big) - 8\cdot\big(\frac{\lambda_2}{\lambda_1}\big)^2\cdot \big[\cos\big(2Q(R)\big) - 1\big]
\end{equation}
This expression will easily emerge when explicitly computing $E_2$. Making the corresponding replacement in \eqref{eq:simplevariationofconstantshzero}, as well as in \eqref{eq:Etwoorthogonality}, we infer that 
\begin{align*}
&\int_0^\infty  \Big[\frac{\lambda_1''}{\lambda_1^3}\cdot\Phi(R) + \big(\frac{\lambda_1'}{\lambda_1^2}\big)^2\cdot \big(R\Phi'(R) - \Phi(R)\big)\Big]\cdot \Phi(R)\,R\,dR\\
& =  -\big(\frac{\lambda_2}{\lambda_1}\big)^2\cdot 8\int_0^\infty \big[{-}\cos\big(2Q(R)\big) + 1\big]\cdot R\Phi(R)\,dR. 
\end{align*}
Taking advantage of the relations
\begin{equation}\label{eq:integralrelationsone}
\int_0^\infty\Phi^2(R)R\,dR = 3\pi,\,8\int_0^\infty \big[{-}\cos\big(2Q(R)\big) + 1\big]\cdot R\Phi(R)\,dR = 12\pi. 
\end{equation}
We then derive the following equation which gives the leading order approximation to the high frequency scaling parameter $\lambda_1(t)$: 
\begin{equation}\label{eq:lambdaoneleadingorder}
\boxed{\frac{\lambda_1''}{\lambda_1^3} - 2 \big(\frac{\lambda_1'}{\lambda_1^2}\big)^2 = -4\cdot\big(\frac{\lambda_2}{\lambda_1}\big)^2}
\end{equation}
It can be easily verified that this equation admits a solution $\lambda_1(t)$ of the form $\lambda_1(t) = e^{\alpha(t)}$ with $\alpha(t)\sim \big|\log t\big|^{\beta+1}$, provided $\lambda_2(t) = \frac{\big|\log t\big|^{\beta}}{t}$ as before. We shall do this below, and moreover we shall analyse suitable perturbations of \eqref{eq:lambdaoneleadingorder} which shall be required to {\it{refine the first approximation}} $h_0$, resulting in the 
\[
v_N = \sum_{j=0}^N h_j. 
\]

\section{Details on the modulation equation \eqref{eq:lambdaoneleadingorder} and a perturbed version}

 We start by giving details on the solution of 
 \[
 \frac{\lambda_1''}{\lambda_1^3} - 2\cdot \big(\frac{\lambda_1'}{\lambda_1^2}\big)^2 = -4\frac{\lambda_2^2}{\lambda_1^2}\Longleftrightarrow \frac{\lambda_1''}{\lambda_1} - 2\cdot \big(\frac{\lambda_1'}{\lambda_1}\big)^2 = -4\lambda_2^2. 
 \]
 We shall set 
 \[
 \lambda_1(t) = e^{\alpha(t)}.
 \]
 Then the preceding equation becomes 
 \[
 \alpha'' - (\alpha')^2 = -4\lambda_2^2
 \]
 Here we expect $\alpha'(t)\sim -2\lambda_2(t)$, and it is natural to consider the variable $\zeta(t): = (\alpha')^{-1}(t)$ instead, which results in 
 \[
 \zeta' = (2\lambda_2\zeta)^2 - 1. 
 \]
 Write 
 \begin{equation}\label{eq:zetaansatz}
 \zeta(t) = -\frac{t}{2|\log t|^{\beta}} + \frac{t}{8|\log t|^{2\beta}} + \frac{\beta\cdot t}{8|\log t|^{2\beta+1}} + w. 
 \end{equation}
 Then we infer 
 \begin{align*}
 \zeta' &= -\frac{1}{2|\log t|^{\beta}} - \frac{\beta}{2|\log t|^{\beta+1}} + O\big(\frac{1}{|\log t|^{2\beta}}\big)+ w'\\
 & = \big({-}1 + \frac{1}{4|\log t|^{\beta}} + \frac{\beta}{4|\log t|^{\beta+1}} + 2\lambda_2 w\big)^2  - 1\\
 & = -\frac{1}{2|\log t|^{\beta}} - \frac{\beta}{2|\log t|^{\beta+1}} + O\big(\frac{1}{|\log t|^{2\beta}}\big) - 4\lambda_2 w\cdot \big(1+ o(1)\big)\\
 & + O\big(\lambda_2^2 w^2\big). 
 \end{align*}
 Here the term $o(1)$ vanishes as $t\rightarrow 0$ like a inverse power of $|\log t|$. 
 We then deduce the equation 
 \begin{align*}
 w' = -4\lambda_2 w\cdot \big(1+ o(1)\big) +  O\big(\frac{1}{|\log t|^{2\beta}}\big) + O\big(\lambda_2^2 w^2\big). 
 \end{align*}
 By formulating this as a fixed point problem, we infer the existence of a solution of class $w\in C^1\big((0, 1]\big)$: in fact, denoting 
 \begin{align*}
  \lambda_2 \cdot \big(1+ o(1)\big) = \tilde{\lambda}_2, 
 \end{align*}
 we can write
 \begin{equation}\label{eq:wfixedpoint}
 w(t) = e^{4\int_t^{t_0} \tilde{\lambda}_2(s)\,ds}\cdot \int_{0}^{t} e^{-4\int_{t'}^{t_0} \tilde{\lambda}_2(s')\,ds'}\cdot F(t',w)\,dt',
 \end{equation}
 where we set 
 \[
 F(t, w) = O\big(\frac{1}{|\log t|^{2\beta}}\big) + O\big(\lambda_2^2 w^2\big).
 \]
 Observe that setting $\tau': = 4\int_{t'}^{t_0} \tilde{\lambda}_2(s')\,ds'$, and similarly $\tau: = 4\int_{t}^{t_0} \tilde{\lambda}_2(s')\,ds'$, we have after a simple change of variables
 \begin{align*}
  e^{4\int_t^{t_0} \tilde{\lambda}_2(s)\,ds}\cdot \int_{0}^{t} e^{-4\int_{t'}^{t_0} \tilde{\lambda}_2(s')\,ds'}\cdot\big|\log t'\big|^{-2\beta}\,dt'\lesssim  e^{\tau}\cdot \int_{\tau}^\infty e^{-\tau'}\cdot t'\cdot |\log t'|^{-3\beta}\,d\tau', 
 \end{align*}
 where we interpret $t' = t'(\tau')$ as a function of $\tau'$. Since $t'$ is a decreasing function of $\tau'$, we infer that 
  \begin{align*}
 e^{\tau}\cdot \int_{\tau}^\infty e^{-\tau'}\cdot t'\cdot |\log t'|^{-3\beta}\,d\tau'\leq t\cdot |\log t|^{-3\beta}. 
 \end{align*}
 If we pick $t_0$ sufficiently small depending on $\beta$, we find that 
\begin{align*}
 e^{4\int_t^{t_0} \tilde{\lambda}_2(s)\,ds}\cdot \int_0^t e^{-4\int_{t'}^{t_0} \tilde{\lambda}_2(s')\,ds'}\cdot \frac{1}{|\log t'|^{2\beta}}\,dt' \ll \frac{t}{|\log t|^{2\beta}}. 
\end{align*}
Using a standard continuity argument, we deduce that \eqref{eq:wfixedpoint} admits a solution on $(0, t_0]$ satisfying the bound 
\begin{equation}\label{eq:wbasebound}
\big|w(t)\big|\lesssim \frac{t}{|\log t|^{2\beta}},\,t\in (0, t_0]. 
\end{equation}
From here we obtain the function $\zeta$ by taking advantage of \eqref{eq:zetaansatz}. \\
We can also infer derivative bounds on the function $\zeta$ just constructed, by applying $t\partial_t$ to the fixed point equation for $w$. This results in  
\begin{equation}\label{eq:wfixedpointdifferentiated}\begin{split}
\big(t\partial_t\big)^lw(t) &= \sum_{l_1+l_2 = l, l_2\geq 1}C_{l_1,l_2}\big(t\partial_t\big)^{l_1}\big(e^{4\int_t^{t_0} \tilde{\lambda}_2(s)\,ds}\big)\cdot \big(t\partial_t\big)^{l_2-1}\big(e^{-4\int_{t}^{t_0} \tilde{\lambda}_2(s')\,ds'}\cdot F(t,w)\big)\\
& +  \big(t\partial_t\big)^{l}\big(e^{4\int_t^{t_0} \tilde{\lambda}_2(s)\,ds}\big)\cdot \int_{0}^{t} e^{-4\int_{t'}^{t_0} \tilde{\lambda}_2(s')\,ds'}\cdot F(t',w)\,dt',\,l\geq 1. 
\end{split}\end{equation}
Using the bounds 
\begin{align*}
\Big| \big(t\partial_t\big)^{l}\big(e^{4\int_t^{t_0} \tilde{\lambda}_2(s)\,ds}\Big|\lesssim_l \big|\log t\big|^{l\cdot \beta}\cdot e^{4\int_t^{t_0} \tilde{\lambda}_2(s)\,ds},\,l\geq 0, 
\end{align*}
and exploiting that $l_2 -1<l$, an induction on $l\geq 0$, starting with \eqref{eq:wbasebound} leads to the bounds 
\begin{equation}\label{eq:wderivativebounds}
\big|\big(t\partial_t\big)^lw(t)\big|\lesssim_l \big|\log t\big|^{\beta\cdot (l-2)}\cdot t,\,l\geq 0. 
\end{equation}
In turn we infer the similar bound 
\begin{equation}\label{eq:zetaderivativebounds}
\big|\big(t\partial_t\big)^l\zeta(t))\big|\lesssim_l \big|\log t\big|^{\beta\cdot (l-2)}\cdot t,\,l\geq 0. 
\end{equation}

In our construction of the approximate solution, we shall have to replace the modulation equation for $\lambda_1$ by replacing $\lambda_2$ by a perturbation which in turn depends on $\lambda_1$. Specifically, we shall replace $\zeta$ as constructed in the preceding by 
\[
\tilde{\zeta} = \zeta + \nu, 
\]
and we replace $\lambda_2$ by 
\[
\lambda_2 + m.
\]
Here $m$ shall depend (mildly) on $\nu$, via the dependence of $\lambda_1$ on $\nu$. We derive the following equation for $\nu$, where all displayed functions are interpreted as functions of $t\in (0, t_0]$:
\begin{equation}\label{eq:nueqn}
\nu' = 4\big(m^2 + 2\lambda_2 m\big)\zeta^2 + 4(\lambda_2 + m)^2\big(\nu^2 + 2\nu\zeta\big). 
\end{equation}
 Here $\zeta\sim -\frac{t}{2|\log t|^{\beta}}$ is the function constructed earlier. Then we shall require the following lemma:
 \begin{lem}\label{lem:modulationfinetuning} Let 
 \[
 \tau: = e^{2(\beta+1)^{-1}|\log t|^{\beta+1}},
 \]
 and assume the bound 
 \[
 |m(t)|\leq \tau^{-\frac12}.
 \]
 Then letting $t_0 = t_0(\beta)$ be sufficiently small and positive, the equation \eqref{eq:nueqn} admits a solution $\nu\in C^1\big((0, t_0]\big)$ satisfying the bound 
 \begin{align*}
 \big|\nu(t)\big|\leq \tau^{-\frac12}. 
 \end{align*}
 In particular, we have $\big|\nu(t)\big|\ll \big|\zeta(t)\big|$ for $0<t\leq t_0$, $t_0$ sufficiently small. 
 Finally, we have the estimate 
 \begin{align*}
 \big|(t\partial_t)^k\nu\big|\lesssim  \tau^{-\frac12}\cdot \log^k\tau\cdot \big(1 +\sum_{j=0}^{k-1}\sup_{0<t\leq t_0}\tau^{\frac12}\big|(t\partial_t)^j m\big|\big)^l.
 \end{align*}
 
 \end{lem}
 \begin{proof} This follows by rewriting \eqref{eq:nueqn} in fixed point form 
 \begin{equation}\label{eq:nuformula}
 \nu(t) = e^{-2\int_t^{t_0} \tilde{\lambda}_2(s)\,ds}\cdot \int_0^{t} e^{2\int_{t'}^{t_0} \tilde{\lambda}_2(s')\,ds'}\cdot G(t', \nu)\,dt',
 \end{equation}
 where we set
 \begin{align*}
  \tilde{\lambda}_2 = 4 (\lambda_2 + m)^2\cdot \zeta
 \end{align*}
 whence $| \tilde{\lambda}_2|\lesssim \lambda_2$, and further 
 \begin{align*}
 G =  4\big(m^2 + 2\lambda_2 m\big)\zeta^2 +  4(\lambda_2 + m)^2\nu^2. 
 \end{align*}
 Our assumptions imply that 
 \begin{align*}
  \big(m^2 + 2\lambda_2 m\big)\zeta^2\ll \tau^{-\frac12}.
 \end{align*}
 Moreover, if $\nu\leq \tau^{-\frac12}$, then 
 \begin{align*}
 (\lambda_2 + m)^2\nu^2\ll \tau^{-\frac12}
 \end{align*}
 if $t_0$ is sufficiently small. The first part of the lemma follows by a simple continuity argument, while the derivative bounds follow via induction and differentiating the fixed point equation satisfied by $\nu$, and the bound 
 \begin{align*}
\big|(t\partial_t)^l\big(e^{2\int_t^{t_0} \tilde{\lambda}_2(s;\cdot)\,ds}\big)\big|\lesssim_l \big|\log t\big|^{l\beta}\cdot \big(1 +\tau^{-\frac12}\cdot\sum_{j=0}^{l}\sup_{0<t\leq t_0}\tau^{\frac12}\big|(t\partial_t)^j m\big|\big)^l. 
\end{align*}
 In order to get derivative type bounds with respect to $m$, we need to refine the preceding approach a bit. First, introduce the following family of norms to control the function $m$:
 \begin{equation}\label{eq:mnorms}
 \big\|m\big\|_{p,l} := \sup_{0<t\leq t_0}\tau^p\cdot \sum_{j=0}^l\big|(t\partial_t)^jm(t,\cdot)\big|,\,0\leq l, p. 
 \end{equation}
 where we interpret $\tau = \tau(t)$. 
 \\
 Next, returning to \eqref{eq:nuformula}, write
 \begin{align*}
 2\int_t^{t_0} \tilde{\lambda}_2(s;m)\,ds &= 2\int_t^{t_0} \tilde{\lambda}_2(s;0)\,ds + 2\int_0^{t_0} [\tilde{\lambda}_2(s;m)\,ds - \tilde{\lambda}_2(s;0)]\,ds
 \\& -  2\int_0^{t} [\tilde{\lambda}_2(s;m)\,ds - \tilde{\lambda}_2(s;0)]\,ds
 \end{align*}
 Proceeding similarly for the integral $2\int_{t'}^{t_0}$, the second term on the right cancels in the difference of these integrals: for $0<t'\leq t\leq t_0$, we write
 \begin{align*}
& -2\int_t^{t_0} \tilde{\lambda}_2(s;m)\,ds + 2\int_{t'}^{t_0} \tilde{\lambda}_2(s;m)\,ds\\
& = 2\int_{t'}^{t} \tilde{\lambda}_2(s;0)\,ds + 2\int_{t'}^{t} [\tilde{\lambda}_2(s;m)\,ds - \tilde{\lambda}_2(s;0)]\,ds
\end{align*}
 
 Now consider the expression 
 \begin{align*}
 H(t; m, n ;\alpha): = e^{2\int_t^{t'} \tilde{\lambda}_2(s;m+\alpha n)\,ds}. 
 \end{align*}
 Observe that if we write
 \begin{align*}
  \triangle\big(e^{2\int_t^{t_0} \tilde{\lambda}_2(s;\cdot)\,ds}\big)(t; m, n): = e^{2\int_t^{t_0} \tilde{\lambda}_2(s; m+n)\,ds} - e^{2\int_t^{t_0} \tilde{\lambda}_2(s; m)\,ds}, 
 \end{align*}
  then we have the identity 
 \begin{align*}
&\triangle\big(e^{2\int_t^{t_0} \tilde{\lambda}_2(s;\cdot)\,ds}\big)(t; m, n) = \int_0^1\frac{\partial}{\partial\alpha} H(t; m, n,\alpha)\,d\alpha. 
 \end{align*}
 Due to the bound 
 \begin{align*}
 \Big|\frac{\partial}{\partial\alpha}H(t; m, n ;\alpha)\Big|\lesssim \tau^{-p}\cdot\Big\|n\Big\|_{p,0},
 \end{align*}
 for $p\geq 0$, valid provided $\Big\|m\Big\|_{p,0} + \Big\|n\Big\|_{p,0}\leq 1$, 
we deduce that 
\begin{align*}
\Big|\triangle\big(e^{2\int_t^{t_0} \tilde{\lambda}_2(s;\cdot)\,ds}\big)(t; m, n)\Big|\lesssim \tau^{-p}\cdot\Big\|n\Big\|_{p,0},\,p\geq 0, 
\end{align*}
again under the condition that 
\begin{equation}\label{eq:aprioriboundonmn}
\Big\|m\Big\|_{p,0} + \Big\|n\Big\|_{p,0}\leq 1.
\end{equation}
Using the Leibniz type relation 
\begin{equation}\label{eq:Leibnizdifferencing}\begin{split}
&\triangle\big(f\cdot g\big)(t;m,n)\\
& = \big(\triangle f\cdot  \triangle g\big)(t;m,n) + \big(f\cdot \triangle g\big)(t;m,n) +  \big(g\cdot \triangle f\big)(t;m,n),
\end{split}\end{equation}
applied iteratively, as well as a continuity argument, one infers the bound
\begin{equation}\label{eq:nudifferencebounds}
\big|\triangle\nu(t;m,n)\big|\lesssim \tau^{-p}\cdot\Big\|n\Big\|_{p,0},\,p\geq 0, 
\end{equation} 
under the assumption \eqref{eq:aprioriboundonmn}.
We shall also require a version of these inequalities involving multiple applications of the operator $t\partial_t$. 
Using induction as well as continuity arguments, one concludes in analogy to the preceding that 
\begin{equation}\label{eq:generalnubounds}
\big|\big(t\partial_t\big)^l\triangle\nu(t;m,n)\big|\lesssim_{l} \tau^{-p}\cdot\big(\Big\|m\Big\|_{p,l} +\Big\|n\Big\|_{p,l}+1\big)^{l-1}\cdot\Big\|n\Big\|_{p,l},\,p\geq 0. 
\end{equation}
We note that setting
 \[
 \alpha'(t) = -\big(\zeta(t) + \nu(t)\big)^{-1}, 
 \]
 with $\zeta, \nu$ as before, and 
 \begin{equation}\label{eq:alphadefinition}
 \alpha(t) = {-}\int_0^{t}\big[\big(\zeta(s) + \nu(s)\big)^{-1} - \zeta^{-1}(s)\big]\,ds +  \int_t^{t_0}\big(\zeta(s)\big)^{-1}\,ds,
 \end{equation}
 we have 
 \begin{align*}
 \lim_{t\rightarrow 0}\Big(\alpha(t) -  \int_t^{t_0}\big(\zeta(s)\big)^{-1}\,ds\Big) = 0
 \end{align*}
 uniformly over $\nu$ as determined via the preceding lemma, with arbitrary function $m$ satisfying the bound in the statement of the lemma. In particular, it follows that setting 
 \[
 \lambda_1(t) = e^{\alpha(t)},
 \]
 we have 
 \[
 \lambda_1(t)\geq \tau^{1-}
 \]
 provided $t\in (0, t_0]$ with $t_0$ sufficiently small, uniformly over $\nu$ determined via the preceding lemma. 
 \\
 Taking advantage of the preceding bound for $\nu$, we then infer the following similar bounds for $\lambda_1(t)$: 
 \begin{equation}\label{eq:generallambdaonebounds}\begin{split}
&\big|\lambda_1^{-1}\cdot\big(t\partial_t\big)^l\lambda_1(t;m)\big|\lesssim_{l}\tau^{0+}\cdot\big(1+\Big\|m\Big\|_{p,l}\big)^{l-1},\\
&\big|\lambda_1^{-1}\cdot\big(t\partial_t\big)^l\triangle\lambda_1(t;m,n)\big|\lesssim_{l} \tau^{-p+}\cdot\big(\Big\|m\Big\|_{p,l} +\Big\|n\Big\|_{p,l}+1\big)^{l-1}\cdot\Big\|n\Big\|_{p,l}
\end{split}\end{equation}
 \end{proof}
  where we restrict to $p\geq \frac12$, say, and we again assume \eqref{eq:aprioriboundonmn}.
 
  \section{Construction of accurate approximate solution}
  
  The goal in this part is to prove the following 
  \begin{prop}\label{prop:approxsoln} Let $\beta>\frac32$ and set $\lambda_2(t): = \frac{|\log t|^{\beta}}{t}$. Let $N\geq 1$ be given. Then there exists $t_0 = t_0(\beta,  N)$, such that there is a continuously differentiable function $\lambda_1\in C^1\big((0, t_0]\big)$ with 
  \[
  \lambda_1(t)\in [e^{\frac12 (\beta+1)^{-1}|\log t|^{\beta+1}}, e^{2(\beta+1)^{-1}|\log t|^{\beta+1}}]
  \]
 and an approximate solution 
  \[
  u_N(t, r) = Q\big(\lambda_1(t)r\big) - \tilde{Q}_2(t, r) + v_N(t, r)
  \]
  satisfying the bound 
  \begin{align*}
  \big\|{-}u_{N, tt} + u_{N, rr} + \frac{1}{r}u_{N,r} - 2\frac{\sin (2u_N)}{r^2}\big\|_{L^2_{r\,dr}(r\lesssim t)}\lesssim \tau^{-N},
  \end{align*}
  where we define the {\it{inner wave time $\tau$}} by means of 
  \[
  \tau = \frac{\lambda_1(t)}{\lambda_2(t)}. 
  \]
  \end{prop}
  \begin{rem}\label{rem:tauredefinition} We have slightly changed here the definition of $\tau$ compared to Lemma~\ref{lem:modulationfinetuning}. However, these two definitions essentially agree, since $\lambda_1$ will be constructed via Lemma~\ref{lem:modulationfinetuning} and suitable choice of $m$ satisfying the bound stated in that lemma, and then we have that 
  \begin{align*}
  \big(\frac{\lambda_1(t)}{\lambda_2(t)}\big)^{1-}\leq  e^{2(\beta+1)^{-1}|\log t|^{\beta+1}}\leq  \big(\frac{\lambda_1(t)}{\lambda_2(t)}\big)^{1+}
  \end{align*}
  for $0<t\leq t_0$ with $t_0$ sufficiently small. 
   \end{rem}
  \begin{proof} In the following, we shall first let $t_0$ a fixed small number, which shall have to be chosen small enough depending on various conditions depending on $\beta, N$. 
  Let $m\in C^0\big((0, t_0]\big)$ a function satisfying the bound 
  \[
  \big|m(t)\big|\leq e^{-(\beta+1)^{-1}|\log t|^{\beta+1}},\,t\in (0, t_0].
  \]
  Determine $\nu(t)$ by means of Lemma~\ref{lem:modulationfinetuning}, already requiring $t_0$ to be sufficiently small, and then set 
  \[
  \lambda_1(t): = e^{\alpha(t)},
  \]
   where 
   \[
   \alpha(t) = -\int_0^{t}\big[\big(\zeta(s) + \nu(s)\big)^{-1} - \zeta(s)^{-1}\big]\,ds + \int_{t}^{t_0}\zeta(s)^{-1}\,ds.
   \]
  We shall construct $v_N$ by a process of adding finitely many increments. The end result of this process will indeed be a highly precise approximate solution, provided $m$ solves a certain fixed point property approximately, which then determines the final choice of $\lambda_1(t)$. 
  As noted in the preceding section, we have 
  \[
  \lambda_1(t)\geq e^{\frac12 (\beta+1)^{-1}|\log t|^{\beta+1}}
  \]
  provided $t_0$ is sufficiently small. 
  \\
  The construction of the correction $v_N$ proceeds in the following steps. We shall set $v_N = \sum_{j=0}^N h_j$. To begin with, we record the equation which $v_N$ would have to satisfy if it led to an exact solution: 
  \begin{equation}\label{eq:exactvNequation-rec}\begin{split}
  &{-}v_{N,tt} + v_{N, rr} + \frac{1}{r}v_{N, r} - \frac{4\cos\big(2Q_1- 2\tilde{Q}_2\big)}{r^2}v_N = \sum_{j=1}^3 E_j,\\
  &E_1 = \frac{2\cos\big(2Q_1- 2\tilde{Q}_2\big)}{r^2}\cdot \big[\sin\big(2v_N\big) - 2v_N\big],\\
  &E_2 = \frac{2}{r^2}\big[\sin\big(2Q_1-2\tilde{Q}_2\big) - \sin\big(2Q_1\big) + \sin\big(2\tilde{Q}_2\big)\big] +Q_{1,tt},\\
  &E_3 =  \frac{2\sin\big(2Q_1- 2\tilde{Q}_2\big)}{r^2}\cdot \big[\cos\big(2v_N\big) - 1\big]. 
  \end{split}\end{equation}
  Throughout we set $Q_1 = Q(\lambda_1(t)r)$.
  \\
  
  {\bf{Step 0}}: {\it{Construction of the correction $h_0$.}} Here we use the simple elliptic equation 
  \begin{equation}\label{eq:elliptic}
   h_{0, rr} + \frac{1}{r}h_{0, r} - \frac{4\cos\big(2Q_1\big)}{r^2}h_0 = E_2 - \boxed{8(2m\lambda_2 + m^2)\cdot \big[\cos(2Q_1) - 1\big]}. 
  \end{equation}
  The boxed term on the far right is a correction term which is required due to the modified law for $\lambda_1$. Then we can formulate 
  \begin{lem}\label{lem:hzerocorrection} The equation \eqref{eq:elliptic} admits a solution $h_0\in H^3_{r\,dr}$ satisfying the uniform bound 
  \begin{align*}
  \big|h_0\big|\lesssim \tau^{-2+},\,r\lesssim t. 
  \end{align*}
  Denoting the scaling vector field $S = t\partial_t + r\partial_r$, we have the uniform bounds 
  \begin{align*}
  \big\|S^k h_0\big\|_{H^1_{r\,dr}}\lesssim_k  \tau^{-2+},\,k\geq 0,\,r\lesssim t. 
  \end{align*}
  provided $\big|(t\partial_t)^jm\big|\lesssim 1$, $0\leq j\leq k$. More precisely, interpreting $h_0$ as a function depending on $m$, and correspondingly writing 
\[
h_0(t, r) = h_0(t,r; m),
\]
we can estimate
  \begin{align*}
  \big|S^k h_0(\cdot,\cdot;m)\big|\lesssim _k \tau^{-2+}\cdot \big(1 + \tau^{-p}\cdot\Big\|m\Big\|_{p,k}\big),\,r\lesssim t, 
  \end{align*}
  provided we have the a priori bound \eqref{eq:aprioriboundonmn}.
    \end{lem}
  \begin{proof} To begin with, we can write 
  \begin{align*}
  E_2 &=  \frac{\lambda_1''}{\lambda_1}\cdot\Phi(R) + \big(\frac{\lambda_1'}{\lambda_1}\big)^2\cdot \big(R\Phi'(R) - \Phi(R)\big) - 2\frac{\sin\big(2 \tilde{Q}_2\big)}{r^2}\cdot \big[\cos(2Q_1\big) - 1\big]\\
& +  2\frac{\big[\cos\big(2 \tilde{Q}_2\big) - 1\big]}{r^2}\cdot \sin\big(2Q_1\big)
  \end{align*}
  Here we use the notation 
  \[
  \Phi(R) = RQ'(R) = \frac{4R^2}{1+R^4}. 
  \]
  Introduce the following modified source term, which also includes an extra factor $\lambda_1^{_2}$: 
  \begin{align*}
  \tilde{E}_2: = \frac{\lambda_1''}{\lambda_1^3}\cdot\Phi(R) + \big(\frac{\lambda_1'}{\lambda_1^2}\big)^2\cdot \big(R\Phi'(R) - \Phi(R)\big) - 8\cdot\big(\frac{\lambda_2}{\lambda_1}\big)^2\cdot \big[\cos\big(2Q(R)\big) - 1\big]
  \end{align*}
  This term, when adding the term 
  \[
 -8 \lambda_1^{-2}\cdot (2m\lambda_2 + m^2)\cdot \big[\cos(2Q_1) - 1\big],
  \]
  enjoys a special cancellation condition in light of the equation satisfied by $\lambda_1$, as will become apparent when using the variation of constants formula below. For this consider the time independent operator 
  \begin{equation}\label{eq:mathcalL}
\mathcal{L}: = \partial_{RR} + \frac{1}{R}\partial_R - \frac{4\cos(2Q(R))}{R^2},\,R = \lambda_1(t)r. 
\end{equation}
We have the following fundamental system for the homogeneous equation $\mathcal{L}f = 0$: 
\[
\frac14\Phi(R) = \frac{R^2}{1+R^4},\,\Theta(R): = \frac{-1 + 8R^4\log R + R^8}{4R^2(1+R^4)}. 
\]
We can then solve $\mathcal{L}h= f$ by the formula
\begin{equation}\label{eq:hformula}\begin{split}
h(t, R) &= \frac14\Theta(R)\cdot \int_0^R f(s)\Phi(s)\,s ds\\
& -  \frac14\Phi(R)\cdot \int_0^R f(s)\Theta(s)\,s ds.
\end{split}\end{equation}
We shall set 
\[
f = \lambda_1^{-2}E_2 - 8 \lambda_1^{-2}\cdot (2m\lambda_2 + m^2)\cdot \big[\cos(2Q_1) - 1\big].
\]
We can then write
\[
h = \sum_{j=1}^3 h_j, 
\]
where we define (with $f$ defined as in the immediately preceding)
\begin{align*}
&h_1: =  -  \frac14\Phi(R)\cdot \int_0^R f(s)\Theta(s)\,s ds\\
&h_2: =  \frac14\Theta(R)\cdot \int_0^R \tilde{f}(s)\Phi(s)\,s ds,
\end{align*}
where we set (recalling \eqref{eq:tildeE2})
\[
\tilde{f} = \tilde{E}_2 - 8 \lambda_1^{-2}(2m\lambda_2 + m^2)\cdot \big[\cos(2Q_1) - 1\big].
\]
Finally, we define 
\begin{align*}
h_3: =  \frac14\Theta(R)\cdot \int_0^R\big[f(s) - \tilde{f}(s)\big]\Phi(s)\,s ds
\end{align*}
In order to estimate these terms, the following observations, which result from Theorem~\ref{thm:approxoutersolnYManalogue}, Theorem~\ref{thm:epssolutioncompletion},  shall be useful:
\begin{equation}\label{eq:curdetileQ2bound1}
\big|\tilde{Q}_2(t,r)\big|\lesssim \big(\lambda_2 r\big)^2 + \frac{r^2}{t^2|\log t|},\,r\lesssim t. 
\end{equation}
In particular, we infer the bound
\begin{equation}\label{eq:curdetileQ2bound2}
\big(\lambda_1 r\big)^{-2}\cdot \big|\tilde{Q}_2(t,r)\big|\lesssim\big(\frac{\lambda_2}{\lambda_1}\big)^2,\,r\lesssim t. 
\end{equation}
We also have the estimate (again resulting from Theorem~\ref{thm:approxoutersolnYManalogue} and its proof)
\begin{equation}\label{eq:curdetileQ2bound3}
\big|\tilde{Q}_2(t,r) - Q\big(\lambda_2\cdot r\big)\big|\lesssim (\log |t|)^{-1}\big(\frac{r}{t}\big)\cdot \frac{(r\lambda_2)^2}{(r\lambda_2)^2 + 1},\,r\lesssim t. 
\end{equation}
The same bounds apply if we apply an arbitrary power of $S$ to the expressions on the left. 
\\

{\it{(1): The estimate for $h_1(t, r)$.}} Using \eqref{eq:curdetileQ2bound2}, observe the estimates
\begin{align*}
&\lambda_1^{-2}\cdot \big|2\frac{\sin\big(2 \tilde{Q}_2\big)}{r^2}\cdot \big[\cos(2Q_1\big) - 1\big]\big|\lesssim \big(\frac{\lambda_2}{\lambda_1}\big)^2\cdot \frac{R^4}{1+R^8},\\
&\lambda_1^{-2}\cdot \big| 2\frac{\big[\cos\big(2 \tilde{Q}_2\big) - 1\big]}{r^2}\cdot \sin\big(2Q_1\big)\big|\lesssim \big(\frac{\lambda_2}{\lambda_1}\big)^2\cdot\frac{R^2}{1+R^4},
\end{align*}
and finally 
\begin{align*}
\Big|\frac{\lambda_1''}{\lambda_1^3}\cdot\Phi(R) + \big(\frac{\lambda_1'}{\lambda_1^2}\big)^2\cdot \big(R\Phi'(R) - \Phi(R)\big)\Big|\lesssim( \lambda_1\cdot t)^{-2}\cdot\big|\log t\big|^{2\beta}\cdot \frac{R^2}{1+R^4}. 
\end{align*}
In light of $\big(\frac{\lambda_2}{\lambda_1}\big)^2 + ( \lambda_1\cdot t)^{-2}\cdot\big|\log t\big|^{2\beta}\lesssim \tau^{-2}$, the estimate 
\[
\big|h_1(t, r)\big|\lesssim \tau^{-2}
\]
follows easily from the asymptotic bounds for $\Phi(R), \Theta(R)$. 
\\

{\it{(2): The estimate for $h_2(t, r)$.}} Thanks to the equation 
\[
\frac{\lambda_1''}{\lambda_1} - 2\big(\frac{\lambda_1'}{\lambda_1}\big)^2 = -4(\lambda_2+m)^2, 
\]
as well as the identities 
\begin{equation}\label{eq:identities1}\begin{split}
&\int_0^\infty\Phi^2(R)R\,dR = 3\pi,\,8\int_0^\infty \big[{-}\cos\big(2Q(R)\big) + 1\big]\Phi(R)R\,dR = 12\pi,
\end{split}\end{equation}
we can write 
\begin{align*}
h_2 &= \chi_{R\lesssim 1}\cdot  \frac14\Theta(R)\cdot \int_0^R\tilde{f}(s)\Phi(s)\,s ds\\
& - \chi_{R\gtrsim 1}\cdot  \frac14\Theta(R)\cdot \int_R^\infty\tilde{f}(s)\Phi(s)\,s ds.
\end{align*}
Taking advantage of the same bound for $\frac{\lambda_1''}{\lambda_1^3}\cdot\Phi(R) + \big(\frac{\lambda_1'}{\lambda_1^2}\big)^2\cdot \big(R\Phi'(R) - \Phi(R)\big)$ as before, as well as the estimate 
\begin{align*}
\Big|4\cdot\big(\frac{\lambda_2}{\lambda_1}\big)^2\cdot \big[\cos\big(2Q(R)\big) - 1\big]\Big|\lesssim \big(\frac{\lambda_2}{\lambda_1}\big)^2\cdot \frac{R^4}{1+R^8},
\end{align*}
the bound 
\[
\big|h_2(t, r)\big|\lesssim \tau^{-2}
\]
easily follows, due to the fact that 
\begin{align*}
&\chi_{R\gtrsim 1}\cdot\big|\Theta(R)\cdot\int_R^\infty \frac{s^2}{1+s^4}\cdot \Phi(s)s\,ds\big|\lesssim 1,\\
&\chi_{R\lesssim 1}\cdot\big|\Theta(R)\cdot\int_0^R \frac{s^2}{1+s^4}\cdot \Phi(s)s\,ds\big|\lesssim 1
\end{align*}

{\it{(3): The estimate for $h_3(t, r)$.}} By definition of this contribution, it suffices to estimate the following two integrals:
\begin{align*}
&\big(I\big) := 2\Theta(R)\int_0^R\frac{\big[\cos\big(2 \tilde{Q}_2\big) - 1\big]}{S^2}\cdot \sin\big(2Q_1(S\big)\cdot \Phi(S)S\,dS\\
&\big(II\big) := 2\Theta(R)\int_0^R2\frac{\big[\sin\big(2 \tilde{Q}_2\big) - 4\big(\frac{\lambda_2}{\lambda_1}S\big)^2\big]}{S^2}\cdot \big[\cos(2Q_1\big) - 1\big]\cdot \Phi(S)S\,dS.
\end{align*}
Here $S = \lambda_1(t)\cdot s$ is the re-scaled variable. Estimating $\big|\sin\big(2Q_1(S\big)\big|\lesssim \frac{S^2}{1+S^4}$, and using \eqref{eq:curdetileQ2bound2}, we find concerning the first integral
\begin{align*}
\Big|\frac{\big[\cos\big(2 \tilde{Q}_2\big) - 1\big]}{S^2}\cdot \sin\big(2Q_1(S)\big)\Big|\lesssim \big(\frac{\lambda_2}{\lambda_1}\big)^4,
\end{align*}
and from there, under the restriction $R\lesssim \lambda_1\cdot t$,
\begin{align*}
\big|\big(I\big)\big|\lesssim \big(\lambda_1\cdot t\big)^2\cdot \big|\log \big(\lambda_1\cdot t\big)\big|\cdot \big(\frac{\lambda_2}{\lambda_1}\big)^4\lesssim \tau^{-2}\cdot (\log \tau)^C,\,C = C(\beta). 
\end{align*}
In order to estimate the integral $\big(II\big)$, we take advantage of \eqref{eq:curdetileQ2bound3} and Taylor expansion of the sine function to deduce (for $r\lesssim t$)
\begin{align*}
\Big|2\frac{\big[\sin\big(2 \tilde{Q}_2\big) - 4\big(\frac{\lambda_2}{\lambda_1}S\big)^2\big]}{S^2}\cdot \big[\cos(2Q_1\big) - 1\big]\Big|\lesssim  \big(\frac{\lambda_2}{\lambda_1}\big)^4.
\end{align*}
Then we infer the same bound for $\big(II\big)$ as for $\big(I\big)$. 
\\

In order to control the differentiated expressions $S^k h_0$, we need to apply powers of $S = t\partial_t + r\partial_r$ to the variation of constants formula, interpreting $R = \lambda_1(t)\cdot r$. Then we inductively infer the bounds (recall \eqref{eq:generallambdaonebounds})
\begin{align*}
&\big|S^k(R)\big|\lesssim _kR\cdot \tau^{0+}\cdot \big(1 + \tau^{-p}\cdot\Big\|m\Big\|_{p,k}\big)\\
&\big|S^k\big(\Phi(R)\big)\big|\lesssim_k \Phi(R)\cdot \tau^{0+}\cdot \big(1 + \tau^{-p}\cdot\Big\|m\Big\|_{p,k}\big),\\
&\big|S^k\big(\Theta(R)\big)\big|\lesssim_k \Theta(R)\cdot \tau^{0+}\cdot \big(1 + \tau^{-p}\cdot\Big\|m\Big\|_{p,k}\big).
\end{align*}
Now consider the differentiated expression 
\begin{align*}
&S^k\Big(\Phi(R)\cdot \int_0^R \frac{\cos\big(2Q(S)\big) - 1}{S^2}\cdot \sin\big(\tilde{Q}_2(t, s)\big)\cdot S\Theta(S)\,dS\Big)\\
& = S^k\Big(\Phi(r\lambda_1(t))\cdot \lambda_1(t)\int_0^{r} \frac{\cos\big(2Q(s\lambda_1(t))\big) - 1}{(s\lambda_1(t))^2}\cdot \sin\big(\tilde{Q}_2(t, s)\big)\cdot s\lambda_1(t)\Theta(s\lambda_1(t))\,ds\Big)\\
\end{align*}
which occurs in $h_1$. When the temporal derivative $t\partial_t$ hits the function $\sin\big(\tilde{Q}_2(t, s)\big)$, we write this as 
\begin{align*}
t\partial_t\big(\sin\big(\tilde{Q}_2(t, s)\big)\big) = S\big(\sin\big(\tilde{Q}_2(t, s)\big)\big) - s\partial_s\big(\sin\big(\tilde{Q}_2(t, s)\big)\big).
\end{align*}
While the first term on the right is bounded, and in fact a function in $H^1_{sds}$, see Theorem~\ref{thm:outersolnYManalogue}, the second term leads to an expression which can be integrated by parts, resulting in the integral 
\begin{align*}
\Phi(r\lambda_1(t))\cdot \lambda_1(t)\int_0^{r}  (\partial_s s)\big(s\lambda_1(t)\Theta(s\lambda_1(t))\cdot\frac{\cos\big(2Q(s\lambda_1(t))\big) - 1}{(s\lambda_1(t))^2}\big)\cdot \sin\big(\tilde{Q}_2(t, s)\big)\cdot\,ds
\end{align*}
as well as the boundary term 
\begin{align*}
R\Phi(R)\cdot \Theta(R)\cdot \frac{\cos(2Q(R)) - 1}{R^2}\cdot \sin\big(\tilde{Q}_2(t, r)\big).
\end{align*}
The last part of the lemma follows by applying this procedure sufficiently many times, and also recalling Lemma~\ref{lem:modulationfinetuning}.
\end{proof}

We will also require differencing bounds with respect to $m$ of the type as in \eqref{eq:generallambdaonebounds}. For this the following basic technical lemma is useful:
\begin{lem}\label{lem:technicalsymbolfunctionsdifferencing} Let $H: \mathbb{R}_+\longrightarrow \mathbb{R}_+$ a $C^\infty$ function satisfying the symbol bounds 
\begin{align*}
\big|\big(R\partial_R)^j H\big|\lesssim \langle R\rangle^{-\gamma},\,j\geq 0,
\end{align*}
for some fixed $\gamma\geq 0$. 
Writing $R = \lambda_1(t)r$ and considering (keeping in mind that $\lambda_1 = \lambda_1(t;m)$ depends implicitly on $m$)
\[
h(t,r; m): = H(\lambda_1(t)r), 
\]
and assuming the a priori bound \eqref{eq:aprioriboundonmn}, we then have the differencing bound
\begin{align*}
\big(t\partial_t\big)^l\triangle h(t,r;m,n)\big|\lesssim_{l}&\frac{1}{\tau^{p-}\langle R\rangle^{\gamma}}\cdot\big(\Big\|m\Big\|_{p,l} +\Big\|n\Big\|_{p,l}+1\big)^{l}\cdot\Big\|n\Big\|_{p,l},\\
&p\geq \frac12. 
\end{align*}
\end{lem}
\begin{proof} This follows by using 
\begin{align*}
&H\big(\lambda_1(t;m+n)r\big) - H\big(\lambda_1(t;m)r\big)\\& = \frac{\triangle\lambda_1(t;m,n)}{\lambda_1(t;m)}\cdot \int_0^1  \lambda_1(t;m)r\cdot H'\big([s\lambda_1(t;m+n) + (1-s)\lambda_1(t;m)]r\big)\,ds
\end{align*}
as well as the bounds \eqref{eq:generallambdaonebounds}.
\end{proof}

Applying the previous lemma to the terms in the variation of constants formula for $h_0$, we deduce the following refined bounds for $h_0$, the latter as in Lemma~\ref{lem:hzerocorrection}: 
\begin{equation}\label{eq:nzerodifferencingbounds}
\big|S^l\triangle h_0(t,r;m,n)\big|\lesssim_{k,l} \tau^{-2-p+}\cdot\big(\Big\|m\Big\|_{p,l} + \Big\|n\Big\|_{p,l}+1\big)^{l}\cdot\Big\|n\Big\|_{p,l},\,r\lesssim t. 
\end{equation}

 We next analyse the error generated by the approximation $v_N = h_0$ in \eqref{eq:exactvNequation-rec}. This error is given by 
 \begin{equation}\label{eq:e0error}
 e_0 = -h_{0,tt} + E_1(h_0) + E_3(h_0) - 8(2m\lambda_2 + m^2)\cdot \big[\cos(2Q_1) - 1\big].
 \end{equation}
 Then we have the following properties of this source term:
 \begin{lem}\label{lem:ezerobound} We have the estimate 
 \begin{align*}
 &\lambda_2^{-1}(t)\cdot \big\|S^k \big(e_0 +  8(2m\lambda_2 + m^2)\cdot \big[\cos(2Q_1) - 1\big]\big)\big\|_{L^2_{r\,dr}(r\lesssim t)}\\&\lesssim_{k,l} \tau^{-2+}\cdot\big(1 + \tau^{-l}\sum_{0\leq j\leq k+2}\sup_{0<t\leq t_0}\tau^{l}\cdot \big|(t\partial_t)^jm(t)\big|\big)
 \end{align*}
 provided we restrict to $l\geq \frac12$. Furthermore, setting $E_0: = e_0 +  8(2m\lambda_2 + m^2)\cdot \big[\cos(2Q_1) - 1\big]$ and writing $E_0 = E_0(t,r;m)$, we have the differencing bounds 
 \begin{align*}
&\lambda_2^{-1}(t)\cdot \big\|S^l\triangle E_0(t,r;m,n)\big\|_{L^2_{r\,dr}}\lesssim_{l} \tau^{-2-p+}\cdot\big(\Big\|m\Big\|_{p,l+2} +\Big\|n\Big\|_{p,l+2}+1\big)^{l}\cdot\Big\|n\Big\|_{p,l+2}.
 \end{align*}
 Here we restrict to $p\geq \frac12$, and assume \eqref{eq:aprioriboundonmn}.
 \end{lem}
  \begin{proof} For the term $-h_{0,tt}$ one relies on the proof of the preceding lemma, and uses that 
  \[
  \partial_t^2 = t^{-2}\cdot (t\partial_t)^2 - t^{-2}(t\partial_t), 
  \]
  When this operator falls on a factor of the form $g(\tilde{Q}_2)$ in the variations of constants formula, one re-writes this in terms of the scaling operator $S = t\partial_t + r\partial_r$ and performs integration by parts as before. The two extra temporal derivatives force control over $l+2$ derivatives of $m, n $. 
  The remaining terms $E_1(h_0), E_3(h_0)$ are much smaller. In fact, consider the term $E_3$. Using the Leibniz rule, we infer
  \begin{align*}
 &S^l\big(E_3(h_0)\big)\\& = \sum_{l_1+l_2 = l}C_{l_{1,2}}\cdot S^{l_1}\big( \frac{2\sin\big(2Q_1- 2\tilde{Q}_2\big)}{r^2}\big)\cdot  S^{l_2}\big(\cos(2h_0) - 1\big).
  \end{align*}
 Then we can directly estimate (expanding $\sin\big(2Q_1 - 2\tilde{Q}_2\big)$ and using Lemma~\ref{lem:technicalsymbolfunctionsdifferencing})
 \begin{align*}
&\big\|S^{l_1}\big( \frac{2\sin\big(2Q_1- 2\tilde{Q}_2\big)}{r^2}\big)(t,r)\big\|_{L^2_{r\,dr}}\lesssim_{l_1} \tau^{1+}, \\
&\big\|S^{l_1}\triangle\big( \frac{2\sin\big(2Q_1 - 2\tilde{Q}_2\big)}{r^2}\big)(t,r;m,n)\big\|_{L^2_{r\,dr}}\lesssim_{l_1} \tau^{1-p+}\cdot \big(1+ \Big\|m\Big\|_{p,l_1} +  \Big\|n\Big\|_{p,l_1}\big)^{l_1}\cdot\Big\|n\Big\|_{p,l_1},\\
&\,p\geq \frac12.
 \end{align*}
 Furthermore, expanding $\cos(2h_0)$ in a Taylor series and iteratively using \eqref{eq:Leibnizdifferencing} as well as \eqref{eq:nzerodifferencingbounds}, we deduce 
 \begin{align*}
&\big|S^{l_2}\big(\cos(2h_0) - 1\big)\big|\lesssim_{l_2} \tau^{-4+},\\
&\big|S^{l_2}\triangle \big(\cos(2h_0) - 1\big)(t,r;m,n)\big|\lesssim_{l_2} \tau^{-4-p+}\cdot \big(1+\Big\|m\Big\|_{p,l_2} + \Big\|n\Big\|_{p,l_2}\big)^{l_2}\cdot \Big\|n\Big\|_{p,l_2}.\\
 \end{align*}
 Further taking advantage of \eqref{eq:Leibnizdifferencing} as well as Holder's inequality, we infer that 
 \begin{align*}
 \big\|S^l\triangle \big(E_3(h_0)(\cdot,\cdot;m,n)\big)\big\|_{L^2_{r\,dr}}\lesssim_l\tau^{-3-p+}\cdot \big(\Big\|m\Big\|_{p,l} +\Big\|n\Big\|_{p,l}+1\big)^{l}\cdot\Big\|n\Big\|_{p,l}.
 \end{align*}
 The term $E_1(h_0)$ can be handled similarly. 
 
  \end{proof}

{\bf{Step 1}}:  {\it{Construction of the correction $h_1$.}} This step involves approximately solving the following wave equation 
\begin{equation}\label{eq:honewave}
-h_{1,tt} + h_{1,rr} + \frac{1}{r}h_{1,r} - \frac{4\cos\big(2Q_1 - 2\tilde{Q}_2\big)}{r^2}h_1 = e_{00} + \boxed{m_1\cdot \big[\cos\big(2Q_1\big) - 1\big]}, 
\end{equation}
 Here the term 
 \[
 e_{00} = e_0 + 8(2m\lambda_2 + m^2)\cdot \big[\cos(2Q_1) - 1\big],
 \]
  and the added boxed term on the right in \eqref{eq:honewave} serves two purposes, namely on the one hand ensuring a cancellation required to construct $h_1$ of smaller size than $h_0$, and second it removes a portion of the other boxed term in \eqref{eq:elliptic}, provided $m$ is chosen suitably (which we will do at the end). 
\\
The strategy to construct $h_1$ will be via a two stage process: first we solve an {\it{outer wave equation}} involving the potential term $-\frac{4\cos\big(2\tilde{Q}_2\big)}{r^2}$; then we solve an {\it{inner elliptic equation}} akin to the one satisfied by $h_0$. The former step has a s{\it{moothing effect at the outer scale $\lambda_2^{-1}$}}, while the latter step will have a {\it{shrinking effect}}. The role of the boxed term in \eqref{eq:honewave} will be to ensure good bounds for the second step. 
\begin{lem}\label{lem:honeequation} Let $m$ be as before. Then there exists a function $m_1$ on $(0, t_0]$ satisfying the bound 
\[
\big|m_1(t)\big|\ll_{t_0}\tau^{-\frac12}
\]
such that equation \eqref{eq:honewave} admits an approximate solution on $(0, t_0]\times [0,\infty)$ satisfying the bound ($k\geq 0, l\geq \frac12$)
\begin{align*}
&\big\|S^k h_1(t,\cdot)\big\|_{H^1_{r\,dr}}\lesssim_{k} \tau^{-2+}\cdot\big(1 + \tau^{-l}\cdot\sum_{0\leq j\leq k+4}\sup_{0<t\leq t_0}\tau^l\cdot \big|\big(t\partial_t\big)^j m(t)\big|\big),\\& 
\end{align*}
and further, denoting by $e_{10}$ the difference of the left hand side in \eqref{eq:honewave} and the right hand side there (the solution $h_1$ asserted here is only approximate!), we have  
\begin{align*}
&\lambda_2^{-1}(t)\cdot \big\|S^k e_{10}\big\|_{L^2_{r\,dr}(r\lesssim t)}\\&\hspace{2cm}\lesssim _{k} \tau^{-3+}\cdot\big(1 + \tau^{-l}\cdot\sum_{0\leq j\leq k+6}\sup_{0<t\leq t_0}\tau^l\cdot \big|\big(t\partial_t\big)^j m(t)\big|\big)
\end{align*}
We similarly have the differencing bounds 
\begin{align*}
&\big\| S^l\triangle h_{1}(t,r;m; n)\big\|_{H^1_{r\,dr}}\lesssim_{l} \tau^{-3-p+\epsilon}\cdot\big(1+\Big\|m\Big\|_{p, l+4}+\Big\|n\Big\|_{p, l+4}\big)^l\cdot\Big\|n\Big\|_{p, l+4},\\
&\lambda_2^{-1}(t)\cdot \big\| S^l\triangle e_{10}(t,r;m; n)\big\|_{L^2_{r\,dr}}\\&\hspace{2cm}\lesssim_{l} \tau^{-3-p+\epsilon}\cdot \big(1+\Big\|m\Big\|_{p, l+6} + \Big\|n\Big\|_{p, l+6}\big)^l\cdot \Big\|n\Big\|_{p, l+6}.
\end{align*}
\end{lem}
\begin{proof} We shall construct $h_1 = h_{10} + h_{11}$, where $h_{10}$ solves the following equation precisely
\begin{equation}\label{eq:honezero}
-h_{10,tt} + h_{10,rr} + \frac{1}{r}h_{10,r} - \frac{4\cos\big(2\tilde{Q}_2\big)}{r^2}h_{10} = e_{00},  
\end{equation}
while $h_{11}$ precisely solves 
\begin{equation}\label{eq:honeone}\begin{split}
h_{11,rr} + \frac{1}{r}h_{11,r} - \frac{4\cos\big(2Q_1\big)}{r^2}h_{11} &= -\frac{4\big[\cos(2Q_1 - 2\tilde{Q}_2\big) - \cos\big(2\tilde{Q}_2\big)\big]}{r^2}h_{10}\\
&\hspace{1cm} + \boxed{m_1\cdot \big[\cos\big(2Q_1\big) - 1\big]}
\end{split}\end{equation}
for suitable choice of $m_1$. The remaining error generated by the approximate solution $h_1 = h_{10} + h_{11}$ for \eqref{eq:honewave} is then given by 
\begin{equation}\label{eq:eonezero}
e_{10} = -h_{11,tt} + \frac{4\big[\cos(2Q_1 - 2\tilde{Q}_2\big) - \cos\big(2Q_1\big)\big]}{r^2}h_{11}.
\end{equation}
{\it{(i): Control of the solution of \eqref{eq:honezero}}}. We use Lemma~\ref{lem:outerpropagatorinhombounds} in conjunction with Lemma~\ref{lem:ezerobound} to infer the existence of a solution $h_{10}$ satisfying the bounds (as before $k\geq 0,\,l\geq \frac12$)
\begin{align*}
\big\|S^k h_{10}\big\|_{H^1_{r\,dr}}\lesssim_{k} \tau^{-2+}\cdot\big(1 + \tau^{-l}\sum_{0\leq j\leq k+2}\sup_{0<t\leq t_0}\tau^{l}\cdot \big|(t\partial_t)^jm(t)\big|\big)
\end{align*}
Moreover, the solution $h_{10}$ depends continuously on $m$ in the following precise sense: writing 
\[
h_{10} = h_{10}(t,r;m),
\]
we have the estimate 
\begin{align*}
\big\|S^l\triangle h_{10}(t,r;m,n)\big\|_{H^1_{r\,dr}}\lesssim_{k}\tau^{-2-p+}\cdot\big(\Big\|m\Big\|_{p,l+2} +\Big\|n\Big\|_{p,l+2}+1\big)^{l}\cdot\Big\|n\Big\|_{p,l+2},\,p\geq \frac12. 
\end{align*}
This is a consequence of Lemma~\ref{lem:ezerobound}.
\\ 

{\it{(ii): Choice of the parameter in $m_1 = m_1(t)$ in \eqref{eq:honeone}.}} We shall solve \eqref{eq:honeone} by the variation of constants formula. Letting $S = \lambda_1(t)s$, we set\footnote{For the factor $12\pi$ recall \eqref{eq:identities1}.} 
\begin{equation}\label{eq:m1definition}
12\pi m_1(t) = \lambda_1^2\cdot 8\cdot \int_0^\infty \frac{4\big[\cos(2Q_1 - 2\tilde{Q}_2\big) - \cos\big(2\tilde{Q}_2\big)\big]}{S^2}h_{10}(t, s)\cdot S\Phi(S)\,dS. 
\end{equation}
Here we let $Q_1 = Q(S) = Q(\lambda_1(t)s)$. Then we have the following 
\begin{lem}\label{lem:m1properties} The function $m_1$ satisfies the bounds
\begin{align*}
\big||(t\partial_t)^l m_1(t)\big|\lesssim_{l}\tau^{-1+\epsilon}\cdot \big(1 + \Big\|m\Big\|_{p,l+4}\big)^l,\,p\geq \frac12. 
\end{align*}
 More generally, interpreting $m_1 = m_1(t;m)$, we have the estimate 
 \begin{align*}
 &\big|(t\partial_t)^l\triangle m_1(t ;m,n)\big|\lesssim_{l} \tau^{-1-p+\epsilon}\cdot\big(1+\Big\|m\Big\|_{p_0, l+4}+\Big\|n\Big\|_{p, l+4}\big)^l\cdot\Big\|n\Big\|_{p, l+4}.
 \end{align*}
 where as before $p\geq \frac12$. 
\end{lem}
\begin{proof}(Lemma~\ref{lem:m1properties}) We first consider the simplest case $l = 0$ of the first estimate. Then we use that for $t_0$ sufficiently small
\begin{align*}
\Big|\frac{4\big[\cos(2Q_1 - 2\tilde{Q}_2\big) - \cos\big(2\tilde{Q}_2\big)\big]}{S^2}\cdot s^2\Big|\lesssim \langle S\rangle^{-2}\cdot \lambda_1^{-2} \leq \langle S\rangle^{-2}\cdot \tau^{-2}. 
\end{align*}
 Then we use the Cauchy-Schwarz inequality to bound the integral defining $m_1$, resulting in 
 \begin{align*}
 \big|\lambda_1^{-2}m_1(t)\big|&\lesssim \big\| S\Phi(S)\cdot \frac{4\big[\cos(2Q_1 - 2\tilde{Q}_2\big) - \cos\big(2\tilde{Q}_2\big)\big]}{S^2}\cdot s^2\big\|_{L^2_{S\,dS}}\cdot \big\|\frac{h_{10}}{s^2}\big\|_{L^2_{S\,dS}}\\
 &\lesssim \tau^{-2}\cdot  \big\|\frac{h_{10}}{s^2}\big\|_{L^2_{S\,dS}}. 
 \end{align*}
 Further we have the bound 
 \[
  \big\|\frac{h_{10}}{s^2}\big\|_{L^2_{S\,dS}} = \lambda_1\cdot \big\|\frac{h_{10}}{s^2}\big\|_{L^2_{s\,ds}}\lesssim \lambda_1\cdot \lambda_2^2\cdot \tau^{-2+}\cdot \big(1 + \Big\|m\Big\|_{p,4}\big),\,p\geq \frac12.
 \]
 This is indeed a consequence of Lemma~\ref{lem:outerpropagatorinhombounds},
 together with the bounds on $h_{10}$ from {\it{(i)}}. 
 It follows that 
 \begin{align*}
 \big|\lambda_1^{-2}m_1(t)\big|\lesssim \tau^{-3+}\cdot \big(1 + \Big\|m\Big\|_{p, 4}\big),\,p\geq \frac12.
 \end{align*}
 The differencing bound is obtained similarly, taking advantage of the estimate 
 \begin{align*}
 &\big|\triangle \Big(\frac{4\big[\cos(2Q_1 - 2\tilde{Q}_2\big) - \cos\big(2\tilde{Q}_2\big)\big]}{R^2}\cdot R\Phi(R)\cdot r^2\Big)\big|(t,r;m,n)\\&\lesssim \langle R\rangle^{-2}\tau^{-2-p+}\cdot \Big\|n\Big\|_{p,0}. 
 \end{align*}
 This estimate is a consequence of Lemma~\ref{lem:technicalsymbolfunctionsdifferencing}. \\
 The differentiated bound follows by re-writing the integral with respect to the variable $s = \frac{S}{\lambda_1(t)}$ and re-writing things in terms of $t\partial_t + r\partial_r$, and further performing integration by parts with respect to $s$.Thus we write 
 \begin{align*}
&(t\partial_t)\Big(  \frac{4\big[\cos(2Q_1 - 2\tilde{Q}_2\big) - \cos\big(2\tilde{Q}_2\big)\big]}{S^2}h_{10}(t, s)\Big)\\
&= (t\partial_t + s\partial_s)\Big(  \frac{4\big[\cos(2Q_1 - 2\tilde{Q}_2\big) - \cos\big(2\tilde{Q}_2\big)\big]}{S^2}h_{10}(t, s)\Big)\\
& - (s\partial_s)\Big(  \frac{4\big[\cos(2Q_1 - 2\tilde{Q}_2\big) - \cos\big(2\tilde{Q}_2\big)\big]}{S^2}h_{10}(t, s)\Big)
 \end{align*}
  and handle the contribution of the second term on the right by integration by parts with respect to $s$. The contribution of the first term on the right is handled by observing that 
  \begin{align*}
  \Big|(t\partial_t + s\partial_s)\Big(  \frac{4\big[\cos(2Q_1 - 2\tilde{Q}_2\big) - \cos\big(2\tilde{Q}_2\big)\big]}{S^2}\Big)\cdot s^2\Big|\lesssim S^{-2}\cdot \tau^{-2}, 
  \end{align*}
 and also taking advantage of the differentiated bounds for $h_{10}$. 
\\
Higher derivatives, as well as the differencing bounds with respect to $m$, are handled analogously. 
\end{proof}

{\it{(iii): Bounds for the solution $h_{11}$ of \eqref{eq:honeone}.}} We accomplish this via 
\begin{lem}\label{lem:honeonebounds} There is a solution of \eqref{eq:honeone} satisfying the bounds 
\begin{align*}
&\big\| S^l\triangle h_{11}(t,r;m; n)\big\|_{H^1_{r\,dr}}\\&\hspace{2cm}\lesssim_{l} \tau^{-3-p+\epsilon}\cdot\big(1+\Big\|m\Big\|_{p, l+4}+\Big\|n\Big\|_{p, l+4}\big)^l\cdot\Big\|n\Big\|_{p, l+4}.
\end{align*}
Here $k, l\geq 0$ and as usual $p\geq \frac12$.
\end{lem}
\begin{proof}(Lemma~\ref{lem:honeonebounds})
Calling the right hand side in  \eqref{eq:honeone} $G$, we have to estimate 
\begin{align*}
\frac14\Theta(R)\cdot\int_0^R \lambda_1^{-2}G(t,S)\cdot S\Phi(S)\,dS - \frac14\Phi(R)\cdot\int_0^R \lambda_1^{-2}G(t,S)\cdot S\Theta(S)\,dS.
\end{align*}
Due to our choice of $m_1$, the first term can also be re-written as 
\begin{align*}
&\chi_{R\lesssim 1}\cdot \frac14\Theta(R)\cdot\int_0^R \lambda_1^{-2}G(t,S)\cdot S\Phi(S)\,dS\\&\hspace{3cm} - \chi_{R\gtrsim 1}\cdot \frac14\Theta(R)\cdot\int_R^\infty \lambda_1^{-2}G(t,S)\cdot S\Phi(S)\,dS.
\end{align*}
As for the estimate of $m_1$, we observe the bounds ($p\geq \frac12$)
\begin{align*}
\Big\|R\Phi(R))\frac{4\big[\cos(2Q_1 - 2\tilde{Q}_2\big) - \cos\big(2\tilde{Q}_2\big)\big]}{R^2}h_{10}\Big\|_{L^1_{R\,dR}}\lesssim_p \tau^{-3+\epsilon}\cdot \big(1+\Big\|m\Big\|_{p,2}\big), 
\end{align*}
and similarly if we replace $\Phi$ by $\Theta$. 
\\
It is then straightforward to verify that with $G$ representing the first term on the right in \eqref{eq:honeone}, we have 
\begin{align*}
&\Big\|\chi_{R\lesssim 1}\cdot \frac14\Theta(R)\cdot\int_0^R \lambda_1^{-2}G(t,S)\cdot S\Phi(S)\,dS\Big\|_{H^1_{r\,dr}}\lesssim _p \tau^{-3+\epsilon}\cdot \big(1+\Big\|m\Big\|_{p,2}\big),\\
&\Big\| \chi_{R\gtrsim 1}\cdot \frac14\Theta(R)\cdot\int_R^\infty \lambda_1^{-2}G(t,S)\cdot S\Phi(S)\,dS\Big\|_{H^1_{r\,dr}}\lesssim _p \tau^{-3+\epsilon}\cdot \big(1+\Big\|m\Big\|_{p,2}\big),\\
&\Big\| \frac14\Phi(R)\cdot\int_0^R \lambda_1^{-2}G(t,S)\cdot S\Theta(S)\,dS\Big\|_{H^1_{r\,dr}}\lesssim _p \tau^{-3+\epsilon}\cdot \big(1+\Big\|m\Big\|_{p,2}\big).
\end{align*}
The analogous terms when $G$ is replaced by $m_1\cdot \big[\cos\big(2Q_1\big) - 1\big]$ are handled similarly, taking advantage of Lemma~\ref{lem:m1properties}. The differentiated bounds follow as usual by integration by parts within the integral and the bounds established in {\it{(i)}}. 
The differencing bounds follow from the corresponding bounds for $m_1, h_{10}$. 
\end{proof}
{\it{(iv): Estimate of the error term $e_{10}$ from \eqref{eq:eonezero}}}. Finally, we estimate the remaining error term via the following 
\begin{lem}\label{lem:e10errorestimate} We have the estimates 
\begin{equation}\label{eq:eonezerobounds}\begin{split}
&\lambda_2^{-1}(t)\cdot \big\| S^l\triangle e_{10}(t,r;m; n)\big\|_{L^2_{r\,dr}}\\&\hspace{2cm}\lesssim_{l} \tau^{-3-p+\epsilon}\cdot \big(1+\Big\|m\Big\|_{p, l+6} + \Big\|n\Big\|_{p, l+6}\big)^l\cdot \Big\|n\Big\|_{p, l+6}.
\end{split}\end{equation}
\end{lem}
\begin{proof}(Lemma~\ref{lem:e10errorestimate}) This follows by converting $\partial_{tt}$ into $\frac{1}{t}(S- r\partial_r)\big(\frac{1}{t}(S- r\partial_r)\big)$ when this operator falls on the integrand in the variation of constants formula and getting rid of the $r\partial_r$ via integration by parts. 
\end{proof}

The proof of Lemma~\ref{lem:honeequation} is completed by combining {\it{(i)}} - {\it{(iv)}}. 
\end{proof}

Let us summarise what we have accomplished thus far via steps 1 and 2: The approximation 
\begin{equation}\label{eq:steponeandtwoapproximation}
u_1(t, r): = Q_1 - \tilde{Q}_2 + v_1,\,v_1 = h_0 + h_1
\end{equation}
solves the following equation 
\begin{equation}\label{eq:uoneequationwitherror}\begin{split}
&{-}u_{1, tt} + u_{1, rr} + \frac{1}{r}u_{1,r} - 2\frac{\sin (2u_1)}{r^2} = e_1,\\
&e_1 = e_{11} + e_{10} + \sum_{j=1}^3 E_{1j}(h_0, h_1)\\
&e_{11} = \boxed{8(2m\lambda_2 + m^2)\cdot \big[\cos(2Q_1) - 1\big]} + \boxed{m_1\cdot \big[\cos\big(2Q_1\big) - 1\big]},
\end{split}\end{equation}
where the terms $E_{1j}$ are given by the following explicit expressions:
\begin{equation}\label{eq:Eonejerrorterms}\begin{split}
&E_{11}(h_0, h_1) = \frac{2\cos\big(2Q_1 - 2\tilde{Q}_2 + 2h_0\big)}{r^2}\cdot\big(\sin\big(2h_1\big) - 2h_1\big),\\
&E_{12}(h_0, h_1) =  2h_1\cdot \Big[ \frac{2\cos\big(2Q_1 - 2\tilde{Q}_2 + 2h_0\big)}{r^2} - \frac{\cos\big(2Q_1 - 2\tilde{Q}_2\big)}{r^2}\Big]\\
&E_{13}(h_0, h_1) =  \frac{2\sin\big(2Q_1 - 2\tilde{Q}_2 + 2h_0\big)}{r^2}\cdot\big(\cos(2h_1) - 1\big)
\end{split}\end{equation}

The terms in $e_{11}$ are boxed in order to emphasize that they are correction terms, and we carefully note the bounds for $m_1$ in Lemma~\ref{lem:m1properties}.
\\
Then we can formulate the following error bound 
\begin{lem}\label{lem:eoneerrorbound} We have the error bound 
\begin{align*}
&\lambda_2^{-1}\cdot\big\| S^l\triangle\big(e_1 - e_{11}\big)(t,r;m; n)\big\|_{L^2_{r\,dr}}\\&\hspace{2cm}\lesssim_{l} \tau^{-3-p+\epsilon}\cdot \big(1+\Big\|m\Big\|_{p, l+6}+\Big\|n\Big\|_{p, l+6}\big)^l\cdot \Big\|n\Big\|_{p, l+6}.
\end{align*}
\end{lem}
\begin{proof}(Lemma~\ref{lem:eoneerrorbound}) For the contribution of $e_{10}$ this follows from the preceding lemma. For the nonlinear terms $E_{1j}(h_0, h_1)$ this follows by combining the previously established bounds with product estimates. Precisely, we have 
\begin{align*}
&\Big\|  S^l\triangle\Big(2\cos\big(2Q_1 - 2\tilde{Q}_2 + 2h_0\big)\Big)(t,r;m; n)\Big\|_{L^2_{r\,dr}}\lesssim \tau^{-p+}\cdot\big(1 + \Big\|m\Big\|_{p,l} + \Big\|n\Big\|_{p,l}\big)^l\cdot  \Big\|n\Big\|_{p,l},\\
&\Big\|  S^l\Big(2\cos\big(2Q_1 - 2\tilde{Q}_2 + 2h_0\big)\Big)(t,r;m)\Big\|_{L^2_{r\,dr}}\lesssim \tau^{0+}\cdot\big(1 + \Big\|m\Big\|_{p,l}\big)^l,
\end{align*}
as follows from Lemma~\ref{lem:technicalsymbolfunctionsdifferencing}, \eqref{eq:nzerodifferencingbounds}, as well as repeated use of the Leibniz rule as well as \eqref{eq:Leibnizdifferencing}. Similar bounds apply when replacing $\cos$ by $\sin$. 
\\
We can further estimate 
\begin{align*}
&\Big\|S^l\big(\frac{\cos(2h_1) - 1}{r^2}\big)\Big\|_{L^2_{r\,dr}}\\&\lesssim_l \big(\sum_{l_1=0}^{l} \big\|\frac{S^{l_1}h_1}{r^2}\big\|_{L^2_{r\,dr}}\big)\cdot \big(\sum_{l_2=0}^l \big\|S^{l_2}h_1\big\|_{L^\infty_{r\,dr}}\big)\cdot \big(1+\sum_{l_3=0}^\infty \big\|S^{l_3}h_1\big\|_{L^\infty_{r\,dr}}\big)^{l-2}.
\end{align*}
The preceding expression on the right can be bounded by 
\begin{align*}
\tau^{-4+}\cdot \big(1+\Big\|m\Big\|_{p,l+6}\big)^{l^2}
\end{align*}
We similarly infer the differencing bound 
\begin{align*}
&\Big\|S^l\triangle\big(\frac{\cos(2h_1) - 1}{r^2}\big)(t,r;m,n)\Big\|_{L^2_{r\,dr}}\lesssim \tau^{-4-p+}\cdot \big(1+\Big\|m\Big\|_{p,l+6} + \Big\|n\Big\|_{p,l+6}\big)^{l^2}\cdot\Big\|n\Big\|_{p,l+6}.
\end{align*}
Combining the preceding estimates and taking advantage of the Leibniz rule and \eqref{eq:Leibnizdifferencing}, we derive the desired bounds for the term $E_{13}(h_0, h_1)$. The other terms $E_{1j}(h_0, h_1),\,j = 1, 2$, are handled analogously. 
\end{proof}

{\bf{Step 2}}: {\it{Inductive construction of the correction $h_{j+1},\,j\geq 1$.}} Assume that we have constructed the corrections $h_k,\,k = 0,\ldots, j$, such that setting $v_j: = \sum_{k=0}^j h_k$, the approximate solution 
\[
u_j: = Q_1 - \tilde{Q}_2 + v_j
\]
satisfies the equation 
\begin{equation}
\begin{split}
&{-}u_{j,tt} + u_{j, rr} + \frac{1}{r}u_{j,r} - 2\frac{\sin (2u_j)}{r^2} = e_j,\\
&e_j = e_{j1} + e_{j2}\\
&e_{j1} = \boxed{8(2m\lambda_2 + m^2)\cdot \big[\cos(2Q_1) - 1\big]} + \boxed{\big(\sum_{k=1}^jm_k\big)\cdot \big[\cos\big(2Q_1\big) - 1\big]},
\end{split}\end{equation}
and we assume the following estimates(all for $j\geq k\geq 1$): 
\begin{equation}\label{eq:hjbounds}\begin{split}
&\big\|S^lh_k(t,\cdot;m)\big\|_{H^1_{r\,dr}}\lesssim_{k,l}\tau^{-1-k+\epsilon}\cdot \big(1+\Big\|m\Big\|_{p,l+4k+2}\big)^{C_{k,l}}\\
&\big\|S^l\triangle h_k(t,\cdot;m,n)\big\|_{H^1_{r\,dr}}\lesssim_{k,l}\tau^{-1-p-k+\epsilon}\cdot \big(1+\Big\|m\Big\|_{p,l+4k+2}\big)^{C_{k,l}}\cdot \Big\|n\Big\|_{p,l+4k+2},\\
\end{split}\end{equation}
for the increments, and further 
\begin{equation}\label{eq:ejtwomkbounds}\begin{split}
&\lambda_2^{-1}(t)\cdot\big\|S^l e_{j2}(t,\cdot;m))\big\|_{L^2_{r\,dr}}\lesssim_{j,l}\tau^{-2-j+\epsilon}\cdot \big(1+\Big\|m\Big\|_{p,l+4j+2}\big)^{C_{j,l}}\\
&\lambda_2^{-1}(t)\cdot\big\|S^l \triangle e_{j2}(t,\cdot;m,n))\big\|_{L^2_{r\,dr}}\lesssim_{j,l}\tau^{-2-j-p+\epsilon}\cdot \big(1+\Big\|m\Big\|_{p,l+4j+2}\big)^{C_{j,l}}\cdot \Big\|n\Big\|_{p,l+4j+2},\\
&\big|(t\partial_t)^lm_k(t;m)\big|\lesssim_{k,l}\tau^{-k+\epsilon}\cdot \big(1+\Big\|m\Big\|_{p,l+4k}\big)^{C_{k,l}},\,1\leq k\leq j,\\
&\big|(t\partial_t)^l\triangle m_k(t;m,n)\big|\lesssim_{k,l}\tau^{-k-p+\epsilon}\cdot \big(1+\Big\|m\Big\|_{p,l+4k}\big)^{C_{k,l}}\cdot \Big\|n\Big\|_{p,l+4k}\,1\leq k\leq j.\\
\end{split}\end{equation}

Then we have the 
\begin{lem}\label{lem:hjplusoneconstruction} Under the previous assumptions, there exist corrections $h_{j+1}, m_{j+1}$ satisfying the bounds 
\begin{align*}
&\big\|S^lh_{j+1}(t,\cdot;m)\big\|_{H^1_{r\,dr}}\lesssim_{j,l}\tau^{-2-j+\epsilon}\cdot \big(1+\Big\|m\Big\|_{p,l+4(j+1)+2}\big)^{C_{j+1,l}}\\
&\big\|S^l\triangle h_{j+1}(t,\cdot;m,n)\big\|_{H^1_{r\,dr}}\lesssim_{j,l}\tau^{-2-j-p+\epsilon}\cdot \big(1+\Big\|m\Big\|_{p,l+4(j+1)+2}\big)^{C_{j+1,l}}\cdot \Big\|n\Big\|_{p,l+4(j+1)+2},\\
\end{align*}
as well as 
\begin{align*}
&\big|(t\partial_t)^lm_{j+1}(t;m)\big|\lesssim_{j,l}\tau^{-j-1+\epsilon}\cdot \big(1+\Big\|m\Big\|_{p,l+4j+4}\big)^{C_{j+1,l}},\\
&\big|(t\partial_t)^l\triangle m_{j+1}(t;m,n)\big|\lesssim_{j,l}\tau^{-j-1-p+\epsilon}\cdot \big(1+\Big\|m\Big\|_{p,l+4j+4}\big)^{C_{j+1,l}}\cdot \Big\|n\Big\|_{p,l+4j+4},
\end{align*}
for a suitable constant $C_{j+1,l}$, and further the new approximation 
\[
u_{j+1}: = Q_1 - \tilde{Q}_2 + v_{j+1},\,v_{j+1} = \sum_{k=0}^{j+1}h_k
\]
satisfies the equation
\begin{equation}\label{eq:ujequation}\begin{split}
&{-}u_{j+1,tt} + u_{j+1, rr} + \frac{1}{r}u_{j+1,r} - 2\frac{\sin (2u_{j+1})}{r^2} = e_{j+1},\\
&e_{j+1} = e_{j+1,1} + e_{j+1,2}\\
&e_{j+1,1} = \boxed{8(2m\lambda_2 + m^2)\cdot \big[\cos(2Q_1) - 1\big]} + \boxed{\big(\sum_{k=1}^{j+1}m_k\big)\cdot \big[\cos\big(2Q_1\big) - 1\big]}.
\end{split}\end{equation}
The error term $e_{j+1,2}$ satisfies bounds analogous to those of $e_{j2}$ stated above, but with $j$ replaced by $j+1$. 
\end{lem}
\begin{proof}(Lemma~\ref{lem:hjplusoneconstruction}) This relies on a close analogue of Lemma~\ref{lem:honeequation}, involving a source term with faster decay. The following serves as the equation for $h_{j+1}$ which we {\it{strive to solve}}, although only approximately:
\begin{equation}\label{eq:idealhjplusoneequation}
-h_{j+1,tt} + h_{j+1,rr} + \frac{1}{r}h_{j+1,r} - \frac{4\cos\big(2Q_1 - 2\tilde{Q}_2\big)}{r^2}h_{j+1} = e_{j2} + \boxed{m_{j+1}\cdot \big[\cos\big(2Q_1\big) - 1\big]}.
\end{equation} 
We formulate the following lemma asserting the existence of a suitable approximate solution, as follows 
\begin{lem}\label{lem:solutionofhj+1equation} Assuming the bounds \eqref{eq:ejtwomkbounds}, there is a function $m_{j+1}$ satisfying the bounds stated in Lemma~\ref{lem:hjplusoneconstruction} such that \eqref{eq:idealhjplusoneequation} admits an approximate solution $h_{j+1}$ satisfying the bounds stated in 
Lemma~\ref{lem:hjplusoneconstruction}. Calling $\tilde{\Box}h_{j+1}$ the expression on the left of \eqref{eq:idealhjplusoneequation}, and $E_j$ the expression on the right, and further setting 
\[
e_{j+1,0} := \tilde{\Box}h_{j+1} - E_j,
\]
we have the following bounds for suitable constant $C_{j+1,l}$:
\begin{align*}
&\lambda_2^{-1}(t)\cdot \big\|S^le_{j+1,0}(t,\cdot;m)\big\|_{L^2_{r\,dr}}\lesssim_l  \tau^{-3-j+\epsilon}\cdot \big(1+\Big\|m\Big\|_{p,l+4(j+1)+2}\big)^{C_{j+1,l}},\\
&\lambda_2^{-1}(t)\cdot \big\|S^l\triangle e_{j+1,0}(t,\cdot;m,n)\big\|_{L^2_{r\,dr}}\lesssim_l  \tau^{-3-j-p+\epsilon}\cdot \big(1+\Big\|m\Big\|_{p,l+4(j+1)+2}\big)^{C_{j+1,l}}\cdot \Big\|n\Big\|_{p,l+4(j+1)+2}
\end{align*}
\end{lem}
\begin{proof}(Lemma~\ref{lem:solutionofhj+1equation}) This follows exactly along the lines of the proof of Lemma~\ref{lem:honeequation}. We set 
\[
h_{j+1} = h_{j+1,0} + h_{j+1,1},
\]
where $h_{j+1,0}$ solves the {\it{outer wave equation}} 
\begin{equation}\label{eq:hjplusonezero}
-h_{j+1,0;tt} + h_{j+1,0;rr} + \frac{1}{r}h_{j+1,0;r} - \frac{4\cos\big(2\tilde{Q}_2\big)}{r^2}h_{j+1,0} = e_{j2},
\end{equation}
while the second part $h_{j+1,1}$ solves the following 
\begin{equation}\label{eq:hjplusoneone}\begin{split}
h_{j+1,1;rr} + \frac{1}{r}h_{j+1,1;r} - \frac{4\cos\big(2Q_1\big)}{r^2}h_{j+1,1} & =  -\frac{4\big[\cos(2Q_1 - 2\tilde{Q}_2\big) - \cos\big(2\tilde{Q}_2\big)\big]}{r^2}h_{j+1,0}\\
&\hspace{1cm} + \boxed{m_{j+1}\cdot \big[\cos\big(2Q_1\big) - 1\big]}.
\end{split}\end{equation}
The solution of the second equation will rely on suitable choice of $m_{j+1}$. 
\\

{\it{(i): Solution of \eqref{eq:hjplusonezero}.}} Taking advantage of Lemma~\ref{lem:outerpropagatorinhombounds} as well as \eqref{eq:ejtwomkbounds}, we infer that \eqref{eq:hjplusonezero} admits a solution satisfying the bounds 
\begin{equation}\label{eq:hjplusonezerobound}\begin{split}
&\big\|S^l h_{j+1,0}(t,\cdot;m))\big\|_{H^1_{r\,dr}}\lesssim_{j,l}\tau^{-2-j+\epsilon}\cdot \big(1+\Big\|m\Big\|_{p,l+4j+2}\big)^{C_{j,l}}\\
&\big\|S^l \triangle h_{j+1,0}(t,\cdot;m,n))\big\|_{H^1_{r\,dr}}\lesssim_{j,l}\tau^{-2-j-p+\epsilon}\cdot \big(1+\Big\|m\Big\|_{p,l+4j+2}\big)^{C_{j,l}}\cdot \Big\|n\Big\|_{p, l+4j+2}.\\
\end{split}\end{equation}

{\it{(ii): Choice of $m_{j+1}$ in \eqref{eq:hjplusoneone}.}} We shall set 
\begin{equation}\label{eq:mjplusoneformula}
12\pi m_{j+1}(t) = \lambda_1^2\cdot8\cdot \int_0^\infty \frac{4\big[\cos(2Q_1 - 2\tilde{Q}_2\big) - \cos\big(2\tilde{Q}_2\big)\big]}{S^2}h_{j+1,0}\cdot S\Phi(S)\,ds. 
\end{equation}
Then in exact analogy to step {\it{(ii)}} in the proof of Lemma~\ref{lem:honeequation}, we deduce the bounds 
\begin{equation}\label{eq:mjplusonebounds}\begin{split}
&\big|(t\partial_t)^lm_{j+1}(t;m)\big|\lesssim_{j,l}\tau^{-j-1+\epsilon}\cdot \big(1+\Big\|m\Big\|_{p, l + 4j + 4}\big)^{C_{j,l+2}},\\
&\big|(t\partial_t)^l\triangle m_{j+1}(t;m,n)\big|\lesssim_{j,l}\tau^{-j-1-p+\epsilon}\cdot \big(1+\Big\|m\Big\|_{p, l + 4j + 4}\big)^{C_{j,l+2}}\cdot \Big\|n\Big\|_{p,l+4j+4}. 
\end{split}\end{equation}

{\it{(iii): Solution of \eqref{eq:hjplusoneone}.}} Using the same argument as in {\it{(iii)}} in the proof of Lemma~\ref{lem:honeequation}, we deduce the existence of a solution $h_{j+1,1}$ satisfying the bounds (also recall the bounds at the beginning of the proof of Lemma~\ref{lem:eoneerrorbound})
\begin{align*}
&\big\|S^l h_{j+1,1}(t,\cdot;m))\big\|_{H^1_{r\,dr}}\lesssim_{j,l}\tau^{-3-j+\epsilon}\cdot  \big(1+\Big\|m\Big\|_{p, l + 4j + 4}\big)^{C_{j,l+2}+l},\\
&\big\|S^l \triangle h_{j+1,1}(t,\cdot;m,n))\big\|_{H^1_{r\,dr}}\lesssim_{j,l}\tau^{-3-j-p+\epsilon}\cdot  \big(1+\Big\|m\Big\|_{p, l + 4j + 4}\big)^{C_{j,l+2}+l}\cdot \Big\|n\Big\|_{p,l+4j+4}.
\end{align*}
It follows that 
\[
h_{j+1} =   h_{j+1,0} + h_{j+1,1}
\]
satisfies the desired bounds, provided we arrange that 
\[
C_{j+1,l}\geq C_{j, l+2}+l. 
\]
{\it{(iv): Estimate of the error term $e_{j+1,0}$.}} This is given by 
\begin{align*}
e_{j+1,0} = -\partial_{tt} h_{j+1,1} + \frac{4\big[\cos\big(2Q_1 - 2\tilde{Q}_2\big) - \cos \big(2Q_1\big)\big]}{r^2}\cdot h_{j+1,1}. 
\end{align*}
Then the bounds stated in Lemma~\ref{lem:solutionofhj+1equation} follow from those for $h_{j+1,1}$, in analogy to the bounds for $e_{10}$. 
\end{proof}

In order to complete the proof of Lemma~\ref{lem:hjplusoneconstruction}, it suffices to estimate the error generated by the approximation 
\[
u_{j+1} =  Q_1 - \tilde{Q}_2 + v_{j+1},\,v_{j+1} = \sum_{k=0}^{j+1}h_k. 
\]
We can write this error in the form 
\[
e_{j+1} = e_{j+1,1} + e_{j+1,2},
\]
where $e_{j+1,1}$ is given in \eqref{eq:ujequation}, while 
\[
e_{j+1,2} = e_{j+1,0} + \sum_{k=1}^3 E_{j+1,k}
\]
where the terms $E_{j+1,k},\,k = 1, 2, 3$ are as in \eqref{eq:Eonejerrorterms} but with $h_0$ replaced by $\sum_{k=0}^j h_k$, and $h_1$ replaced by $h_{j+1}$.We now estimate each of these terms, which will suffice to estimate $e_{j+1,2}$ in light of the preceding lemma.
\\

{\it{The estimate for $E_{j+1,1}$.}} First using the bounds \eqref{eq:hjbounds} as well as Taylor expansion, the Leibniz rule as well as \eqref{eq:Leibnizdifferencing}, we infer 
\begin{align*}
&\big\|S^l\cos\big(2Q_1 - 2\tilde{Q}_2 + \sum_{k=0}^j h_k\big)\big\|_{H^1_{r\,dr}}\lesssim_{l,j}\tau^{0+}\cdot\big(1+\Big\|m\Big\|_{p,l+4j+2}\big)^{C_{j,l}},\\
&\big\|S^l\triangle\cos\big(2Q_1 - 2\tilde{Q}_2 + \sum_{k=0}^j h_k\big)(t,r;m,n)\big\|_{H^1_{r\,dr}}\\&\hspace{4cm}\lesssim_{l,j}\tau^{-p+}\cdot\big(1+\Big\|m\Big\|_{p,l+4j+2}\big)^{C_{j,l}}\cdot \Big\|n\Big\|_{p,l+4j+2}. 
\end{align*}
Furthermore, we infer the bounds 
\begin{align*}
&\big\|S^l\Big(\frac{\sin(2h_{j+1}) - 2h_{j+1}}{r^2}\Big)\big\|_{L^2_{r\,dr}}\\&\hspace{2cm}\lesssim_l\sum_{l_1+l_2\leq l}\big\|\frac{S^{l_1}h_{j+1}}{r^2}\big\|_{L^2_{r\,dr}}\cdot \big\|S^{l_2}h_{j+1}\big\|_{L^\infty_{r\,dr}}\cdot\big(1+\sum_{r=0}^l\Big\|S^r h_{j+1}\Big\|_{L^\infty_{r\,dr}}\big)^{l-2},
\end{align*}
using Taylor expansion, the Leibniz rule and Holder inequality. Taking advantage of the already established bounds for $h_{j+1}$, we can estimate the preceding expression by (with $l_1 + l_2\leq l$)
\begin{align*}
&\big\|\frac{S^{l_1}h_{j+1}}{r^2}\big\|_{L^2_{r\,dr}}\cdot \big\|S^{l_2}h_{j+1}\big\|_{L^\infty_{r\,dr}}\cdot\big(1+\sum_{r=0}^l\Big\|S^r h_{j+1}\Big\|_{L^\infty_{r\,dr}}\big)^{l-2}\\
&\lesssim_{l,j}\tau^{-j-1+}\cdot \tau^{-j-2+}\cdot \big(1+\Big\|m\Big\|_{p,l+4j+4}\big)^{D_{l,j}}
\end{align*}
and we set 
\[
D_{l,j} = l\cdot\big(C_{j,l+2} + l\big). 
\]
Again using the Leibniz rule, we then infer that 
\begin{align*}
\lambda_2^{-1}\cdot\big\|S^lE_{j+1,1}\big\|_{L^2_{r\,dr}}\lesssim_{j,l}\tau^{-3-2j+}\cdot  \big(1+\Big\|m\Big\|_{p,l+4j+4}\big)^{D_{l,j}+C_{j,l}}. 
\end{align*}
Using similar estimates in conjunction with  \eqref{eq:Leibnizdifferencing}, we deduce the differencing bound
\begin{align*}
\lambda_2^{-1}\cdot\big\|S^l\triangle E_{j+1,1}(t,\cdot;m,n)\big\|_{L^2_{r\,dr}}\lesssim_{j,l}\tau^{-3-2j-p+}\cdot  \big(1+\Big\|m\Big\|_{p,l+4j+4}\big)^{D_{l,j}+C_{j,l}}\cdot \Big\|n\Big\|_{p,l+4j+4}.
\end{align*}
The desired bounds for this contribution to $e_{j+1,2}$ follow if we ensure that $C_{j+1,l}\geq D_{l,j}+C_{j,l}$.
\\

For the term $E_{j+2,2}$, we take advantage of the estimates
\begin{align*}
&\big\|S^l\big[\frac{2\cos\big(2Q_1 - 2\tilde{Q}_2 + \sum_{k=0}^j h_k\big)}{r^2} - \frac{2\cos\big(2Q_1 - 2\tilde{Q}_2 \big)}{r^2}\big]\big\|_{L^2_{r\,dr}}\leq F_1 + F_2,
\end{align*}
where 
\begin{align*}
&F_1\lesssim_l \sum_{l_1+l_2=l}\big\|S^{l_1}\big(\cos\big(2Q_1 - 2\tilde{Q}_2\big)\big)\big\|_{L^\infty_{r\,dr}}\cdot \big\|S^{l_2}\big(\frac{\cos\big(\sum_{k=0}^j h_k\big) - 1}{r^2}\big)\big\|_{L^2_{r\,dr}}\\
&F_2\lesssim \sum_{l_1+l_2=l}\big\|S^{l_1}\big(\sin\big(2Q_1 - 2\tilde{Q}_2\big)\big)\big\|_{L^\infty_{r\,dr}}\cdot \big\|S^{l_2}\big(\frac{\sin\big(\sum_{k=0}^j h_k\big)}{r^2}\big)\big\|_{L^2_{r\,dr}}.\\
\end{align*}
In analogy to the estimate for $E_{j+2,1}$, and recalling Lemma~\ref{lem:hzerocorrection} as well as the bounds \eqref{eq:hjbounds}, we have the bounds (here $l_2\leq l$)
\begin{align*}
&\big\|S^{l_2}\big(\frac{\cos\big(\sum_{k=0}^j h_k\big) - 1}{r^2}\big)\big\|_{L^2_{r\,dr}}\lesssim_{j,l}\tau^{-3+}\cdot \big(1+\Big\|m\Big\|_{p,l+4j+2}\big)^{C_{j,l}},\\
&\big\|S^{l_2}\big(\frac{\sin\big(\sum_{k=0}^j h_k\big)}{r^2}\big)\big\|_{L^2_{r\,dr}}\lesssim_{j,l}\tau^{-1+}\cdot \big(1+\Big\|m\Big\|_{p,l+4j+2}\big)^{C_{j,l}}.
\end{align*}
We then conclude that 
\begin{align*}
F_1 + F_2\lesssim_{j,l}\tau^{-1+}\cdot \big(1+\Big\|m\Big\|_{p,l+4j+2}\big)^{C_{j,l}}.
\end{align*}
In turn we deduce that 
\begin{align*}
\lambda_2^{-1}(t)\cdot \big\|S^l E_{j+2,2}\big\|_{L^2_{r\,dr}}&\lesssim_{l,j}\big(\sum_{0\leq l_1\leq l}\big\|S^{l_1}h_{j+1}\big\|_{L^\infty_{r\,dr}}\big)\cdot\big(F_1 + F_2\big)\\
&\lesssim_{l,j}\tau^{-3-j+}\cdot \big(1+\Big\|m\Big\|_{p,l+4j+2}\big)^{C_{j+1,l}}
\end{align*}
by increasing $C_{j+1,l}$ if necessary. The difference bound for 
\[
S^l \triangle E_{j+2,2}
\]
is more of the same.  
\\

The term $E_{j+2,3}$ can be estimated in analogy to the term $E_{j+1,1}$, using the bound 
\begin{align*}
\big\|S^l\big(\frac{\cos\big(2h_{j+1}\big) - 1}{r^2}\big)\big\|_{L^2_{r\,dr}}\lesssim_{l,j}\tau^{-2j-3+}\cdot \big(1+\Big\|m\Big\|_{p,l+4j+4}\big)^{D_{l,j}}
\end{align*}
where $D_{l,j}$ is defined as in the estimates for $E_{j+1,1}$. From here one obtains the same estimates for $E_{j+1,3}$ as for the term $E_{j+1,1}$. 
\\

The estimates for $E_{j+1,k}$, $k = 1, 2, 3$ furnish the desired bound for $e_{j+1,2}$ completing the proof of Lemma~\ref{lem:hjplusoneconstruction}.
  \end{proof}
  
  {\bf{Step 3}}: {\it{Proof of Proposition~\ref{prop:approxsoln}}}. Repeating application of Lemma~\ref{lem:hjplusoneconstruction} sufficiently many times, we arrive at an approximate solution 
  \[
  u_N = Q_1 - \tilde{Q}_2 + \sum_{j=0}^N h_j,
  \]
  which generates the error 
  \begin{align*}
  -u_{N,tt} + u_{N,rr} + \frac{1}{r}u_{N,r} - 2\frac{\sin\big(2u_N\big)}{r^2} = e_N,
  \end{align*}
  such that we can write 
  \[
  e_N = e_{N,1} + e_{N,2}.
  \]
 Here we have on the one hand the bounds  
 \begin{equation}\label{eq:eNtwobound}\begin{split}
 &\lambda_2^{-1}(t)\cdot \big\|S^l e_{N,2}(t,\cdot;m)\big\|_{L^2_{r\,dr}}\lesssim_{l,N}\tau^{-2-N+}\cdot \big(1 + \big\|m\big\|_{p,l+4N+2}\big)^{C_{N,l}},\\
 &\lambda_2^{-1}(t)\cdot \big\|S^l e_{N,2}(t,\cdot;m,n)\big\|_{L^2_{r\,dr}}\lesssim_{l,N}\tau^{-2-N-p+}\cdot \big(1 + \big\|m\big\|_{p,l+4N+2}\big)^{C_{N,l}}\cdot \big\|n\big\|_{p,l+4N+2}.\\
\end{split}\end{equation}
 On the other hand, we have the representation formula
 \begin{equation}\label{eq:eNoneformula}
 e_{N,1} =  \boxed{8(2m\lambda_2 + m^2)\cdot \big[\cos(2Q_1) - 1\big]} + \boxed{\big(\sum_{k=1}^{N}m_k\big)\cdot \big[\cos\big(2Q_1\big) - 1\big]}
 \end{equation} 
 The $m_k$ in turn depend on $m$ in the way detailed in Lemma~\ref{lem:hjplusoneconstruction} and the preparations leading to it. Then we need the following 
 \begin{lem}\label{lem:mchoice} Given $N\geq 1$, there is $t_0 = t_0(N)>0$ sufficiently small, so that there is a function $m\in C^{\infty}\big((0,t_0]\big)$ with 
 \[
 \big\|m\big\|_{p, 10N}\leq 1
 \]
 for some fixed $p\geq \frac12$, and such that 
 \begin{align*}
 \lambda_2^{-1}\cdot \big\|e_{N,1}\big\|_{L^2_{r\,dr}}\leq \tau^{-N+1+}. 
 \end{align*}
 \end{lem}
 \begin{proof}(Lemma~\ref{lem:mchoice}) We need to solve the equation 
 \begin{equation}\label{eq:mequation1}
 8(2m\lambda_2 + m^2) + \sum_{k=1}^{N}m_k = 0
 \end{equation}
 approximately. We rewrite this equation in the form 
  \begin{equation}\label{eq:mequation2}
 m +  (16\lambda_2)^{-1}\sum_{k=1}^{N}m_k +  (2\lambda_2)^{-1}m^2 = 0
 \end{equation}
 Let us denote 
 \[
 (16\lambda_2)^{-1}\sum_{k=1}^{N}m_k +  (2\lambda_2)^{-1}m^2=: Q_N(t;m).
 \]
 In light of the bounds \eqref{eq:ejtwomkbounds}, we have the estimates (for $p>\frac12$)
 \begin{equation}\label{eq:PNbounds}\begin{split}
&\big|(t\partial_t)^lQ_N(t;m)\big|\lesssim_{l,N}\tau^{-1+}\cdot \big(1+\big\|m\big\|_{p, l+2N+2}\big)^{C_{N,l}},\\
&\big|(t\partial_t)^l \triangle Q_N(t;m,n)\big|\lesssim_{l,N}\tau^{-1-p+}\cdot \big(1+\big\|m\big\|_{p, l+2N+2}\big)^{C_{N,l}}\cdot \big\|n\big\|_{p,l+2N+2}. 
 \end{split}\end{equation}
 Let us now set 
 \[
 Q_N(t;m) - Q_N(t;0) = : -P_N(t;m),\,Q_N(t;0) =: -d(t). 
 \]
 Then we can reformulate \eqref{eq:mequation2} as 
 \begin{equation}\label{eq:mdeqn}
 m = P_N(m) + d,
 \end{equation}
 where we suppress the dependence on the variable $t$. We note right away that on account of \eqref{eq:PNbounds} we have 
 \[
 \big|(t\partial_t)^l d\big|\lesssim_{l,N}\tau^{-1+}. 
 \]
 Setting for simplicity $P_N = P$, define inductively
\[
P_1(d) = P(d),\,P_{k+1}(d) = P_1\big(d + P_k(d)\big),\,k\geq 1. 
\]
Further, set 
\[
m_k: = d + P_k(d) 
\]
Using \eqref{eq:PNbounds} we inductively infer the bounds 
\begin{equation}\label{eq:mkbounds}
\big|(t\partial_t)^l m_k\big|\lesssim_{l,k}\tau^{-1+}. 
\end{equation}
Observe that 
\begin{align*}
P_1(d+f) - P_1(d) = \triangle P(d,f),
\end{align*}
and hence 
\[
P_2(d) = P_1(d) +  \triangle P(d,P(d)). 
\]
From here we deduce 
\begin{align*}
P_3(d) -P_2(d) & = P_1\big(d+P_2(d)\big) - P_1\big(d+P_1(d)\big)\\
&= \triangle P\big(d+P_1(d), \triangle P(d,P(d))\big)\\
& =  \triangle P\big(m_1, \triangle P(d,P(d))\big).
\end{align*}
Inductively we derive the relation 
\begin{equation}\label{eq:Pkdifferencerelation}
P_{k+1}(d) -P_k(d) = \triangle P\big(m_{k-1}, \triangle P\big(m_{k-2},\ldots, \triangle P\big(d, P(d)\big)\ldots\big)\big).
\end{equation}
Observing that $P_N$ satisfies the same bounds \eqref{eq:PNbounds} as $Q_N$, we then infer that
\begin{align*}
\big|m_{N+1} - d - P(m_{N+1})\big| &= \big|P_{N+1}(d) - P\big(d + P_{N+1}(d)\big)\big|\\
& = \big|P_{N+1}(d) - P_{N+2}(d)\big|
\end{align*}
satisfies the bound 
\begin{equation}\label{eq:mNerrorbound}
\big|\big(t\partial_t\big)^l\big(m_{N+1} - d - P(m_{N+1})\big)\big|\lesssim_{N,l} \tau^{-N+}. 
\end{equation}
The desired bound for $e_{N,1}$ is then an immediate consequence. 
 \end{proof}
 
 Picking the function $m$ obtained in the preceding lemma, Proposition~\ref{prop:approxsoln} follows from the bounds \eqref{eq:eNtwobound} with $l = 0$. 
 
 \end{proof}
 
\section{Preparations for the last perturbation step leading to the exact solution}

Let 
\[
\lambda_1(\tau): = e^{\alpha(t)},
\]
where $\alpha(t)$ is given by formula \eqref{eq:alphadefinition}, and $\nu$ given by Lemma~\ref{lem:modulationfinetuning}, and the $m$ used there coming from Lemma~\ref{lem:mchoice}. We deduce in particular 
that for any $\gamma>0$ there exists $t_0 = t_0(\gamma, \beta,N)>0$ such that 
\begin{equation}\label{eq:trutaubounds}
\tau_1: = \int_{t}^1 \lambda_1(s)\,ds\in [t|\log t|^{-\beta}\cdot e^{2(1-\gamma)(\beta+1)^{-1}|\log t|^{\beta+1}},\,t|\log t|^{-\beta}\cdot e^{2(\beta+1)^{-1}|\log t|^{\beta+1}}],
\end{equation}
provided $0<t\leq t_0$. In the sequel we shall have to work with the precise time $\tau_1$, while we used an approximation $\tau$ before. We carefully observe the following asymptotic estimates: first we have 
\[
\alpha(t) = -\int_t^{t_0}\zeta^{-1}(s)\,ds + O(1) = \int_t^{t_0}\frac{2\big|\log s\big|^{\beta}}{s}\cdot\big(1 +O\big(\frac{1}{\big|\log s\big|^{\beta}}\big)\big)\,ds
\]
where the term $O\big(\ldots\big)$ obeys symbol bounds with respect to $s$. It follows that
\begin{equation}\label{eq:alphaasympto}
\alpha(t) =  2(\beta+1)^{-1}\cdot|\log t|^{\beta+1} + O\big(\big|\log t\big|\big). 
\end{equation}
Letting $\lambda_1(t) = e^{\alpha(t)}$, we then deduce that 
\begin{equation}\label{eq:lambdaoneasympto}
\tau_1 = \frac{t}{2}\cdot \big|\log t\big|^{-\beta}\cdot \lambda_1(t)\cdot \big(1 + O\big(\frac{1}{\big|\log t\big|^{\beta}}\big)\big). 
\end{equation}
where the error term again satisfies symbol type bounds. In particular, we have 
\[
\tau_1\sim \frac{\lambda_1}{\lambda_2}.
\]
We also derive the following asymptotic bounds:
\begin{equation}\label{eq:lambdaonederivative}\begin{split}
&\frac{\lambda_{1,\tau_1}}{\lambda_1} = \big|\frac{\lambda_{1,t}}{\lambda_1}\big|\cdot \lambda_1^{-1} = \frac{2\big|\log t\big|^{\beta}}{t}\cdot \lambda_1^{-1}\sim \tau_1^{-1},\\
&\big|\partial_{\tau_1}\big(\frac{\lambda_{1,\tau_1}}{\lambda_1}\big)\big|\sim \tau_1^{-2}. 
\end{split}\end{equation}
Furthermore, we shall later need to control the size of the following integral expression (where we restrict to $\sigma_1\geq \tau_1$)
\begin{equation}\label{eq:xzeroDuhamelpropagator}\begin{split}
&\lambda_1(\tau_1)\cdot\int_{\tau_1}^{\sigma_1}\lambda_1^{-1}(s)\,ds \lesssim \lambda_1(\tau_1)\cdot\int_{\tau_1}^{\sigma_1}s^{-1}\cdot t(s)\cdot \big|\log t(s)\big|^{-\beta}\,ds\lesssim \log\big(\frac{\sigma_1}{\tau_1}\big)\cdot \tau_1,\\
&\big|\partial_{\tau_1}\big(\lambda_1(\tau_1)\cdot\int_{\tau_1}^{\sigma_1}\lambda_1^{-1}(s)\,ds\big)\big|\lesssim \log\big(\frac{\sigma_1}{\tau_1}\big).
\end{split}\end{equation}

We now introduce the function spaces we shall work with, in terms of suitable norms: setting $N_j= \frac{N}{2^j},\,j = 0, 1, 2$, define   
\begin{equation}\label{eq:Xnorm}\begin{split}
\big\|\epsilon\big\|_{X}: &= \sup_{0<t\leq t_0}\tau_1^{N_0}\cdot\big\|\epsilon(t,\cdot)\big\|_{H^1_{r\,dr}} +  \sup_{0<t\leq t_0}\tau_1^{N_1}\cdot\big\|S\epsilon(t,\cdot)\big\|_{H^1_{r\,dr}} +  \sup_{0<t\leq t_0}\tau_1^{N_2}\cdot\big\|S^2\epsilon(t,\cdot)\big\|_{H^1_{r\,dr}}\\
& + \sup_{0<t\leq t_0}\tau_1^{N_2}\cdot\lambda_1^{-1}(t)\big\|\frac{\epsilon}{r^2}\big\|_{L^2_{r\,dr}} +  \sup_{0<t\leq t_0}\tau_1^{N_2}\cdot\lambda_1^{-\frac12}(t)\big\|\frac{S\epsilon}{r}\big\|_{L^4_{r\,dr}}.
\end{split}\end{equation}
Here in analogy to \cite{KST3} we carefully define the norm $\big\|\cdot\big\|_{H^1_{r\,dr}}$ as follows 
\begin{equation}\label{eq:Honenormdefn}
\big\|\epsilon(t,\cdot)\big\|_{H^1_{r\,dr}}: = \big\|L_t^{\frac12}\epsilon\big\|_{L^2_{r\,dr}} + \big\|\epsilon_t\big\|_{L^2_{r\,dr}} + \lambda_1(t)\cdot\big\|\epsilon\big\|_{L^2_{r\,dr}}.
\end{equation}

For the source terms, we use the norm 
\begin{equation}\label{eq:Ynorm}\begin{split}
\big\|f\big\|_{Y}: &=  \sup_{0<t\leq t_0}\tau_1^{N_0+2}\cdot\lambda_1^{-1}(t)\big\|f(t,\cdot)\big\|_{L^2_{r\,dr}} +  \sup_{0<t\leq t_0}\tau_1^{N_1+2}\cdot\lambda_1^{-1}(t)\big\|Sf(t,\cdot)\big\|_{L^2_{r\,dr}}\\& +  \sup_{0<t\leq t_0}\tau_1^{N_2+2}\cdot\lambda_1^{-1}(t)\big\|S^2f(t,\cdot)\big\|_{L^2_{r\,dr}}
\end{split}\end{equation}

Consider now the inhomogeneous linear problem 
\begin{equation}\label{eq:inhomhighfreq}
-\partial_{tt}\epsilon + \partial_{rr}\epsilon + \frac{1}{r}\partial_r\epsilon - \frac{4\cos\big(2Q_1\big)}{r^2}\epsilon = f. 
\end{equation}
Then we have 
\begin{lem}\label{lem:highfreqinhom} Let $N\gg 1$ be large enough. Assuming $\big\|f\big\|_{Y}<\infty$, the problem \eqref{eq:inhomhighfreq} admits a solution satisfying the bound 
\begin{align*}
\big\|\epsilon\big\|_{X}\lesssim N^{-1}\big\|f\big\|_{Y}. 
\end{align*}
\end{lem}
\begin{proof} Letting as usual $R = \lambda_1(t)r$ and $\tau_1$ as before, we change to the variable 
\[
\tilde{\epsilon} = R^{\frac12}\cdot \epsilon.
\]
Setting further
\[
\omega: = \frac{\lambda_{\tau_1}}{\lambda},\,\dot{\omega}: = \partial_{\tau_1}\omega, 
\]
we derive the following analogue of \eqref{eq:generalinhomlin}:
\begin{equation}\label{eq:inhomgeneraltildeepsilon}\begin{split}
&\Big[-\big(\partial_{\tau_1} + \omega R\partial_R\big)^2 - \omega\big(\partial_{\tau_1} + \omega R\partial_R\big) +  \frac12\dot{\omega} + \frac14 \omega^2 - \tilde{\mathcal{L}}\Big]\tilde{\epsilon} = \tilde{f},\\
& \tilde{f} : = \lambda_1^{-2}\cdot R^{\frac12}f. 
\end{split}\end{equation}
We observe right away that 
\begin{align*}
\big\|\tilde{f}\big\|_{L^2_{dR}} = \lambda_1^{-1}\cdot \big\|f\big\|_{L^2_{r\,dr}}. 
\end{align*}

We shall derive the desired standard energy bounds first for $\epsilon$ and by differentiating the equation for $S\epsilon, S^2\epsilon$. Then we derive the more delicate weighted bounds for $\epsilon$ included at the end in the $\big\|\cdot\big\|_{X}$-norm. 
\\

In the following we shall refer freely to the results from Proposition~\ref{prop:Fourierbasis}, Proposition~\ref{prop:WeylTitchmarsh}, as well as Proposition~\ref{prop:Kcctransference}, Proposition~\ref{prop:transferencemappingbounds}. 
Denoting the distorted Fourier transform, which is a vector valued function, by 
\begin{equation}\label{eq:distortedFourier}
\mathcal{F}\tilde{\epsilon} = \left(\begin{array}{c}\langle \tilde{\epsilon}, \phi_0(R)\rangle_{L^2_{dR}}\\ \langle \tilde{\epsilon}, \phi(R, \xi)\rangle_{L^2_{dR}}\end{array}\right), 
\end{equation}
where the second component on the right is a function in $L^2(\mathbb{R}_+, \rho)$, we have the following equation for the Fourier transform:
\begin{equation}\label{eq:distFouriereqnKST}\begin{split}
\Big[{-}D_{\tau_1}^2 - \omega D_{\tau_1} - \xi\Big]\mathcal{F}\tilde{\epsilon} &= \mathcal{F}\tilde{f} - 2\omega\mathcal{K}_{nd}D_{\tau_1}\mathcal{F}\tilde{\epsilon} + \omega^2\big[\mathcal{K}_{nd}, \mathcal{K}_d\big]\mathcal{F}\tilde{\epsilon} \\
& + \omega^2\cdot\big(\mathcal{K}_{nd}^2 - \mathcal{K}_{nd}\big)\mathcal{F}\tilde{\epsilon} - \dot{\omega}\mathcal{K}_{nd}\mathcal{F}\tilde{\epsilon}.
\end{split}\end{equation} 
The operator $D_{\tau_1}$ here is defined by 
\[
D_{\tau_1} = \partial_{\tau_1} - \omega(1+\mathcal{K}_d). 
\]
Let is now write 
\begin{equation}\label{eq:tildeepsilonFouriercomponents}
\mathcal{F}\tilde{\epsilon}  =:\underline{x}= \left(\begin{array}{c}x_0\\ x_1\end{array}\right),\,x_0 = x_0(\tau_1),\,x_1 = x_1(\tau_1,\xi). 
\end{equation}
Let us introduce the energy type norm 
\begin{equation}\label{eq:normforxcomponents}\begin{split}
\big\|\underline{x}\big\|_{\tilde{X}}:& =  \sup_{0<t\leq t_0}\tau_1^{N_0}\big\|\rho^{\frac12}\cdot x_1(\tau_1,\cdot)\big\|_{L^2_{d\xi}} +   \sup_{0<t\leq t_0}\tau_1^{N_0+1}\big\|\rho^{\frac12}\cdot D_{\tau_1}x_1(\tau_1,\cdot)\big\|_{L^2_{d\xi}}\\
& +  \sup_{0<t\leq t_0}\tau_1^{N_0+1}\big\|\rho^{\frac12}\cdot \xi^{\frac12}x_1(\tau_1,\cdot)\big\|_{L^2_{d\xi}} + \sup_{0<t\leq t_0}\tau_1^{N_0}\cdot \big(\big|x_0(\tau_1)\big| + \tau_1\big|\dot{x}_{0}(\tau_1)\big|\big).
\end{split}\end{equation}
On the other hand, for the source terms we use the simpler norm
\begin{equation}\label{eq:sourceforfnorm}
\big\|\underline{f}\big\|_{\tilde{Y}}: =  \sup_{0<t\leq t_0}\tau_1^{N_0+2}\big\|\rho^{\frac12}\cdot f_1(\tau_1,\cdot)\big\|_{L^2_{d\xi}} +  \sup_{0<t\leq t_0}\tau_1^{N_0+2}\big|f_0(\tau_1)\big|. 
\end{equation}
We observe that application of the Plancherel's theorem for the distorted Fourier transform implies 
\[
\sup_{0<t\leq t_0}\tau_1^{N_0}\cdot \big\|\epsilon\big\|_{H^1_{r\,dr}}\lesssim \big\|\mathcal{F}\tilde{\epsilon} \big\|_{\tilde{X}}. 
\]
{\bf{Step 1}}: {\it{For large enough $N$, \eqref{eq:distFouriereqnKST} admits a solution $\mathcal{F}\tilde{\epsilon} = \underline{x}$ satisfying the bound $\big\|\underline{x}\big\|_{\tilde{X}}\lesssim N_0^{-1}\cdot \big\|\mathcal{F}\tilde{f}\big\|_{\tilde{Y}}$.}} To see this, we argue in close analogy to the proof of Proposition 6.2 in \cite{KST3}. Denoting the right hand side of \eqref{eq:distFouriereqnKST} as $\underline{g} = \left(\begin{array}{c}g_0\\ g_1\end{array}\right)$, the equations for $x_j, j = 0, 1$ become 
\begin{equation}\label{eq:xzeroonedecoupled}\begin{split}
&{-}\partial_{\tau_1}\big(\partial_{\tau_1} - \omega\big)x_0 = g_0\\
&\big({-}D_{\tau_1}^2 - \omega D_{\tau_1} - \xi\big)x_1 = g_1. 
\end{split}\end{equation}
Then we choose the solution $x_0$ of the first equation which is given by the Duhamel formula 
\begin{equation}\label{eq:xzeropropagator}
x_0(\tau_1) = \int_{\tau_1}^\infty U_0\big(\tau_1,\sigma_1\big)\cdot g_0(\sigma_1)\,d\sigma_1,\,U_0\big(\tau_1,\sigma_1\big) = \lambda_1(\tau_1)\cdot\int_{\tau_1}^{\sigma_1}\lambda_1^{-1}(s)\,ds. 
\end{equation}
Then taking advantage of \eqref{eq:xzeroDuhamelpropagator} we infer the bound
\begin{align*}
&\sup_{0<t\leq t_0}\tau_1^{N_0}\big|x_0(\tau_1)\big| + \sup_{0<t\leq t_0}\tau_1^{N_0+1}\big|\dot{x}_0(\tau_1)\big|\\
& \lesssim \sup_{0<t\leq t_0}\tau_1^{N_0}\cdot \int_{\tau_1}^\infty \Big[\big|U_0\big(\tau_1,\sigma_1\big)\big| + \tau_1\big|\partial_{\tau_1}U_0\big(\tau_1,\sigma_1\big)\big|\Big] \cdot \big|g_0(\sigma_1)\big|\,d\sigma_1\\
&\lesssim  \sup_{0<t\leq t_0}\tau_1^{N_0+1}\cdot \int_{\tau_1}^\infty \log\big(\frac{\sigma_1}{\tau_1}\big)\cdot \big|g_0(\sigma_1)\big|\,d\sigma_1\\
\end{align*}
The last expression is easily seen to be bounded by 
\begin{equation}\label{eq:gzeropropagatorbound}
 \sup_{0<t\leq t_0}\tau_1^{N_0+1}\cdot \int_{\tau_1}^\infty \log\big(\frac{\sigma_1}{\tau_1}\big)\cdot \big|g_0(\sigma_1)\big|\,d\sigma_1\lesssim  \frac{1}{N_0}\sup_{0<t\leq t_0}\tau_1^{N_0+2}\cdot\big|g_0(\tau_1)\big|.
\end{equation}
As for the second equation in \eqref{eq:xzeroonedecoupled}, we can express its solution via the Duhamel propagator 
\begin{equation}\label{eq:Duhamellambdaonetauone}
U(\tau_1, \sigma_1, \xi): = \frac{\rho^{\frac12}\big(\xi\frac{\lambda_1^2(\tau_1)}{\lambda_1^2(\sigma_1)}}{\rho^{\frac12}(\xi)}\big)\frac{\lambda_1(\tau_1)}{\lambda_1(\sigma_1)}\frac{\sin\big(\xi^{\frac12}\lambda_1(\tau_1)\cdot\int_{\tau_1}^{\sigma_1}\lambda_1^{-1}(s)\,ds\big)}{\xi^{\frac12}},
\end{equation}
which already appears in the proof of Lemma 6.4 in \cite{KST3}. We observe the following estimate, which results from the bounds \eqref{eq:xzeroDuhamelpropagator}:
\begin{equation}\label{eq:Upointwisebounds}\begin{split}
&\big|U(\tau_1, \sigma_1, \xi)\big|\lesssim \frac{\rho^{\frac12}\big(\xi\frac{\lambda_1^2(\tau_1)}{\lambda_1^2(\sigma_1)}}{\rho^{\frac12}(\xi)}\cdot \frac{\lambda_1(\tau_1)}{\lambda_1(\sigma_1)}\cdot \min\{\xi^{-\frac12},\,\log\big(\frac{\sigma_1}{\tau_1}\big)\cdot\tau_1\},\\
&\big|\partial_{\tau_1}U(\tau_1, \sigma_1, \xi)\big|\lesssim  \frac{\rho^{\frac12}\big(\xi\frac{\lambda_1^2(\tau_1)}{\lambda_1^2(\sigma_1)}}{\rho^{\frac12}(\xi)}\cdot \log\big(\frac{\sigma_1}{\tau_1}\big).
\end{split}\end{equation} 

Then a solution of the second equation of \eqref{eq:xzeroonedecoupled} is given by the Duhamel type formula
\begin{equation}\label{eq:xonepropagator}
x_1\big(\tau_1,\xi\big) = \int_{\tau_1}^{\infty}U(\tau_1, \sigma_1, \xi)\cdot g_1\big(\sigma_1, \frac{\lambda^2(\tau_1)}{\lambda^2(\sigma_1)}\xi\big)\,d\sigma_1. 
\end{equation}
The following bound follows from the bounds for the propagator $U$:
\begin{equation}\label{eq:gonepropagatorbound}\begin{split}
&\sup_{0<t\leq t_0}\tau_1^{N_0}\big\|\rho^{\frac12}\cdot x_1(\tau_1,\cdot)\big\|_{L^2_{d\xi}} +   \sup_{0<t\leq t_0}\tau_1^{N_0+1}\big\|\rho^{\frac12}\cdot D_{\tau_1}x_1(\tau_1,\cdot)\big\|_{L^2_{d\xi}}\\
&\hspace{6cm} +  \sup_{0<t\leq t_0}\tau_1^{N_0+1}\big\|\rho^{\frac12}\cdot \xi^{\frac12}x_1(\tau_1,\cdot)\big\|_{L^2_{d\xi}}\\
&\lesssim N_0^{-1}\cdot \sup_{0<t\leq t_0}\tau_1^{N_0+2}\big\|\rho^{\frac12}\cdot g_1(\tau_1,\cdot)\big\|_{L^2_{d\xi}}.
\end{split}\end{equation}
In order to solve \eqref{eq:distFouriereqnKST}, we implement a fixed point argument for $\underline{x} =  \mathcal{F}\tilde{\epsilon} $. Using Proposition~\ref{prop:transferencemappingbounds}, \eqref{eq:lambdaonederivative}, we deduce the following estimate:
\begin{equation}\label{eq:linearsourcetermbounds}\begin{split}
&\Big\|2\omega\mathcal{K}_{nd}D_{\tau_1}\underline{x}\Big\|_{\tilde{Y}} + \Big\| \omega^2\big[\mathcal{K}_{nd}, \mathcal{K}_d\big]\underline{x}\Big\|_{\tilde{Y}} + \Big\| \omega^2\cdot\big(\mathcal{K}_{nd}^2 - \mathcal{K}_{nd}\big)\underline{x}\Big\|_{\tilde{Y}}\\
& + \Big\| \dot{\omega}\mathcal{K}_{nd}\underline{x}\Big\|_{\tilde{Y}} \lesssim \big\|\underline{x}\big\|_{\tilde{X}}. 
\end{split}\end{equation}
Choosing $N_0$ large enough, the existence of a solution $\mathcal{F}\tilde{\epsilon} = \underline{x}$  of \eqref{eq:distFouriereqnKST} is then a consequence of the estimates \eqref{eq:linearsourcetermbounds}, \eqref{eq:gonepropagatorbound}, \eqref{eq:gzeropropagatorbound} and an application of the Banach fixed point theorem. 
\\

{\bf{Step 2}}: {\it{Control over the second and third norm component in \eqref{eq:Xnorm}. }} This is accomplished by applying $S^j, j = 1, 2$ to the equation \eqref{eq:inhomhighfreq} and taking advantage of the bounds established in the previous step. To begin with, we observe that 
\begin{equation}\label{eq:Sepsilon}
-\partial_{tt}S\epsilon + \partial_{rr}S\epsilon + \frac{1}{r}\partial_rS\epsilon - \frac{4\cos\big(2Q_1\big)}{r^2}S\epsilon = Sf + 2f+ \mathcal{R}_1\epsilon,
\end{equation}
where we set 
\[
\mathcal{R}_1\epsilon =  \frac{4S\big(\cos\big(2Q_1\big)\big)}{r^2}\epsilon.
\]
Then we observe the estimate 
\begin{align*}
\big| \frac{4S\big(\cos\big(2Q_1\big)\big)}{r^2}\big|\lesssim \tau_1^{0+}\cdot \lambda_1^2. 
\end{align*}
We conclude that 
\begin{align*}
\lambda_1^{-1}\cdot \big\| \frac{4S\big(\cos\big(2Q_1\big)\big)}{r^2}\epsilon\big\|_{L^2_{r\,dr}}&\lesssim \lambda_1^{-2}\cdot \big\| \frac{4S\big(\cos\big(2Q_1\big)\big)}{r^2}\big\|_{L^\infty_{r\,dr}}\cdot \lambda_1\big\|\epsilon\big\|_{L^2_{r\,dr}}\\
&\lesssim \tau_1^{0+}\cdot  \lambda_1\big\|\epsilon\big\|_{L^2_{r\,dr}}. 
\end{align*}
It follows that 
\begin{equation}\label{eq:Ronebound}
\sup_{0<t\leq t_0}\tau_1^{N_1+2}\cdot \lambda_1^{-1}\big\|\mathcal{R}_1\epsilon(t,\cdot)\big\|_{L^2_{r\,dr}}\ll_{t_0}\big\|\epsilon\big\|_{X}, 
\end{equation}
provided $N_0>N_1+2$, which is the case provided $N$ is sufficiently large. 
Applying the distorted Fourier transform and applying the estimates from the previous step, we arrive at the bound 
\begin{align*}
\big\|\mathcal{F}\big(S\epsilon\big)\big\|_{\tilde{X}_1}\lesssim \big\|\mathcal{F}\big(\widetilde{Sf}\big)\big\|_{\tilde{Y}_1} + \big\|\mathcal{F}\tilde{f}\big\|_{\tilde{Y}}.
\end{align*}
Here we define $\big\|\cdot\big\|_{\tilde{X}_j}, \big\|\cdot\big\|_{\tilde{Y}_j}$, $j = 1, 2$, like the norms $\big\|\cdot\big\|_{\tilde{X}}, \big\|\cdot\big\|_{\tilde{Y}}$, except we replace $N_0$ by $N_j$. 
\\

Deducing the corresponding bound for $S^2\epsilon$ is handled analogously, by first deriving the equation 
\begin{equation}\label{eq:S2epsilon}\begin{split}
-\partial_{tt}S^2\epsilon + \partial_{rr}S^2\epsilon + \frac{1}{r}\partial_rS^2\epsilon - \frac{4\cos\big(2Q_1\big)}{r^2}S^2\epsilon &= S\big(Sf + 2f+ \mathcal{R}_1\epsilon\big)\\
& + 2\big(Sf + 2f+ \mathcal{R}_1\epsilon\big) + \mathcal{R}_1S\epsilon.
\end{split}\end{equation}
Assuming $N_1>N_2+2$, which is ensured if $N$ is sufficiently large, we deduce that 
\begin{align*}
\big\|\mathcal{F}\big(S^2\epsilon\big)\big\|_{\tilde{X}_2}\lesssim \big\|\mathcal{F}\big(\widetilde{S^2f}\big)\big\|_{\tilde{Y}_2} + \big\|\mathcal{F}\big(\widetilde{Sf}\big)\big\|_{\tilde{Y}_1} + \big\|\mathcal{F}\tilde{f}\big\|_{\tilde{Y}}.
\end{align*}

{\bf{Step 3}}: {\it{Control over the fourth and fifth norm components in \eqref{eq:Xnorm}. }} Here we use the preceding estimates together with \eqref{eq:inhomhighfreq} to establish the remaining bounds, following closely the procedure in \cite{KST3}. To begin with, note the identity (from \cite{KST3})
\begin{equation}\label{eq:trickidentity}\begin{split}
\Big(\frac{t^2 - r^2}{t^2}\partial_{rr} + \frac{1}{r}\partial_r - \frac{4}{r^2}\Big)\epsilon &= t^{-2}\cdot\Big({-}S^2\epsilon + 2t\partial_t S\epsilon + S\epsilon - 2t\partial_t\epsilon\Big)\\
& + \Box\epsilon - \frac{4\epsilon}{r^2},
\end{split}\end{equation}
where we recall our convention $\Box = -\partial_{tt} + \partial_{rr} + \frac{1}{r}\partial_r$. It follows that 
\begin{align*}
\lambda_1^{-1}\cdot \Big\|\Big(\frac{t^2 - r^2}{t^2}\partial_{rr} + \frac{1}{r}\partial_r - \frac{4}{r^2}\Big)\epsilon\Big\|_{L^2_{r\,dr}}&\lesssim \big(\lambda_1\cdot t\big)^{-2}\cdot\big[\lambda_1\big\|S^2\epsilon\big\|_{L^2_{r\,dr}} + \lambda_1\big\|S\epsilon\big\|_{L^2_{r\,dr}}\big]\\
& +  \big(\lambda_1\cdot t\big)^{-1}\cdot\big[\big\|\partial_t S\epsilon\big\|_{L^2_{r\,dr}} + \big\|\partial_t\epsilon\big\|_{L^2_{r\,dr}}\big]\\
& + \lambda_1^{-1}\cdot\big\| \Box\epsilon - \frac{4\epsilon}{r^2}\big\|_{L^2_{r\,dr}}. 
\end{align*}
Here the first two terms can be estimated in terms of the already established norm bounds for $S^j\epsilon$, $j = 0, 1, 2$. For the last term on the right, taking advantage of \eqref{eq:inhomhighfreq} we can write this as 
\begin{align*}
 \lambda_1^{-1}\cdot\big\| \Box\epsilon - \frac{4\epsilon}{r^2}\big\|_{L^2_{r\,dr}} &=  \lambda_1^{-1}\cdot\big\|\frac{4\big[\cos\big(2Q_1\big)-1\big]}{r^2}\epsilon\big\|_{L^2_{r\,dr}} +  \lambda_1^{-1}\cdot\big\|f\big\|_{L^2_{r\,dr}}\\
 & \lesssim \lambda_1\cdot\big\|\epsilon\big\|_{L^2_{r\,dr}} + \lambda_1^{-1}\cdot\big\|f\big\|_{L^2_{r\,dr}}. 
\end{align*}
We conclude that 
\begin{align*}
\sup_{0<t\leq t_0}\tau^{N_2}\cdot \lambda_1^{-1}\cdot \Big\|\Big(\frac{t^2 - r^2}{t^2}\partial_{rr} + \frac{1}{r}\partial_r - \frac{4}{r^2}\Big)\epsilon\Big\|_{L^2_{r\,dr}}\lesssim \big\|f\big\|_{Y},
\end{align*}
taking into account steps 1 and 2. 
\\
Finally invoking the estimate (2.15) from \cite{KST3} which is 
\begin{equation}\label{eq:rminustwoestimate}
\big\|r^{-2}\epsilon\big\|_{L^2_{r\,dr}}\lesssim\big\|t^{-1}\nabla \epsilon\big\|_{L^2_{r\,dr}} + \big\|t^{-1}r^{-1}\epsilon\big\|_{L^2_{r\,dr}} + \Big\|\Big(\frac{t^2 - r^2}{t^2}\partial_{rr} + \frac{1}{r}\partial_r - \frac{4}{r^2}\Big)\epsilon\Big\|_{L^2_{r\,dr}},
\end{equation}
the estimate 
\begin{equation}\label{eq:epsilonoverrtwobound}
\sup_{0<t\leq t_0}\tau^{N_2}\cdot \lambda_1^{-1}\big\|r^{-2}\epsilon\big\|_{L^2_{r\,dr}}\lesssim\big\|f\big\|_{Y}
\end{equation}
follows from the preceding estimates. For this we also need to take advantage of the elementary bound 
\begin{equation}\label{eq:basicenergybound}
\big\|\nabla\epsilon\big\|_{L^2_{r\,dr}} + \big\|r^{-1}\epsilon\big\|_{L^2_{r\,dr}}\lesssim \big\|L_t^{\frac12}\epsilon\big\|_{L^2_{r\,dr}} + \lambda_1\cdot \big\|\epsilon\big\|_{L^2_{r\,dr}},
\end{equation}
which corresponds to (2.14) in \cite{KST3}. 
\\

In order to complete the proof of Lemma~\ref{lem:highfreqinhom}, it only remains to establish the estimate  
\begin{align*}
 \sup_{0<t\leq t_0}\tau_1^{N_2}\cdot\lambda_1^{-\frac12}(t)\big\|\frac{S\epsilon}{r}\big\|_{L^4_{r\,dr}}\lesssim \big\|f\big\|_{Y}. 
\end{align*}
As in \cite{KST3}, this bound follows by interpolating between the preceding estimate for $r^{-2}\epsilon$ and the following one:
\begin{align*}
\big\|S^2\epsilon\big\|_{L^\infty_{r\,dr}}\lesssim \big\|\nabla S^2\epsilon\big\|_{L^2_{r\,dr}} + \big\|\frac{S^2\epsilon}{r}\big\|_{L^2_{r\,dr}}, 
\end{align*}
which together with the last bound but one yields
\begin{align*}
\sup_{0<t\leq t_0}\tau_1^{N_2}\cdot \big\|S^2\epsilon\big\|_{L^\infty_{r\,dr}}\lesssim \sup_{0<t\leq t_0}\tau_1^{N_2}\cdot\big\|S^2\epsilon\big\|_{H^1_{r\,dr}}\lesssim \big\|f\big\|_{Y}. 
\end{align*}
\end{proof}

 \section{Construction of exact solution}
 
 It remains to find a final correction $\epsilon(t,\cdot)\in H^1_{r\,dr}$ with the property that 
 \[
 u: = u_N + \epsilon
 \]
 is an exact solution for \eqref{eq:2corotational} in the cone $r\lesssim t$, where $u_N$ is as in Proposition~\ref{prop:approxsoln}. The equation for $\epsilon$ is then the following 
 \begin{equation}\label{eq:epsilonfinaleqn}
 -\epsilon_{tt} + \epsilon_{rr} + \frac{1}{r}\epsilon_r - \frac{4\cos\big(2Q_1 - 2\tilde{Q}_2\big)}{r^2}\epsilon = e_N + \sum_{k=1}^3 E_k, 
 \end{equation}
where the terms $E_k = E_k(\epsilon)$ are given by the following expressions:
\begin{equation}\label{eq:finalEkterms}\begin{split}
&E_1 = \frac{2\cos\big(2Q_1 - 2\tilde{Q}_2 + v_N\big)}{r^2}\cdot\big[\sin\big(2\epsilon\big) - 2\epsilon\big],\\
& E_2 = 2\epsilon\cdot\big[\frac{2\cos\big(2Q_1 - 2\tilde{Q}_2 + 2v_N\big)}{r^2} - \frac{2\cos\big(2Q_1 - 2\tilde{Q}_2 \big)}{r^2}\big]\\
& E_3 =   \frac{2\sin\big(2Q_1 - 2\tilde{Q}_2 + v_N\big)}{r^2}\cdot\big[\cos\big(2\epsilon\big) - 1\big]. 
\end{split}\end{equation}
Furthermore, we define 
\begin{align*}
e_N: =\chi_{r\lesssim t}\cdot\Big( {-}u_{N, tt} + u_{N, rr} + \frac{1}{r}u_{N,r} - 2\frac{\sin (2u_N)}{r^2}\Big).
\end{align*}
Then we can assert the following 
\begin{prop}\label{prop:finalepsiloncorrection} There exist $N\gg1$ sufficiently large as well as $t_0 = t_0(\beta, N)>0$ sufficiently small so that the equation \eqref{eq:epsilonfinaleqn} admits a solution 
\[
\epsilon\in C^0\big((0, t_0]; H^1_{r\,dr}\big)
\]
satisfying the bound
\[
\sum_{k=0}^2\big\|S^k\epsilon(t,\cdot)\big\|_{H^1_{r\,dr}}\ll_{N,t_0} \tau^{-N+},\,S = t\partial_t + r\partial_r. 
\]
\end{prop}
\begin{proof} The main idea is to take advantage of the fast decay of the source term by simplifying the linear operator. Specifically, we recast the equation \eqref{eq:epsilonfinaleqn} in the form 
 \begin{equation}\label{eq:epsilonfinaleqn1}
 -\epsilon_{tt} + \epsilon_{rr} + \frac{1}{r}\epsilon_r - \frac{4\cos\big(2Q_1\big)}{r^2}\epsilon =    \big(\frac{4\cos\big(2Q_1 - 2\tilde{Q}_2\big) - 4\cos\big(2Q_1\big)}{r^2}\big)\epsilon + e_N + \sum_{k=1}^3 E_k.
 \end{equation}
For the additional term on the right we take advantage of the following estimate, where we recall $R = \lambda_1(t)\cdot r$:
\begin{align*}
\big| \frac{4\cos\big(2Q_1 - 2\tilde{Q}_2\big) - 4\cos\big(2Q_1\big)}{R^2}\big|\lesssim \frac{\lambda_2^2}{\lambda_1^2}\sim \tau^{-2}. 
\end{align*}
Using that $\big\|\cdot\big\|_{L^2_{r\,dr}} = \lambda_1^{-1}\big\|\cdot\big\|_{L^2_{R\,dR}}$, It follows that 
\begin{equation}\label{eq:extrasourcerermestimate1}\begin{split}
&\lambda_1^{-1}\cdot \Big\|\big(\frac{4\cos\big(2Q_1 - 2\tilde{Q}_2\big) - 4\cos\big(2Q_1\big)}{r^2}\epsilon\Big\|_{L^2_{r\,dr}}
\lesssim \tau^{-2}\cdot \big\|\epsilon\big\|_{L^2_{R\,dR}}. 
\end{split}\end{equation}
More generally, we infer the estimate 
\begin{equation}\label{eq:extrasourcerermestimate2}\begin{split}
&\lambda_1^{-1}\cdot \Big\|S^l\Big(\frac{4\cos\big(2Q_1 - 2\tilde{Q}_2\big) - 4\cos\big(2Q_1\big)}{r^2}\epsilon\Big)\Big\|_{L^2_{r\,dr}}\\&\lesssim \tau^{-2}\cdot \big(\sum_{k=0}^l\big\|S^k\epsilon\big\|_{L^2_{R\,dR}}\big). 
\end{split}\end{equation}
Concerning the nonlinear terms $E_k$, $k = 1,2,3$, we have the estimates 
\begin{equation}\label{eq:Eonefinal1}
\big\|\lambda_1^{-1}\cdot S^l E_1\big\|_{L^2_{r\,dr}}\lesssim \lambda_1^{-1}\cdot \big\|\frac{\epsilon}{r^2}\big\|_{L^2_{r\,dr}}\cdot \big(\sum_{k=0}^2\big\|S^{k}\epsilon\big\|_{L^\infty}^2\big),\,l\leq 2. 
\end{equation}
Next, we write the term $E_2$ as 
\begin{align*}
E_2 =  2\epsilon\cdot \big[\cos\big(2Q_1 - 2\tilde{Q}_2\big)\cdot \frac{\cos\big(2v_N\big) - 1}{r^2} - \sin\big(2Q_1 - 2\tilde{Q}_2\big)\cdot \frac{\sin\big(2v_N\big)}{r^2}\big]
\end{align*}
We then take advantage of the weighted estimate 
\[
\big|\langle R\rangle^{-1}v_N\big|\lesssim \tau^{-2},
\]
which yields
\begin{align*}
\big|\sin\big(2Q_1 - 2\tilde{Q}_2\big)\cdot \frac{\sin\big(2v_N\big)}{r^2}\big|&\lesssim r^{-2}\cdot \langle R^2\rangle^{-1}\cdot \big|v_N\big| + \frac{\big|\tilde{Q}_2\big|}{r^2}\cdot \big|v_N\big|\\
&\lesssim  \langle R\rangle^{-1}\cdot \tau^{-2}\cdot r^{-2} + \lambda_2^2\cdot \tau^{-2+}. 
\end{align*}
It follows that 
\begin{align*}
\lambda_1^{-1}\big\|\epsilon\cdot \sin\big(2Q_1 - 2\tilde{Q}_2\big)\cdot \frac{\sin\big(2v_N\big)}{r^2}\big\|_{L^2_{r\,dr}}&=\lambda_1^{-2}\big\|\epsilon\cdot \sin\big(2Q_1 - 2\tilde{Q}_2\big)\cdot \frac{\sin\big(2v_N\big)}{r^2}\big\|_{L^2_{R\,dR}}\\
&\lesssim \tau^{-2}\cdot \big\|\epsilon\big\|_{L^\infty_{r\,dr}}. 
\end{align*}
Furthermore, since 
\begin{align*}
\big\|\frac{\cos\big(2v_N\big) - 1}{R^2}\big\|_{L^2_{R\,dR}}\lesssim \tau^{-4+}, 
\end{align*}
we deduce the estimate 
\begin{align*}
\lambda_1^{-1}\cdot \big\|E_2\big\|_{L^2_{r\,dr}}\lesssim \tau^{-2}\cdot \big\|\epsilon\big\|_{L^\infty_{r\,dr}}. 
\end{align*}
Due to the estimates \eqref{eq:hjbounds}, we also deduce the differentiated bounds 
\begin{equation}\label{eq:Etwofinal1}
\lambda_1^{-1}\cdot \big\|S^lE_2\big\|_{L^2_{r\,dr}}\lesssim \tau^{-2}\cdot \big(\sum_{k=0}^l\big\|S^k\epsilon\big\|_{L^\infty_{r\,dr}}\big),\,l\leq 2. 
\end{equation}
Using the elementary bound 
\begin{equation}\label{eq:Linftyenergybound}
\big\|S^k\epsilon\big\|_{L^\infty_{r\,dr}}\lesssim \big\|\nabla S^k\epsilon\big\|_{L^2_{r\,dr}} + \big\|\frac{S^k\epsilon}{r}\big\|_{L^2_{r\,dr}}
\end{equation}
as well as \eqref{eq:basicenergybound}, we find in light of \eqref{eq:Honenormdefn} the estimate 
\begin{equation}\label{eq:Etwofinal2}
\lambda_1^{-1}\cdot \big\|S^lE_2(t,\cdot)\big\|_{L^2_{r\,dr}}\lesssim \tau^{-2}\cdot \big(\sum_{0\leq k\leq l}\big\|S^k\epsilon(t,\cdot)\big\|_{H^1_{r\,dr}}\big),\,l\leq 2. 
\end{equation}
Similarly, the bound \eqref{eq:Eonefinal1} leads to the following one 
\begin{equation}\label{eq:Eonefinal2}
\big\|\lambda_1^{-1}\cdot S^l E_1\big\|_{L^2_{r\,dr}}\lesssim \lambda_1^{-1}\cdot \big\|\frac{\epsilon}{r^2}\big\|_{L^2_{r\,dr}}\cdot\big(\sum_{0\leq k\leq l}\big\|S^k\epsilon(t,\cdot)\big\|_{H^1_{r\,dr}}\big),\,l\leq 2. 
\end{equation}

Finally, we turn to the term $E_3$, which is again nonlinear and similar to the term $E_1$. We have the estimate 
\begin{align*}
\lambda_1^{-1}\big\|S^l\big(\frac{\cos(2\epsilon) - 1}{r^2}\big)\big\|_{L^2_{r\,dr}}&\lesssim \lambda_1^{-1}\cdot\big\|\frac{\epsilon}{r^2}\big\|_{L^2_{r\,dr}}\cdot \big(\sum_{k=0}^2\big\|S^k\epsilon\big\|_{L^\infty_{r\,dr}}\big)\\
&+\lambda_1^{-1}\cdot \big\|\frac{S\epsilon}{r}\big\|_{L^4_{r\,dr}}^2,\,l\leq 2.  
\end{align*}
Again taking advantage of the estimates \eqref{eq:hjbounds}, we infer that 
\begin{align*}
\big\|S^l\big(2\cos\big(2Q_1 - 2\tilde{Q}_2 + v_N\big)\big\|_{L^\infty_{r\,dr}}\lesssim_l 1. 
\end{align*}
Using the Leibniz rule and Holder's inequality, we deduce that 
\begin{equation}\label{eq:Ethreefinal1}\begin{split}
\lambda_1^{-1}\cdot \big\|S^lE_3\big\|_{L^2_{r\,dr}}&\lesssim  \lambda_1^{-1}\cdot\big\|\frac{\epsilon}{r^2}\big\|_{L^2_{r\,dr}}\cdot \big(\sum_{k=0}^2\big\|S^k\epsilon\big\|_{L^\infty_{r\,dr}}\big)\\
&+\lambda_1^{-1}\cdot \big\|\frac{S\epsilon}{r}\big\|_{L^4_{r\,dr}}^2,\,l\leq 2.
\end{split}\end{equation}
Taking advantage of \eqref{eq:Linftyenergybound}, we can further estimate this by 
\begin{equation}\label{eq:Ethreefinal2}\begin{split}
\lambda_1^{-1}\cdot \big\|S^lE_3\big\|_{L^2_{r\,dr}}&\lesssim   \lambda_1^{-1}\cdot\big\|\frac{\epsilon}{r^2}\big\|_{L^2_{r\,dr}}\cdot\big(\sum_{k=0}^2\big\|S^k\epsilon\big\|_{H^1_{r\,dr}}\big)\\
&+\lambda_1^{-1}\cdot \big\|\frac{S\epsilon}{r}\big\|_{L^4_{r\,dr}}^2,\,l\leq 2.
\end{split}\end{equation}
Combining the estimates \eqref{eq:extrasourcerermestimate2}, \eqref{eq:Etwofinal2}, \eqref{eq:Eonefinal2} as well as \eqref{eq:Ethreefinal2}, recalling \eqref{eq:Xnorm}, \eqref{eq:Ynorm}, and choosing $N$ large enough, we infer the estimate 
\begin{equation}\label{eq:epsiloneqnsourcetermbound}
\Big\|g\Big\|_{Y}\lesssim \big\|\epsilon\big\|_{X} + o_{t_0}(1), 
\end{equation}
where 
\[
g = \big(\frac{4\cos\big(2Q_1 - 2\tilde{Q}_2\big) - 4\cos\big(2Q_1\big)}{r^2}\big)\epsilon + e_N + \sum_{k=1}^3 E_k
\]
is the right hand side in \eqref{eq:epsilonfinaleqn1}. The proof of Proposition~\ref{prop:finalepsiloncorrection} now follows by taking advantage of the preceding bound, Lemma~\ref{lem:highfreqinhom} and a standard Banach iteration, provided we choose $N$ large enough.
\end{proof}
 
 The proof of Theorem~\ref{thm:Main} is a direct consequence of the preceding proposition, together with \eqref{eq:alphaasympto}, as well as the Huyghen's principle. In fact, the function $u_N+\epsilon$ solves \eqref{eq:2corotational}
 in the light cone 
 \[
\mathcal{A}: =  \{r\leq t\}\cap \{t\in (0, t_0]\}. 
 \]
 Letting $u(t, r)$ be the solution of \eqref{eq:2corotational} with initial\footnote{We use the notation $v[t] = \big(v(t,\cdot), v_t(t, \cdot)\big)$.} data $\big(u_N+\epsilon\big)[t_0]$ at time $t = t_0$, it holds that 
 \[
 u\big|_{r\leq t} = \big(u_N+\epsilon\big)\big|_{r\leq t} 
 \]
 for $t\in (0, t_0]$, due to Huyghen's principle. Energy conservation implies that $u$ cannot develop a singularity outside the light cone $r\leq t$, and so the first singularity of $u$ on $[0, t_0]\times\mathbb{R}_+$ occurs indeed at time $t = 0$ as described by $u_N+\epsilon$, resulting in two bubbles concentrating simultaneously but at different rates at the origin.


\begin{thebibliography}{10}

\bibitem{CTZduke}
\newblock D.\ Christodoulou,\ A.\ S.\ Tahvildar-Zadeh
\newblock \emph{On the asymptotic behavior of spherically symmetric wave maps},
\newblock Duke Math. J. 71 (1993), no. 1, 31--69.

\bibitem{CTZcpam}
\newblock D.\ Christodoulou,\ A.\ S.\ Tahvildar-Zadeh
\newblock \emph{On the regularity of spherically symmetric wave maps},
\newblock Comm. Pure Appl. Math. 46 (1993), no. 7, 1041--1091.

\bibitem{Cote15}
\newblock R.\ C{\^o}te
\newblock \emph{On the soliton resolution for equivariant wave maps to the sphere},
\newblock Comm. Pure Appl. Math. 68 (2015), no. 11, 1946--2004.

\bibitem{CKLS15}
\newblock R.\ C{\^o}te,\ C.\ E.\ Kenig,\ A.\ Lawrie,\ W.\ Schlag
\newblock \emph{Characterization of large energy solutions of the equivariant wave
  map problem: {I}},
\newblock Amer. J. Math. 137 (2015), no. 1, 139--207.

\bibitem{CKLS2}
\newblock R.\ C{\^o}te,\ C.\ E.\ Kenig,\ A.\ Lawrie,\ W.\ Schlag
\newblock \emph{Characterization of large energy solutions of the equivariant wave
  map problem: {II}},
\newblock Amer. J. Math. 137 (2015), no. 1, 209--250.

\bibitem{DKMM}
\newblock T.\ Duyckaerts,\ C.\ Kenig,\ Y.\ Martel,\ F.\ Merle
\newblock \emph{ Soliton resolution for critical co-rotational wave maps and radial cubic wave equation},
\newblock Comm. Math. Phys. 391 (2022), no. 2, 779--871.

\bibitem{J1}
\newblock J.\ Jendrej
\newblock \emph{Construction of two-bubble solutions for some energy-critical wave equations},
\newblock \'{E}ditions de l'\'{E}cole Polytechnique, Palaiseau, 2017, Exp. No. XXI, 10 pp.
ISBN: 978-2-7302-1658-6

\bibitem{J2}
\newblock J.\ Jendrej
\newblock \emph{Construction of two-bubble solutions for energy-critical wave equations},
\newblock Amer. J. Math. 141 (2019), no. 1, 55--118.

\bibitem{JL1}
\newblock J.\ Jendrej,\ A.\ Lawrie
\newblock \emph{Continuous time soliton resolution for two-bubble equivariant wave maps},
\newblock Math. Res. Lett.   29 (2022), no. 6, 1745--1766.

\bibitem{JL2}
\newblock J.\ Jendrej,\ A.\ Lawrie
\newblock \emph{Uniqueness of two-bubble wave maps in high equivariance classes},
\newblock Comm. Pure Appl. Math.   76 (2023), no. 8, 1608--1656.

\bibitem{JL3}
\newblock J.\ Jendrej,\ A.\ Lawrie
\newblock \emph{Soliton resolution for energy-critical wave maps in the equivariant case},
\newblock J. Amer. Math. Soc. (2024) DOI: https://doi.org/10.1090/jams/1012.

\bibitem{JL4}
\newblock J.\ Jendrej,\ A.\ Lawrie
\newblock \emph{ Two-bubble dynamics for threshold solutions to the wave maps equation},
\newblock Invent. Math. 213 (2018), no. 3, 1249--1325.

\bibitem{JL5}
\newblock J.\ Jendrej,\ A.\ Lawrie
\newblock \emph{Continuous time soliton resolution for two-bubble equivariant wave
  maps},
\newblock Math. Res. Lett. 29 (2022), no. 6, 1745--1766.

\bibitem{JK}
\newblock H.\ Jia,\ C.\ Kenig
\newblock \emph{Asymptotic decomposition for semilinear wave and equivariant wave map
  equations},
\newblock Amer. J. Math. 139 (2017), no. 6, 1521--1603.

\bibitem{KlM}
\newblock S.\ Klainerman, M.\ Machedon
\newblock \emph{Smoothing estimates for null forms and applications},
\newblock Duke Math. J. 81 (1995), no. 1, 99--133. 

\bibitem{KS97}
\newblock S.~Klainerman,\ S.~Selberg
\newblock \emph{Remark on the optimal regularity for equations of wave maps type},
\newblock Comm. Partial Differential Equations 22 (1997), no. 5--6, 901--918.

\bibitem{KST1} 
\newblock J.\ Krieger,   W.\  Schlag, D.\ Tataru 
\newblock \emph{ Renormalization and blow up for charge one equivariant critical wave maps}, 
\newblock Inventiones Mathematicae, vol. 171, p. 543--615, 2008.

\bibitem{KST2} 
\newblock J.\ Krieger,   W.\  Schlag, D.\ Tataru  (MR2494455)
\newblock \emph{   Slow blow-up solutions for the $H^1(\R^3)$ critical focusing semilinear wave equation},  
\newblock Duke Math.\ J., no.~1, \textbf{ 147}  (2009),   1--53.

\bibitem{KST3} 
\newblock J.\ Krieger,   W.\  Schlag, D.\ Tataru 
\newblock \emph{ Renormalization and blow up for the critical Yang-Mills problem}, 
\newblock Advances In Mathematics, vol. 221, p. 1445--1521, 2009.

\bibitem{RaRod}
\newblock P.\ Raphael,\, I.\ Rodnianski
\newblock \emph{Stable blow up dynamics for the critical co-rotational wave maps and equivariant Yang-Mills problems},
\newblock Publ. Math. Inst. Hautes \'{E}tudes Sci. 115 (2012), 1--122. 

\bibitem{RoSt10}
\newblock I.\ Rodnianski,\ J.\ Sterbenz
\newblock \emph{On the formation of singularities in the critical {$O(3)$}
  {$\sigma$}-model},
\newblock Ann. of Math. 172 (2010), no. 1, 187--242.

\bibitem{Rodrig}
\newblock C.\ Rodriguez
\newblock \emph{Threshold dynamics for corotational wave maps}
\newblock Anal. PDE 14 (2021), no. 7, 2123--2161.

\bibitem{ShSt00}
\newblock J.\ Shatah,\ M.\ Struwe
\newblock {\em {Geometric Wave Equations}}, volume~2 of {\em Courant Lecture
  Notes in Mathematics},
\newblock AMS, 2000.

\bibitem{STZ92}
\newblock J.\ Shatah,\ A.\ S.\ Tahvildar-Zadeh
\newblock \emph{Regularity of harmonic maps from the {M}inkowski space into
  rotationally symmetric manifolds},
\newblock Comm. Pure Appl. Math. 45 (1992), no. 8, 947--971.

\bibitem{STZ94}
\newblock J.\ Shatah,\ A.\ S.\ Tahvildar-Zadeh
\newblock \emph{On the {C}auchy problem for equivariant wave maps},
\newblock Comm. Pure Appl. Math. 47 (1994), no. 5, 719--754.

\bibitem{Struwe03}
\newblock M.\ Struwe
\newblock \emph{Equivariant wave maps in two space dimensions},
\newblock Comm. Pure Appl. Math. 56 (2003), vol. 7, 815--823.

\bibitem{Tao01}
\newblock T.\ Tao
\newblock \emph{Global regularity of wave maps II: Small energy in two dimensions},
\newblock Comm. Math. Phys. 224 (2001). no. 2, :443--544.

\bibitem{Tataru01}
\newblock D.\ Tataru
\newblock \emph{On global existence and scattering for the wave maps equation},
\newblock Amer. J. Math. 123 (2001), no. 1, 37--77.

\bibitem{VanderHout}
\newblock R.\ Van der Hout
\newblock \emph{On the nonexistence of finite time bubble trees in symmetric harmonic map heat flows from the disk to the 2-sphere.}
\newblock J. Differential Equations   192 (2003), no. 1, 188--201.


 \end{thebibliography}
\end{document}